\documentclass[reqno,a4paper]{amsart}
\usepackage[utf8]{inputenc}
\usepackage{amsmath,amssymb}
\usepackage{mathrsfs}
\usepackage{yhmath}
\usepackage{booktabs,url}
\usepackage{color}
\usepackage{hyperref}

\allowdisplaybreaks

\usepackage[left=2.5cm, right=2.5cm, bottom=3cm]{geometry}

\usepackage[font=footnotesize,labelfont=bf]{caption}

\newcommand{\de}[0]{\mathrel{\mathop:}=}
\newcommand{\ed}[0]{=\mathrel{\mathop:}}
\newcommand{\ie}[0]{\mathrm{i}}
\newcommand{\dif}[1]{\mathrm{d}#1}
\newcommand{\C}[0]{\mathbb{C}}
\newcommand{\R}[0]{\mathbb{R}}

\newcommand{\N}[0]{\mathbb{N}}
\newcommand{\K}[0]{\mathbb{K}}
\newcommand{\Q}[0]{\mathbb{Q}}

\newcommand{\cL}[0]{\mathcal{L}}
\newcommand{\cS}[0]{\mathcal{S}}
\newcommand{\cSP}[0]{\mathcal{SP}}
\newcommand{\sdeg}[0]{\dif{}_{\mathcal{L}}}
\newcommand{\sq}[0]{\mathrm{q}_{\mathcal{L}}}
\newcommand{\dL}[1]{\mathrm{d}_{#1}}
\newcommand{\qL}[1]{\mathrm{q}_{#1}}

\usepackage{amsmath}

\newtheorem{theorem}{Theorem}

\newtheorem{corollary}{Corollary}
\newtheorem{lemma}{Lemma}

\newtheorem{conjecture}{Conjecture}

\numberwithin{equation}{section}

\theoremstyle{definition}
\newtheorem{remark}{Remark}

\begin{document}

\title[Conditional estimates for $L$-functions in the Selberg class]{Conditional estimates for $L$-functions in the Selberg class}
\author[N.~Paloj\"{a}rvi and A.~Simoni\v{c}]{Neea Paloj\"{a}rvi and Aleksander Simoni\v{c}}
\address{Department of Mathematics and Statistics, The University of Helsinki, Helsinki, Finland}
\email{neea.palojarvi@helsinki.fi}
\address{School of Science, The University of New South Wales (Canberra), ACT, Australia}
\email{a.simonic@student.adfa.edu.au}
\date{\today}

\begin{abstract}
Assuming the Generalized Riemann Hypothesis, we provide uniform upper bounds with explicit main terms for moduli of $\left(\cL'/\cL\right)(s)$ and $\log{\cL(s)}$ for $1/2+\delta\leq\sigma<1$, fixed $\delta\in(0,1/2)$ and for functions in the Selberg class except for the identity function. We also provide estimates under additional assumptions on the distribution of Dirichlet coefficients of $\cL(s)$ on prime numbers. Moreover, by assuming a polynomial Euler product representation for $\cL(s)$, we establish uniform bounds for $|3/4-\sigma|\leq 1/4-1/\log{\log{\left(\sq|t|^{\sdeg}\right)}}$, $|1-\sigma|\leq 1/\log{\log{\left(\sq|t|^{\sdeg}\right)}}$ and $\sigma=1$, and completely explicit estimates by assuming also the strong $\lambda$-conjecture.
\end{abstract}
	
\maketitle
\thispagestyle{empty}

\section{Introduction}
\label{sec:intro}

The Selberg class of functions $\cS$ was introduced in~\cite{SelbergOldAndNew} and consists of Dirichlet series
\begin{equation*}
\label{eq:DS}
\mathcal{L}(s) = \sum_{n=1}^{\infty} \frac{a(n)}{n^s},
\end{equation*}
where $s=\sigma+\ie t$ for real numbers $\sigma$ and $t$, satisfying the following axioms:
\begin{enumerate}
  \item \emph{Ramanujan hypothesis}. We have $a(n)\ll_{\varepsilon} n^{\varepsilon}$ for any $\varepsilon>0$.
  \item \emph{Analytic continuation}. There exists $k\in\N_{0}$ such that $(s-1)^k\mathcal{L}(s)$ is an entire function of finite order.
  \item \emph{Functional equation}. The function $\mathcal{L}(s)$ satisfies $\mathfrak{L}(s)=\omega\overline{\mathfrak{L}(1-\bar{s})}$, where
        \[
        \mathfrak{L}(s)=\mathcal{L}(s)Q^s\prod_{j=1}^{f}\Gamma\left(\lambda_{j}s+\mu_j\right)
        \]
        with $\left(Q,\lambda_j\right)\in\R_{+}^2$, and $\left(\mu_j,\omega\right)\in\C^2$ with $\Re\{\mu_j\}\geq 0$ and $|\omega|=1$.
  \item \emph{Euler product}. The function $\cL(s)$ has a product representation
        \[
        \mathcal{L}(s) = \prod_{p}\exp{\left(\sum_{k=1}^{\infty}\frac{b\left(p^k\right)}{p^{ks}}\right)},
        \]
        where the coefficients $b\left(p^k\right)$ satisfy $b\left(p^k\right)\ll p^{k\theta}$ for some $0\leq\theta<1/2$.
\end{enumerate}
It is well known that the data from the functional equation is not uniquely defined by $\cL\in\cS$. However, one can show that the degree $\sdeg$ and the conductor $\sq$,
\[
\sdeg = 2\sum_{j=1}^{f}\lambda_j, \quad \sq=(2\pi)^{\sdeg}Q^2\prod_{j=1}^{f}\lambda_{j}^{2\lambda_{j}},
\]
are indeed invariants for the class $\cS$. In the present paper we work only with these two invariants and for the sake of simplicity we introduce $\tau\de\sq|t|^{\sdeg}$. There exist other invariants and the reader is referred to~\cite{PerelliSurveyI} and the references thereof. Notable examples of functions in $\mathcal{S}$ include:
\begin{enumerate}
\item The Riemann zeta-function $\zeta(s)$. We have $\dL{\zeta}=1$ and $\qL{\zeta}=1$.
\item Dirichlet $L$-functions $L(s,\chi)$ for a primitive characters $\chi$ modulo $q\geq 2$. We have $\dL{L(s,\chi)}=1$ and $\qL{L(s,\chi)}=q$.
\item Dedekind zeta-functions $\zeta_{\K}(s)$ for a number fields $\K$ with degree $n_{\K}=[\K:\Q]$ and discriminant $\Delta_{\K}$. We have $\dL{\zeta_{\K}}=n_{\K}$ and $\qL{\zeta_{\K}}=\left|\Delta_{\K}\right|$.
\end{enumerate}
We know that the only element in $\cS$ with the degree zero is the constant function $\cL\equiv1$, and that $\sdeg\geq1$ otherwise, see~\cite[Theorem 6.1]{SteudingBook}. To simplify our considerations, we will assume that $\cL\not\equiv1$ throughout this paper.

The celebrated Generalized Riemann Hypothesis (GRH) claims that $\cL(s)\neq0$ for $\sigma>1/2$. The main purpose of the present paper is to provide various estimates for moduli of the logarithmic derivative $\left(\cL'/\cL\right)(s)$ and the logarithm $\log{\cL(s)}$ in different regions right of the critical line and under the assumption of GRH. We should emphasize that our focus is on the results that hold for $\tau$ and $|t|$ sufficiently large\footnote{This means that there exists an absolute $T>0$ such that the results are valid for all $\tau\geq T$ and $|t|\geq T$.}, although there is no serious obstacle in obtaining also estimates when $t=0$ by following the same approach. For the Riemann zeta-function, Littlewood~\cite{LittlewoodOnTheRiemann} proved that the Riemann Hypothesis (RH) implies
\begin{gather}
\frac{\zeta'}{\zeta}(s) \ll \left((\log t)^{2-2\sigma}+1\right)\min\left\{\dfrac{1}{|\sigma-1|},\log\log t\right\}, \label{eq:Littlewood} \\
\log{\zeta(s)} \ll \frac{(\log t)^{2-2\sigma}}{\log{\log{t}}}\min\left\{\dfrac{1}{|\sigma-1|},\log\log t\right\} + \log{\log{\log{t}}} \label{eq:Littlewood2}
\end{gather}
for $1/2+1/\log{\log{t}}\leq\sigma\leq3/2$ and $t$ large, see~\cite[Corollaries 13.14 and 13.16]{MontgomeryVaughan}. In particular,
\begin{equation}
\label{eq:LittlewoodSpecial}
\frac{\zeta'}{\zeta}(s) \ll \left(\log{t}\right)^{2-2\sigma}, \quad \log{\zeta(s)} \ll \frac{(\log t)^{2-2\sigma}}{\log{\log{t}}}
\end{equation}
for $1/2+\delta\leq\sigma\leq1-\delta$ and fixed $\delta\in(0,1/4)$. In fact, Littlewood stated only~\eqref{eq:LittlewoodSpecial}, but it is possible to extend his techniques (see also~\cite[Section 14.5]{Titchmarsh}) to the region right of the $1$-line, see~\cite[Section 14.33]{Titchmarsh}. However, the argument which is based on a slightly modified version of the Selberg moment formula~\cite[Section 13.2]{MontgomeryVaughan} is more transparent, allows generalisations\footnote{See also~\cite[Theorems 5.17 and 5.19]{IKANT}. Our $\tau$ has a similar r\^{o}le to their analytic conductor $\mathfrak{q}(f,s)$.} to other functions, and easily yields~\eqref{eq:Littlewood} and~\eqref{eq:Littlewood2}. Such an approach with an additional feature of estimating the sum over the non-trivial zeros with a help of bandlimited majorants was explored in~\cite{ChirreSimonicHagen}, and the present paper extends this to the Selberg class. Special care is taken to express all results uniformly in $\cL$.

The aforementioned bounds have many applications: in estimating the number of zeros of $\cL$ and the modulus of its argument~\cite{Carneiro, ChirreANote, Palojarvi,SimonicSonRH}; in number theory~\cite{BennettAP,EHP,SimonicCS,SimonicLfunctions}; in studying moments of $L$-functions~\cite{SoundMoments}; and in studying character sums~\cite{Cech}. It is somewhat surprising that their shape has never been improved, and efforts have been placed into obtaining exact constants for the main terms. Recently, such a result for~\eqref{eq:Littlewood} was given by Chirre and Gon\c{c}alves~\cite[Theorem 1]{chirreGoncalves}, and for~\eqref{eq:Littlewood2} by Carneiro and Chandee~\cite[Theorems 1 and 2]{CarneiroChandee}. Chirre~\cite{ChirreANote} obtained an estimate for $\log{|L(s,\pi)|}$, where $L(s,\pi)$ is an entire $L$-function satisfying special conditions, see~\cite[Section 1.1]{ChirreANote} for a definition. Aistleitner and Pa\'{n}kowski~\cite[Proposition~1.1]{AistleitnerPankowski2017} generalised\footnote{Their main objective was to provide various $\Omega$-results for $\cL(s)$. For recent advances in the theory of $\Omega$-theorems for $\cL'/\cL$ where $\cL\in\{\zeta(s),L(s,\chi)\}$, see~\cite{YangD2023,li2024omega}.}~\eqref{eq:LittlewoodSpecial} for $\log{\zeta(s)}$ to the Selberg class with an additional requirement~\eqref{eq:PrimeMean} that concerns distribution of the coefficients of a Dirichlet series over prime numbers. Concerning estimates on the 1-line, Littlewood~\cite[Theorem 7]{LittlewoodOnTheRiemann} provided a sharper estimate than~\eqref{eq:Littlewood2} by proving
\begin{equation}
\label{eq:Littlewood1LineUpper}
\left|\zeta(1+\ie t)\right| \leq \left(2e^{\gamma}+o(1)\right)\log{\log{t}}.
\end{equation}
Later he also provided in~\cite{Littlewood1928} a lower bound for $\zeta(1+\ie t)$, namely that
\begin{equation}
\label{eq:Littlewood1Line}
\frac{1}{\left|\zeta(1+\ie t)\right|} \leq \left(\frac{12e^{\gamma}}{\pi^2}+o(1)\right)\log{\log{t}}.
\end{equation}
The error terms in~\eqref{eq:Littlewood1LineUpper} and~\eqref{eq:Littlewood1Line} were improved in~\cite{CarneiroChandee},~\cite{Lamzouri} and~\cite{Lum18}. In particular, Lumley~\cite[Corollary 2.3]{Lum18} extended the work~\cite{Lamzouri} and provided fully explicit\footnote{By an explicit (or effective) result we mean that all ``hidden'' constants in a corresponding non-explicit result are computed or expressed with known parameters via some algebraic formulae.} bounds for $L(1+\ie t,f)$ for a large set\footnote{In the framework of~\cite[Chapter~5]{IKANT}. That class of $L$-functions is, under the Ramanujan--Petersson conjecture (which implies axiom (1) and $\Re\{\mu_j\}\geq 0$ from axiom (3)), almost the same as the Selberg class with a polynomial Euler product (Section~\ref{subsec:PolyEuler}).} of entire $L$-functions under the assumption of GRH and the Ramanujan--Petersson conjecture. Recently, a precise and explicit form of~\eqref{eq:Littlewood} on the $1$-line was given by Chirre, Hagen and the second author in~\cite[Theorem~5]{ChirreSimonicHagen} by proving that 
\begin{equation}
\label{eq:CHS}
\left|\frac{\zeta'}{\zeta}(1+\ie t)\right| \leq 2\log{\log{t}} - 0.4989 + 5.35\frac{(\log{\log{t}})^{2}}{\log{t}}, \quad t\geq 10^{30}.
\end{equation}
Interestingly, it shows that one can obtain for sufficiently large $t$ a simple upper bound $2\log{\log{t}}$. In addition to the above, explicit and conditional results on $\left(\zeta'/\zeta\right)(s)$ or $\log{\zeta(s)}$, in different regions right of the critical line and for other $L$-functions as well, appeared in~\cite{IharaMurtyShimura,LanguascoTrudgian,SimonicCS,SimonicLfunctions,Chirre2024}.

Our results are divided into two categories: non-explicit and explicit. In the non-explicit setting, we are not only  providing estimates in the range $1/2+\delta\leq\sigma<1$ and for the full Selberg class (Theorem~\ref{thm:MainGeneral}), but also improved bounds under additional assumptions such as the prime mean-square conjecture (Theorems~\ref{thm:MainGeneralV2} and~\ref{thm:MainGeneral1line}) or the existence of a polynomial Euler product (Theorems~\ref{thm:MainNonExplicit} and~\ref{thm:1line}). For the former we formulated Conjectures~\ref{conj:SelbergVariant} and~\ref{conj:SelbergVariant2} that enclose other well-known conjectures for the Selberg class, and for the latter we provide estimates according to whether $\left|3/4-\sigma\right|\leq1/4-1/\log{\log{\tau}}$ or $|1-\sigma|\log{\log{\tau}}\leq 1$, as well as a separate result when $\sigma=1$. In the explicit setting, we provide effective versions (Theorems~\ref{thm:MainExplicit} and~\ref{thm:1LineExplicit}) of Theorems~\ref{thm:MainNonExplicit} and~\ref{thm:1line}. Although there is no serious obstacle in the method, we are additionally assuming also the \emph{strong $\lambda$-conjecture}, i.e., \emph{all $\lambda_j$ from the functional equation can be chosen to be equal to $1/2$}, in order to simplify the proof. The main results in this direction are relatively complicated, but allow some optimisation on the parameters that might be useful in applications. Because of this we are stating simplified results (Corollaries~\ref{cor:ZetaExplicit} and~\ref{corol:FixedRegions}) for the Riemann zeta-function and for a large set of $L$-functions, respectively.

Comparing the estimates from Sections~\ref{sec:MainResultsNE} and~\ref{sec:Explicit} with known results reveals the following:
\begin{enumerate}
    \item We provide a precise form of~\cite[Proposition~1.1]{AistleitnerPankowski2017}, see Remark~\ref{rem:Main2}. 
    \item In the case of $\cL(s)=\zeta(s)$ and $\alpha=3/4$, estimate~\eqref{eq:LogDerSpecial} improves~\cite[Theorem 1]{chirreGoncalves}. This was already announced in~\cite[p.~16]{ChirreSimonicHagen}, although without a detailed proof. Moreover, the same estimate for $\alpha=1.278$ improves also~\cite[Theorem 2]{chirreGoncalves}.
    \item In the case of $\cL(s)=L(s,\pi)$, $L(s,\pi)$ entire, $m=\sdeg\geq 3$ and $\alpha=1.6$, estimate~\eqref{eq:LogSpecial} improves~\cite[Corollary 2]{ChirreANote} in the range $0.56\leq\sigma\leq0.999$. Note that in~\cite[Theorems~1 and~2 (cases ``otherwise'')]{CarneiroChandee}, the coefficients of the main error terms should be 
    \begin{equation*}
        \pm \frac{1}{2}\left(1+\frac{2\sigma-1}{\sigma(1-\sigma)}\right) \quad\textrm{instead of}\quad \pm \left(\frac{1}{2}+\frac{2\sigma-1}{\sigma(1-\sigma)}\right), 
    \end{equation*}
    see~\cite[Eq.~(2.19) and Case~3 on p.~373]{CarneiroChandee}. This is also in accordance with the previously mentioned result by Chirre. Hence, our results for $\cL(s)=\zeta(s)$ do not improve these results, see also Section~\ref{sec:DiscussionLog}.
    \item Carneiro, Chirre and Milinovich~\cite{Carneiro} obtained RH estimate for $S_{0,\sigma}(t)=\frac{1}{\pi}\arg{\zeta(\sigma+\ie t)}$. Note that $\left|S_{0,\sigma}(t)\right|\leq \frac{1}{\pi}\left|\log{\zeta(s)}\right|$ for $\sigma\in(1/2,1)$ and under RH. Estimate~\eqref{eq:LogSpecial} for $\cL(s)=\zeta(s)$ and $\alpha=1.38$ then improves~\cite[Corollary 5]{Carneiro} in the range $0.55\leq\sigma\leq 1-\delta$ for fixed $\delta\in(0,0.45)$.
    \item In~\cite{GarciaLee22}, explicit lower and upper bounds on the residue of the Dedekind zeta functions are provided under the assumption of GRH and the Dedekind conjecture. Under the same conditions and for large $\left|\Delta_{\mathbb{K}}\right|$, we improve their upper bounds, see Corollary~\ref{cor:DedekindRes}.
\end{enumerate}

The outline of this paper is as follows. In Section~\ref{sec:MainResultsNE} we formulate the main results in a non-explicit setting, while in Section~\ref{sec:Explicit} we formulate effective versions of Theorems~\ref{thm:MainNonExplicit} and~\ref{thm:1line}, together with their simplifications. In Section~\ref{sec:SMF} we revise Selberg's moment formula for functions in the Selberg class, and in Section~\ref{sec:SumsPrimes} we derive estimates for the corresponding sums over prime numbers by considering whether $\cL\in\cS$ or $\cL\in\cSP$. Section~\ref{sec:SumZeros} is devoted to the extension of the Guinand--Weil exact formula to the Selberg class and estimation of the sum over the non-trivial zeros. The proofs of Theorems~\ref{thm:MainGeneral}--\ref{thm:1line} are presented in Section~\ref{sec:ProofFirst}, the proofs of Theorems~\ref{thm:MainExplicit} and~\ref{thm:1LineExplicit} are provided in Section~\ref{sec:ProofSecond}, and the proofs of Corollaries ~\ref{cor:ZetaExplicit} and~\ref{corol:FixedRegions} are given in Section~\ref{sec:simpl}. Section~\ref{sec:discussion} is devoted to a discussion.

\subsection*{Acknowledgements}

A.S.~would like to thank N.P.~and The University of Helsinki for their hospitality during his visit in June 2022. He also gratefully acknowledges the Lift-off Fellowship support from The Australian Mathematical Society. N. P. would like to thank the Emil Aaltonen Foundation for supporting her work. The authors would like to thank Tim Trudgian, J\"{o}rn Steuding, Andr\'{e}s Chirre and Bryce Kerr for taking time to read the manuscript. 

\section{The main results -- non-explicit setting}
\label{sec:MainResultsNE}

Note that $|a(n)|\leq\mathcal{C}_{\cL}^{R}(\varepsilon)n^{\varepsilon}$ for any sufficiently small $\varepsilon>0$, and $\left|b\left(p^k\right)\right|\leq \mathcal{C}^{E}_{\cL} p^{k\theta}$ for some $\theta\in[0,1/2)$, are quantitative versions of the Ramanujan hypothesis and the Euler product representation, respectively. Also, we are going to use the following notation throughout this paper: \begin{equation}
\label{eq:lambdapm}
\lambda^{-} \leq \min_{1\leq j\leq f}\left\{\lambda_{j}\right\} \leq \max_{1\leq j\leq f}\left\{\lambda_{j}\right\} \leq \lambda^{+}, \quad
\max_{1\leq j\leq f}\left\{\left|\mu_{j}\right|\right\} \leq \mu^{+},
\end{equation}
where $f$, $\lambda_j$ and $\mu_j$ are from the functional equation for $\cL$. In addition, $m_{\cL}$ is the order of pole of $\cL(s)$ at $s=1$. We are stating our results uniformly in $\cL$, i.e., the implied $O$-constants are independent of the parameters from the axioms of the Selberg class unless otherwise indicated by $O_{(\cdot)}$. 

\subsection{General estimates for $\cL\in\cS$}

Let $\cL$ be the element in the Selberg class. The generalized von Mangoldt function $\Lambda_{\cL}(n)$ is defined by
\[
\frac{\cL'}{\cL}(s) = -\sum_{n=1}^{\infty}\frac{\Lambda_{\cL}(n)}{n^s}
\]
for $\sigma>1$, and obviously $\Lambda_{\zeta}(n)=\Lambda(n)$ for $\Lambda(n)$ being the von Mangoldt function. We can see that the Euler product representation implies $\Lambda_{\cL}(n)=b(n)\log{n}$ with $\Lambda_{\cL}(n)=0$ for $n$ not being a prime power. Because $a(p)=b(p)$, we have
\begin{equation}
\label{eq:PsiTilde}
\widetilde{\psi}_{\cL}(x)\de\sum_{n\leq x} \left|\Lambda_{\cL}(n)\right| = \sum_{p\leq x}\left|a(p)\right|\log{p} + \sum_{k=2}^{\lfloor\frac{\log{x}}{\log{2}}\rfloor}\sum_{p\leq x^{\frac{1}{k}}}\left|\Lambda_{\cL}\left(p^k\right)\right|.
\end{equation}
It is easy to see that~\eqref{eq:PsiTilde} implies
\begin{equation}
\label{eq:MainForPsi}
\widetilde{\psi}_{\cL}(x) = \sum_{p\leq x}\left|a(p)\right|\log{p} + O\left(\mathcal{C}_{\cL}^{E} x^{\frac{1}{2}+\theta}\log^{2}{x}\right)
\end{equation}
for all $x\geq 2$, see Section~\ref{sec:PrimesForS}. Relation~\eqref{eq:MainForPsi} is fundamental for proving Theorems~\ref{thm:MainGeneral}--\ref{thm:MainGeneral1line}.

As usual and in analogy with $\log{\zeta(s)}$, we define
\[
\log{\cL(s)} = \sum_{n=2}^{\infty}\frac{\Lambda_{\cL}(n)}{n^{s}\log{n}}
\]
for $\sigma>1$, where the value of $\log{\cL(s)}$ is that which tends to $0$ as $\sigma\to\infty$ for any fixed $t$. For $\sigma<2$ and $t\neq 0$, we define $\log{\cL(s)}$ as the analytic continuation of the above function along the straight line from $2+\ie t$ to $\sigma+\ie t$, provided that $\cL(s)\neq 0$ on this segment of line. Note that GRH implies holomorphicity of $\log{\cL(z)}$ on the domain $\left\{z\in\C\colon \Re\{z\}>1/2\right\}\setminus(1/2,1]$.

Our first result is a generalization of~\eqref{eq:LittlewoodSpecial} for functions in the Selberg class. Before stating it, we provide a few definitions used in the statement. We define
\begin{gather}
A\left(a,\alpha,u,\sigma\right) \de \frac{a(2\sigma-1)\left(1-\exp\left(-\frac{2\alpha(1-u)}{2\sigma-1}\right)\right)}{2\alpha(1-u)^2}, \label{eq:A} \\
\eta(\alpha,\sigma,\tau) \de \frac{1}{2}\left(1-\frac{\alpha}{(2\sigma-1)\log{\log{\tau}}}\right)^{-1}, \label{eq:eta}
\end{gather}
together with
\begin{equation}
\label{eq:MainGeneralA_1}
A_1\left(a,k,\varepsilon,\theta,\sigma,\tau\right) \de a\left(1+\left|\Theta_{\theta,\varepsilon}(\sigma)\right|^{k}\left(\log{\tau}\right)^{2\Theta_{\theta,\varepsilon}(\sigma)}\left(\log{\log{\tau}}\right)^{k}\right)\min{\left\{\frac{1}{\left|\Theta_{\theta,\varepsilon}(\sigma)\right|^{k+1}},\left(\log{\log{\tau}}\right)^{k+1}\right\}},
\end{equation}
\begin{equation}
\label{eq:MainGeneralA_2}
A_2\left(a,b,\varepsilon,\theta,\tau\right) \de a\left(\log{\tau}\right)^{2\varepsilon}\min\left\{\frac{1}{\varepsilon},\log{\log{\tau}}\right\} + b\min\left\{\frac{1}{\left|\Theta_{\theta,\varepsilon}(1)\right|^{3}},\left(\log{\log{\tau}}\right)^{3}\right\},
\end{equation}
\begin{equation}
\label{eq:MainGeneralA_3}
A_3\left(a,b,\varepsilon,\theta,\tau\right) \de \frac{a}{(\log{\tau})^{1-2\varepsilon}\log{\log{\tau}}} + \frac{b\log{\log{\tau}}}{(\log{\tau})^{-2\Theta_{\theta,\varepsilon}(3/2)}},
\end{equation}
and 
\begin{equation}
\label{def:theta}
    \Theta_{\theta,\varepsilon}(\sigma)\de 1/2+\max\{\theta,\varepsilon\}-\sigma. 
\end{equation}

\begin{theorem}
\label{thm:MainGeneral}
Let $\cL\in\cS$ and assume the Generalized Riemann Hypothesis for $\cL$ and the Riemann Hypothesis for $\zeta(s)$. Fix $\alpha\geq\log{2}$ and $\delta\in(0,1/2)$. Then
\begin{flalign}
\label{eq:MainGeneral}
\left|\frac{\cL'}{\cL}(s)\right| &\leq A\left((1+\varepsilon)\mathcal{C}_{\cL}^{R}(\varepsilon),\alpha,\sigma-\varepsilon,\sigma\right)\left(\log{\tau}\right)^{2(1-\sigma+\varepsilon)} 
+ \frac{e^{\alpha}+1}{2\alpha}\left(\log{\tau}\right)^{2-2\sigma} - \frac{(1+\varepsilon)\mathcal{C}_{\cL}^{R}(\varepsilon)\sigma 2^{1-\sigma+\varepsilon}}{1-\sigma+\varepsilon} \nonumber \\ 
&+ O\left(m_{\cL}\left(\log{\tau}\right)^{3-4\sigma}\right) + O\left(A_1\left(\mathcal{C}_{\cL}^{R}(\varepsilon)+\mathcal{C}_{\cL}^{E},2,\varepsilon,\theta,\sigma,\tau\right)\right) + O\left(\frac{A_2\left(\mathcal{C}_{\cL}^{R}(\varepsilon),\mathcal{C}_{\cL}^{R}(\varepsilon)+\mathcal{C}_{\cL}^{E},\varepsilon,\theta,\tau\right)}{\left(\log{\tau}\right)^{2\delta}}\right) \nonumber \\ 
&+ O_{\lambda^{+},\lambda^{-}}\left(\frac{f}{(\log{\tau})^{2\delta}}\right) + 
O\left(\frac{\sdeg+m_{\cL}}{\left(\log{\tau}\right)^{2\delta}}+m_{\cL}\left(\frac{\sq}{\tau}\right)^{\frac{2}{\sdeg}}\left(\log{\tau}\right)^{1-2\delta}\right)
\end{flalign}
and
\begin{flalign}
\label{eq:MainGeneralLog}
\left|\log{\cL(s)}\right| &\leq \eta(\alpha,\sigma,\tau)A\left((1+\varepsilon)\mathcal{C}_{\cL}^{R}(\varepsilon),\alpha,\sigma-\varepsilon,\sigma\right)\frac{\left(\log{\tau}\right)^{2(1-\sigma+\varepsilon)}}{\log{\log{\tau}}}
+ \left(\frac{e^{\alpha}+1}{4\alpha}\right)\frac{(\log{\tau})^{2-2\sigma}}{\log{\log{\tau}}} \nonumber \\  
&- \frac{(1+\varepsilon)\mathcal{C}_{\cL}^{R}(\varepsilon)\eta(\alpha,\sigma,\tau)}{(1-\sigma+\varepsilon)\log{\log{\tau}}}
+ \mathcal{C}_{\cL}^{R}(\varepsilon)\log{\left(2\log{\log{\tau}}\right)}
+O\left(\frac{\mathcal{C}_{\cL}^{R}(\varepsilon)\left(\log{\tau}\right)^{2(1-\sigma+\varepsilon)}}{(1-\sigma+\varepsilon)^{2}\left(\log{\log{\tau}}\right)^2}\right) \nonumber \\ &+ O\left(\frac{m_{\cL}\left(\log{\tau}\right)^{3-4\sigma}}{\log{\log{\tau}}}\right) + O\left(A_1\left(\mathcal{C}_{\cL}^{R}(\varepsilon)+\mathcal{C}_{\cL}^{E},1,\varepsilon,\theta,\sigma,\tau\right)\right) \nonumber \\  
&+ O\left(\frac{A_2\left(\mathcal{C}_{\cL}^{R}(\varepsilon),\mathcal{C}_{\cL}^{R}(\varepsilon)+\mathcal{C}_{\cL}^{E},\varepsilon,\theta,\tau\right)+\sdeg+m_{\cL}}{\left(\log{\tau}\right)^{2\delta}\log{\log{\tau}}}\right) + O_{\lambda^{+},\lambda^{-}}\left(\frac{f}{(\log{\tau})^{2\delta}\log{\log{\tau}}}\right) \nonumber \\ 
&+ O\left(A_3\left(\frac{\mathcal{C}_{\cL}^{R}(\varepsilon)}{1-2\varepsilon},\mathcal{C}_{\cL}^{R}(\varepsilon)+\mathcal{C}_{\cL}^{E},\varepsilon,\theta,\tau\right)\right) + O\left(m_{\cL}\left(\frac{\sq}{\tau}\right)^{\frac{2}{\sdeg}}\frac{\left(\log{\tau}\right)^{1-2\delta}}{\log{\log{\tau}}}\right)
\end{flalign}
for $1/2+\delta\leq\sigma<1$, for sufficiently large $\tau$ and $|t|\geq\mu^{+}/\lambda^{-}$, and $\varepsilon\in(0,\delta)$. The implied constants are absolute.
\end{theorem}

\begin{remark}
We can see that, for fixed $a,k>0$, fixed $\theta\in[0,1/2)$ and fixed $\sigma\in(1/2,1)$, 
\[
A_1\left(a,k,\varepsilon,\theta,\sigma,\tau\right) \ll \left(\log{\log{\tau}}\right)^{k+1} + \left(\log{\tau}\right)^{1-2\sigma+2\max\{\theta,\varepsilon\}}\left(\log{\log{\tau}}\right)^{2k+1} = o\left(\left(\log{\tau}\right)^{2(1-\sigma+\varepsilon)}\right)
\]
as $\tau\to\infty$. Also, for fixed $a,b>0$ and fixed $\sigma\in(1/2,1)$,
\[
A_2\left(a,b,\varepsilon,\theta,\tau\right) \ll \left(\log{\tau}\right)^{2\varepsilon}\left(\log{\log{\tau}}\right)^{3} = o\left(\left(\log{\tau}\right)^{2(1-\sigma+\varepsilon)}\right) 
\]
as $\tau\to\infty$. It is clear now that all $O$-terms in~\eqref{eq:MainGeneral} and~\eqref{eq:MainGeneralLog} are asymptotically smaller than the corresponding main terms.
\end{remark}

\begin{remark}
We do not need the Ramanujan hypothesis to hold for all positive integers $n$ in Theorem \ref{thm:MainGeneral}. Instead, it is sufficient that $|a(p)|\leq\mathcal{C}_{\cL}^{R}(\varepsilon)p^{\varepsilon}$ holds for prime numbers $p$.
\end{remark}

It is clear that~\eqref{eq:LittlewoodSpecial} with explicit constants for the main terms can be recovered by taking $\varepsilon\to0$ in Theorem~\ref{thm:MainGeneral} since $\mathcal{C}_{\zeta}^{R}(\varepsilon)\equiv 1$. In fact, the same conclusion holds for all Dirichlet $L$-functions attached to primitive characters. In the case of Dedekind zeta functions one can take $\mathcal{C}_{\zeta_\K}^{R}(\varepsilon)\equiv n_{\K}$, and this cannot be improved due to the fact that there are always infinitely many primes which split completely~\cite[Theorem~4.37]{Narkiewicz04Book}. Note that in these cases, the strength of Theorem~\ref{thm:MainGeneral} is comparable to that of Corollary~\ref{cor:thm2}, where the existence of a polynomial Euler product is guaranteed. However, for $\zeta_{\K}(s)$ better bounds can be obtained due to the prime ideal theorem, see Section~\ref{sec:discussion2}.

\subsection{Estimates for $\cL\in\cS$ under additional assumptions.}
\label{sec:Distr}

At present, the authors are not able to remove $\varepsilon$ from Theorem~\ref{thm:MainGeneral} without imposing additional conjectures on the coefficients $a(n)$ and also keeping all the axioms for $\mathcal{S}$. In order to do so, we are proposing the following assumption.

\begin{conjecture}
\label{conj:SelbergVariant}
Let $\cL\in\cS$. Then 
\[
\sum_{p\leq x}\left|a(p)\right|^2 \leq \mathcal{C}_{\cL}^{P_1}(x)\frac{x}{\log{x}} + \mathcal{C}_{\cL}^{P_2}\frac{x}{\log^{2}{x}}
\]
for all $x\geq 2$, where $0<\mathcal{C}_{\cL}^{P_1}(x)=o(x^{\varepsilon})$ for every $\varepsilon>0$, $\mathcal{C}_{\cL}^{P_1}(x)$ is an increasing continuous function, there exists an absolute $x_0>0$ such that $x^{-1/2}\mathcal{C}_{\cL}^{P_1}(x)$ is a decreasing function for $x\geq x_0$, and $\mathcal{C}_{\cL}^{P_2}$ is a non-negative constant.
\end{conjecture}

Conjecture~\ref{conj:SelbergVariant} contains many open problems that are related to the distribution of $a(n)$ on prime numbers. For example, one can conjecture that ($\cL\not\equiv1$)
\begin{equation}
\label{eq:SelbergAGeneral}
\sum_{p\leq x}\frac{\left|a(p)\right|^2}{p} = n_{\cL}\log{\log{x}} + C_{\cL} + \mathcal{R}(x)
\end{equation}
for some positive real number $n_{\cL}$, for some constant $C_{\cL}$, and for some suitably chosen remainder term. Indeed, \emph{Selberg's conjecture A} then asserts\footnote{See~\cite{SelbergOldAndNew}. He also conjectured that $n_{\cL}=1$ if $\cL$ is primitive.} $n_{\cL}\in\N$ and $\mathcal{R}(x)=O(1)$, while the \emph{strong Selberg's conjecture A} asserts\footnote{A more general version was introduced in~\cite[p.~828]{BombieriHejhal}, see also~\cite[p.~120]{SteudingBook}.} $n_{\cL}\in\N$ and $\mathcal{R}(x)\ll 1/\log{x}$. The \emph{normality conjecture}\footnote{This conjecture was introduced in~\cite[p.~257]{KacPer}, where it was shown that it implies non-vanishing of $\cL$ on the 1-line. Moreover, it was also proposed that $n_{\cL}\leq 1$ if $\cL$ is primitive.} is weaker than Selberg's conjecture A since it asserts only $\mathcal{R}(x)=o(\log{\log{x}})$. By partial summation, the \emph{prime mean-square conjecture}\footnote{This is one of the axioms of the class $\tilde{\mathcal{S}}$, see~\cite[Section~2.2]{SteudingBook}.} 
\[
\sum_{p\leq x}\left|a(p)\right|^2 \sim \kappa\frac{x}{\log{x}}, \quad \kappa>0
\]
follows from~\eqref{eq:SelbergAGeneral} for $n_{\cL}=\kappa$ and $\mathcal{R}(x)=o\left(1/\log{x}\right)$. Similarly, the \emph{strong prime mean-square conjecture}\footnote{In~\cite[p.~346]{AistleitnerPankowski2017}, this is referred to as ``the Selberg normality conjecture in the stronger form''.}
\[
\sum_{p\leq x}\left|a(p)\right|^2 = \kappa\frac{x}{\log{x}} + O\left(\frac{x}{\log^{2}{x}}\right), \quad \kappa>0
\]
follows from~\eqref{eq:SelbergAGeneral} for $n_{\cL}=\kappa$ and $\mathcal{R}(x)\ll 1/\log^{2}{x}$. Using partial summation, the above conjectures imply Conjecture~\ref{conj:SelbergVariant} with the following parameters: 
\begin{enumerate}
    \item Normality conjecture: For every $\varepsilon>0$ there exists $\mathcal{C}_{\cL}^{P}(\varepsilon)>0$ such that $\mathcal{C}_{\cL}^{P_2}=\mathcal{C}_{\cL}^{P}(\varepsilon)$ and $\mathcal{C}_{\cL}^{P_1}(x)=\varepsilon\log{x}\log{\log{x}}$. 
    \item Selberg's conjecture A: $\mathcal{C}_{\cL}^{P_2}=0$ and there is a constant $\mathcal{C}_{\cL}^{P}>0$ such that $\mathcal{C}_{\cL}^{P_1}(x)=\mathcal{C}_{\cL}^{P}\log{x}$.
    \item Strong Selberg's conjecture A: $\mathcal{C}_{\cL}^{P_2}=0$ and there is a constant $\mathcal{C}_{\cL}^{P}>0$ such that $\mathcal{C}_{\cL}^{P_1}(x)\equiv\mathcal{C}_{\cL}^{P}$.
    \item Prime mean-square conjecture: For every $\varepsilon>0$ there exists $\mathcal{C}_{\cL}^{P}(\varepsilon)>0$ such that $\mathcal{C}_{\cL}^{P_2}=\mathcal{C}_{\cL}^{P}(\varepsilon)$ and $\mathcal{C}_{\cL}^{P_1}(x)\equiv\kappa+\varepsilon$. 
    \item Strong prime mean-square conjecture: $\mathcal{C}_{\cL}^{P_2}>0$ and $\mathcal{C}_{\cL}^{P_1}(x)\equiv\kappa$.
\end{enumerate}
A hypothesis very similar to Conjecture~\ref{conj:SelbergVariant} is the following, somehow weaker, assertion.

\begin{conjecture}
\label{conj:SelbergVariant2}
Let $\cL\in\cS$. Then 
\[
\sum_{p\leq x}\left|a(p)\right|\leq \widehat{\mathcal{C}_{\cL}^{P_1}}(x)\frac{x}{\log{x}} + \widehat{\mathcal{C}_{\cL}^{P_2}}\frac{x}{\log^{2}{x}}
\]
for all $x\geq 2$, where $0<\widehat{\mathcal{C}_{\cL}^{P_1}}(x)=o(x^{\varepsilon})$ for every $\varepsilon>0$, $\widehat{\mathcal{C}_{\cL}^{P_1}}(x)$ is an increasing continuous function, there exists an absolute $x_0>0$ such that $x^{-1/4}\widehat{\mathcal{C}_{\cL}^{P_1}}(x)$ is a decreasing function for $x\geq x_0$, and $\widehat{\mathcal{C}_{\cL}^{P_2}}$ is a non-negative constant.
\end{conjecture}

Conjecture~\ref{conj:SelbergVariant2} is weaker than Conjecture~\ref{conj:SelbergVariant} in the sense that the latter clearly implies the former due to the Cauchy--Bunyakovsky--Schwarz inequality and the prime number theorem. It is also clear that if $|a(p)|\ll_{\cL} 1$, then both conjectures hold. This is believed to be true since one hopes for having $\cS=\cSP$, see Section~\ref{subsec:PolyEuler}. However, in many instances one can obtain better estimates of the same form as those from the conjectures, see Section~\ref{sec:discussion2}.

The following theorem states the results under Conjectures ~\ref{conj:SelbergVariant} or ~\ref{conj:SelbergVariant2}. Again, before providing it, we give a few definitions that are used in the results. We define
\begin{gather}
\widehat{A}\left(m,\alpha,\sigma\right) \de A\left(m,\alpha,\sigma,\sigma\right) + \frac{e^{\alpha}+1}{2\alpha}, \label{eq:WidehatA} \\
\widetilde{A}\left(m,\alpha,\sigma,\tau\right) \de \eta(\alpha,\sigma,\tau)A\left(m,\alpha,\sigma,\sigma\right)
+\frac{e^{\alpha}+1}{4\alpha}, \label{eq:TildeA}
\end{gather}
where $A\left(a,\alpha,u,\sigma\right)$ and $\eta(\alpha,\sigma,\tau)$ are defined by~\eqref{eq:A} and~\eqref{eq:eta}, respectively, together with
\begin{gather}
A_4(\sigma,\tau)\de m_3(\tau)\left(\frac{1}{1-\sigma}\left(1+\frac{1}{(1-\sigma)\log{\log{\tau}}}\right)+\frac{\log{\log{\tau}}\log{\log{\log{\tau}}}}{(\log{\tau})^{2-2\sigma}}\right), \label{eq:MainGeneralA_4} \\
A_5(\sigma,\tau)\de 
\frac{m_1(\tau)}{(1-\sigma)^{2}\log{\log{\tau}}} + m_3(\tau)\left(\frac{1}{1-\sigma}+\frac{\left((1-\sigma)\log{\log{\log{\tau}}}+1\right)\left(\log{\log{\tau}}\right)^{2}}{(\log{\tau})^{2-2\sigma}}\right). \label{eq:MainGeneralA_5}
\end{gather}
Functions $m_1(\tau)$ and $m_3(\tau)$ will be defined in Theorem~\ref{thm:MainGeneralV2}. Now, the result is formulated as follows:

\begin{theorem}
\label{thm:MainGeneralV2}
Let $\cL\in\cS$ and assume the Generalized Riemann Hypothesis for $\cL$. Fix $\alpha\geq\log{2}$ and $\delta\in(0,1/2)$. Define $m_i(\tau)$ for $i\in\{1,2,3,4\}$ in the following way: 
\begin{enumerate}
    \item If Conjecture~\ref{conj:SelbergVariant} is true for $\cL$, then
    \begin{equation}
    \label{eq:m123}
    m_1(\tau)\de \sqrt{\mathcal{C}_{\cL}^{P_1}\left(\log^{2}{\tau}\right)}, \quad m_4(\tau) \de m_2(\tau)\de \sqrt{\mathcal{C}_{\cL}^{P_1}\left(\log^{2}{\tau}\right)+\mathcal{C}_{\cL}^{P_2}}, \quad
    m_3(\tau)\de \frac{m_2(\tau)}{\sqrt{\log{\log{\tau}}}}.
    \end{equation}
    \item If Conjecture~\ref{conj:SelbergVariant2} is true for $\cL$, then
    \begin{equation}
    \label{eq:m123V2}
    m_1(\tau)\de \widehat{\mathcal{C}_{\cL}^{P_1}}\left(\log^{2}{\tau}\right), \quad m_2(\tau)\de \widehat{\mathcal{C}_{\cL}^{P_2}}, \quad
    m_3(\tau)\de \frac{m_2(\tau)}{\log{\log{\tau}}}, \quad m_4(\tau)\de m_1(\tau)+m_2(\tau).
    \end{equation}
\end{enumerate}

Then we have
\begin{flalign}
\label{eq:MainGeneralV2}
\left|\frac{\cL'}{\cL}(s)\right| &\leq \widehat{A}\left(m_1(\tau),\alpha,\sigma\right)\left(\log{\tau}\right)^{2-2\sigma} + O\left(A_4(\sigma,\tau)\left(\log{\tau}\right)^{2-2\sigma}\right) - \frac{m_1(\tau)\sigma 2^{1-\sigma}}{1-\sigma} + O\left(m_{\cL}\left(\log{\tau}\right)^{3-4\sigma}\right) \nonumber \\ 
&+ O\left(A_1\left(\mathcal{C}_{\cL}^{E},2,0,\theta,\sigma,\tau\right)\right) + O\left(\frac{A_2\left(m_4(\tau),\mathcal{C}_{\cL}^{E},0,\theta,\tau\right)}{\left(\log{\tau}\right)^{2\delta}}\right) + O_{\lambda^{+},\lambda^{-}}\left(\frac{f}{(\log{\tau})^{2\delta}}\right) \nonumber \\ 
&+ O\left(\frac{\sdeg+m_{\cL}}{\left(\log{\tau}\right)^{2\delta}}+m_{\cL}\left(\frac{\sq}{\tau}\right)^{\frac{2}{\sdeg}}\left(\log{\tau}\right)^{1-2\delta}\right)
\end{flalign}
and
\begin{flalign}
\label{eq:MainGeneralLogV2}
\left|\log{\cL(s)}\right| &\leq 
\widetilde{A}\left(m_1(\tau),\alpha,\sigma,\tau\right)\frac{\left(\log{\tau}\right)^{2-2\sigma}}{\log{\log{\tau}}} + O\left(A_5(\sigma,\tau)\frac{\left(\log{\tau}\right)^{2-2\sigma}}{\log{\log{\tau}}}\right)
- \frac{m_1(\tau)\eta(\alpha,\sigma,\tau)}{(1-\sigma)\log{\log{\tau}}} \nonumber \\
&+ m_1(\tau)\log{\left(2\log{\log{\tau}}\right)}
+ O\left(\frac{m_{\cL}\left(\log{\tau}\right)^{3-4\sigma}}{\log{\log{\tau}}}\right) 
+ O\left(A_1\left(\mathcal{C}_{\cL}^{E},1,0,\theta,\sigma,\tau\right)\right) \nonumber \\ 
&+ O\left(\frac{A_2\left(m_4(\tau),\mathcal{C}_{\cL}^{E},0,\theta,\tau\right)+\sdeg+m_{\cL}}{\left(\log{\tau}\right)^{2\delta}\log{\log{\tau}}}\right) + O_{\lambda^{+},\lambda^{-}}\left(\frac{f}{(\log{\tau})^{2\delta}\log{\log{\tau}}}\right) \nonumber \\ 
&+ O\left(A_3\left(m_4(\tau),\mathcal{C}_{\cL}^{E},0,\theta,\tau\right)\right) 
+ O\left(m_{\cL}\left(\frac{\sq}{\tau}\right)^{\frac{2}{\sdeg}}\frac{\left(\log{\tau}\right)^{1-2\delta}}{\log{\log{\tau}}}\right)
\end{flalign}
for $1/2+\delta\leq\sigma<1$ and sufficiently large $\tau$ and $|t|\geq\mu^{+}/\lambda^{-}$. 
Functions $A_1\left(a,k,\varepsilon,\theta,\sigma,\tau\right)$, $A_2\left(a,b,\varepsilon,\theta,\tau\right)$ and $A_3\left(a,b,\varepsilon,\theta,\tau\right)$ are defined by~\eqref{eq:MainGeneralA_1},~\eqref{eq:MainGeneralA_2} and~\eqref{eq:MainGeneralA_3}, respectively. The implied constants are absolute.
\end{theorem}

\begin{remark}
Note that the second terms in~\eqref{eq:MainGeneralV2} and~\eqref{eq:MainGeneralLogV2} are asymptotically smaller than the main terms.
\end{remark}

\begin{theorem}
\label{thm:MainGeneral1line}
Let $\cL\in\cS$ and assume the Generalized Riemann Hypothesis for $\cL$. Define $m_i(\tau)$ for $i\in\{1,2,3,4\}$ by~\eqref{eq:m123} if Conjecture~\ref{conj:SelbergVariant} is true for $\cL$, and by~\eqref{eq:m123V2} if Conjecture~\ref{conj:SelbergVariant2} is true for $\cL$. Then
\begin{multline}
\label{eq:MainGeneral1lineV1}
\left|\frac{\cL'}{\cL}(1+\ie t)\right| \leq \left(2m_1(\tau)+O\left(m_3(\tau)\log{\log{\log{\tau}}}\right)\right)\log{\log{\tau}} + O\left(1+\frac{\mathcal{C}_{\cL}^{E}}{\left(\frac{1}{2}-\theta\right)^{3}}\right) \\
+ O\left(\frac{m_4(\tau)\log{\log{\tau}}+\sdeg+m_{\cL}}{\log{\tau}}+m_{\cL}\left(\frac{q_{\cL}}{\tau}\right)^{\frac{2}{\sdeg}}\right) + O_{\lambda^{+},\lambda^{-}}\left(\frac{f}{\log{\tau}}\right)
\end{multline}
and
\begin{multline}
\label{eq:MainGeneral1lineV2}
\left|\log{\cL(1+\ie t)}\right| \leq m_1(\tau)\left(\log{\left(2\log{\log{\tau}}\right)}+2\right) + O\left(m_2(\tau)+\frac{\mathcal{C}_{\cL}^{E}}{\left(\frac{1}{2}-\theta\right)^{2}}\left(1+\frac{1}{\left(\frac{1}{2}-\theta\right)\log{\tau}\log{\log{\tau}}}\right)\right) \\
+ O\left(\frac{1+m_{\cL}\left(\frac{q_{\cL}}{\tau}\right)^{\frac{2}{\sdeg}}}{\log{\log{\tau}}}+\frac{m_4(\tau)}{\log{\tau}}+\frac{\sdeg+m_{\cL}}{\log{\tau}\log{\log{\tau}}}\right) + O_{\lambda^{+},\lambda^{-}}\left(\frac{f}{\log{\tau}\log{\log{\tau}}}\right)
\end{multline}
for sufficiently large $\tau$ and $|t|\geq\mu^{+}/\lambda^{-}$. The implied constants are absolute.
\end{theorem}

\begin{remark}
\label{rem:Main2}
In relation to the (resp.~strong) prime mean-square conjecture, one could drop the square and hypothesize that 
\begin{equation}
\label{eq:PrimeMean}
\sum_{p\leq x}\left|a(p)\right| \sim \kappa\frac{x}{\log{x}} \quad \left(\textrm{resp.}\; \sum_{p\leq x}\left|a(p)\right| = \kappa\frac{x}{\log{x}} + O\left(\frac{x}{\log^{2}{x}}\right)\right)
\end{equation}
for $\kappa>0$. If this is true, then for every $\varepsilon>0$ Conjecture~\ref{conj:SelbergVariant2} holds with $\widehat{\mathcal{C}_{\cL}^{P_1}}(x)\equiv\kappa+\varepsilon$ and $\widehat{\mathcal{C}_{\cL}^{P_2}}$ is a constant that depends also on $\varepsilon$ (resp.~Conjecture~\ref{conj:SelbergVariant2} holds with $\widehat{\mathcal{C}_{\cL}^{P_1}}(x)\equiv\kappa$). Therefore, Theorems~\ref{thm:MainGeneralV2} and~\ref{thm:MainGeneral1line} for $\log{\cL(s)}$ provide a precise form of a similar result by Aistleitner and Pa\'{n}kowski, see~\cite[Proposition 1.1]{AistleitnerPankowski2017}.
\end{remark}

\begin{remark}
\label{thm:1Conjectures}
Let $\cL\in\cS$ be entire and assume the Generalized Riemann Hypothesis for $\cL$. Define $m_i(\sq)$ for $i\in\{1,2,3,4\}$ as in Theorem \ref{thm:MainGeneralV2}. When we replace $m_\cL$ with $0$, $\tau$ with $\sq$, $O(\sdeg/\log{\tau})$ with $O_{\lambda^-,\mu^+}(\sdeg/\log{\sq})$ and $O(\sdeg/(\log{\tau} \log{\log{\tau}}))$ with $O_{\lambda^-,\mu^+}(\sdeg/(\log{\sq}\log{\log{\sq}}))$ in \eqref{eq:MainGeneral1lineV1} and \eqref{eq:MainGeneral1lineV2}, we obtain estimates for the case $t=0$ when $\sq$ is large enough.
\end{remark}

\subsection{Estimates when $\cL$ has a polynomial Euler product}
\label{subsec:PolyEuler}

The main reason for a contribution of $\varepsilon$ in~\eqref{eq:MainGeneral} and~\eqref{eq:MainGeneralLog} is because of using $|a(p)|\leq\mathcal{C}_{\cL}^{R}(\varepsilon)p^{\varepsilon}$ in the estimation of the main term in~\eqref{eq:MainForPsi}, see Section~\ref{sec:PrimesForS}. However, we can completely avoid~\eqref{eq:MainForPsi} by assuming that $\cL(s)$ has a polynomial Euler product representation, i.e., axiom~(4) is replaced by
\[
\mathcal{L}(s) = \prod_{p}\prod_{j=1}^{m}\left(1-\frac{\alpha_j(p)}{p^s}\right)^{-1},
\]
where $m\in\N$ is the \emph{order of a polynomial Euler product} and $\alpha_j(p)\in\C$, $1\leq j\leq m$, are some functions which are defined for every prime number $p$. Denote by $\cSP$ the class of such functions. If $\cL\in\cSP$, then
\[
b\left(p^k\right) = \frac{1}{k}\sum_{j=1}^{m}\left(\alpha_j(p)\right)^k,
\]
and because $\left|\alpha_j(p)\right|\leq 1$ for $1\leq j\leq m$, see~\cite[Lemma 2.2]{SteudingBook}, we have that
\begin{equation}
\label{eq:BoundOnLambda}
\left|\Lambda_{\cL}(n)\right|\leq m\Lambda(n).
\end{equation}
Note that then $\theta=0$ and $\mathcal{C}_{\cL}^{E}=m$. Of course $\cSP\subseteq\cS$ and it is conjectured\footnote{This is due to conjectures belonging to the Langlands program, namely that the most general $L$-function should be a product of automorphic $L$-functions, and that for these functions the Ramanujan--Petersson conjecture should hold.} that $\cSP=\cS$. We will show that inequality~\eqref{eq:BoundOnLambda} indeed guarantees the expected conditional estimates for the logarithmic derivative and the logarithm of $\cL(s)$, see Theorem~\ref{thm:MainNonExplicit}. Remember that $\lambda^{+}$, $\lambda^{-}$ and $\mu^{+}$ are as in~\eqref{eq:lambdapm}. 

\begin{theorem}
\label{thm:MainNonExplicit}
Let $\cL\in\cSP$ with a polynomial Euler product of order $m$. Assume the Generalized Riemann Hypothesis for $\cL$ and the Riemann Hypothesis for $\zeta(s)$. Fix $\alpha\in[\log{2},2)$. Then
\begin{equation}
\label{eq:MainNonExplicit}
\left|\frac{\cL'}{\cL}(s)\right| \leq \left\{
                                      \begin{array}{ll}
                                        \widehat{A}\left(m,\alpha,\sigma\right)\left(\log{\tau}\right)^{2-2\sigma} - \frac{m\sigma 2^{1-\sigma}}{1-\sigma} & \\
                                        +\;
                                        O\left(m_{\cL}\left(\log{\tau}\right)^{3-4\sigma} + \frac{m}{(2\sigma-1)^3}\right)
                                        &
                                        \\
                                        +\;
                                        O\left(\frac{\sdeg+m_{\cL}}{\left(\log{\tau}\right)^{2\sigma-1}}
                                        +m_{\cL}\left(\frac{\sq}{\tau}\right)^{\frac{2}{\sdeg}}\left(\log{\tau}\right)^{2-2\sigma}\right)
                                        &
                                        \\
                                        +\;
                                        O_{\lambda^{+},\lambda^{-}}\left(\frac{f}{\left(\log{\tau}\right)^{(2-\alpha)\sigma}} +
                                        \frac{f}{\left(\log{\tau}\right)^{2\sigma-1}}\right),
                                        &
                                        \multicolumn{1}{l}{\smash{\raisebox{2\normalbaselineskip}{$\left|\frac{3}{4}-\sigma\right|\leq\frac{1}{4}-\frac{1}{\log{\log{\tau}}}$,}}} \\
                                        2m\log{\log{\tau}} + O\left(m+m|1-\sigma|\left(\log{\log{\tau}}\right)^{2}\right) &
                                        \\
                                        +\;
                                        O\left(\frac{\sdeg+m_{\cL}}{\log{\tau}}+m_{\cL}\left(\frac{\sq}{\tau}\right)^{\frac{2}{\sdeg}}\right) +
                                        O_{\lambda^{+},\lambda^{-}}\left(\frac{f}{\log{\tau}}\right)
                                        &
                                        \multicolumn{1}{l}{\smash{\raisebox{1\normalbaselineskip}{$|1-\sigma|\log{\log{\tau}}\leq 1$}}}
                                      \end{array}
                                      \right.
\end{equation}
and
\begin{equation}
\label{eq:MainNonExplicit2}
\left|\log{\cL(s)}\right| \leq \left\{
                               \begin{array}{ll}
                               \widetilde{A}\left(m,\alpha,\sigma,\tau\right)\frac{\left(\log{\tau}\right)^{2-2\sigma}}{\log{\log{\tau}}}
                               -\frac{m\eta(\alpha,\sigma,\tau)}{(1-\sigma)\log{\log{\tau}}}
                               + m\log{\left(2\log{\log{\tau}}\right)} & \\
                               +\;
                               O\left(\frac{m\left(\log{\tau}\right)^{2-2\sigma}}{(1-\sigma)^{2}\left(\log{\log{\tau}}\right)^{2}}
                               +\frac{m_{\cL}\left(\log{\tau}\right)^{3-4\sigma}}{\log{\log{\tau}}}+\frac{m}{(2\sigma-1)^{2}}\right)
                               & \\
                               +\;
                               O\left(\frac{\sdeg+m_{\cL}}{\left(\log{\tau}\right)^{2\sigma-1}\log{\log{\tau}}}
                               +m_{\cL}\left(\frac{\sq}{\tau}\right)^{\frac{2}{\sdeg}}\frac{\left(\log{\tau}\right)^{2-2\sigma}}{\log{\log{\tau}}}\right)
                               & \\
                               +\;
                               O\left(\frac{m}{\left(\log{\tau}\right)^{1-\alpha/2}\log{\log{\tau}}}\right)
                               +O_{\lambda^{-}}\left(\frac{f}{\left(\log{\tau}\right)^{2\sigma-1}\log{\log{\tau}}}\right)
                               & \\
                               +\;
                               O_{\lambda^{+}}\left(\frac{f}{(2-\alpha)\left(\log{\tau}\right)^{(2-\alpha)\sigma}\log{\log{\tau}}}\right),
                               &
                               \multicolumn{1}{l}{\smash{\raisebox{3\normalbaselineskip}{$\left|\frac{3}{4}-\sigma\right|\leq\frac{1}{4}-\frac{1}{\log{\log{\tau}}}$,}}} \\
                               m\log{\left(2\log{\log{\tau}}\right)} + O\left(m+\frac{\sdeg+m_{\cL}}{\log{\tau}\log{\log{\tau}}}+\left(\frac{\sq}{\tau}\right)^{\frac{2}{\sdeg}}\frac{m_{\cL}}{\log{\log{\tau}}}\right) & \\
                               +\;
                               O_{\lambda^{+},\lambda^{-}}\left(\frac{f}{\log{\tau}\log{\log{\tau}}}\right),
                               &
                               \multicolumn{1}{l}{\smash{\raisebox{1\normalbaselineskip}{$|1-\sigma|\log{\log{\tau}}\leq 1$}}}
                               \end{array}
                               \right.
\end{equation}
for sufficiently large $\tau$ and $|t|\geq\mu^{+}/\lambda^{-}$. Functions $\widehat{A}\left(m,\alpha,\sigma\right)$ and $\widetilde{A}\left(m,\alpha,\sigma,\tau\right)$ are defined by~\eqref{eq:WidehatA} and~\eqref{eq:TildeA}, respectively.
\end{theorem}

\begin{corollary}
\label{cor:thm2}
Under the same assumptions as in Theorem \ref{thm:MainNonExplicit}, we then have
\begin{gather}
\left|\frac{\cL'}{\cL}(s)\right| \leq \left(\widehat{A}(m,\alpha,\sigma)+o(1)\right)\left(\log{\tau}\right)^{2-2\sigma}, \label{eq:LogDerSpecial} \\
\left|\log{\cL(s)}\right| \leq \left(\widetilde{A}\left(m,\alpha,\sigma,\tau\right)+o(1)\right)\frac{\left(\log{\tau}\right)^{2-2\sigma}}{\log{\log{\tau}}} \label{eq:LogSpecial}
\end{gather}
for $1/2+\delta\leq\sigma\leq1-\delta$, fixed $\delta\in(0,1/4)$ and $\tau\to\infty$, where the implied $o$-constants may depend on $\cL$. 
\end{corollary}

For $s=1+\ie t$ we are able to improve upon the error terms from Theorem~\ref{thm:MainNonExplicit}.

\begin{theorem}
\label{thm:1line}
Let $\cL\in\cSP$ with a polynomial Euler product of order $m$ and assume the Generalized Riemann Hypothesis for $\cL$ and the Riemann Hypothesis for $\zeta(s)$. Fix $\alpha\geq\log{2}$. Then
\begin{multline}
\label{eq:LogDerOn1}
\left|\frac{\cL'}{\cL}(1+\ie t)\right| \leq 2m\log{\log{\tau}} - m\left(\gamma+\alpha\right) + \frac{e^{\alpha}+1}{2\alpha} \\ + O\left(\frac{m\left(\log{\log{\tau}}\right)^2}{\log{\tau}}+\frac{\sdeg+m_{\cL}}{\log{\tau}}+m_{\cL}\left(\frac{\sq}{\tau}\right)^{\frac{2}{\sdeg}}\right)
+ O_{\lambda^{+},\lambda^{-}}\left(\frac{f}{\log{\tau}}\right)
\end{multline}
and
\begin{multline}
\label{eq:LogOn1}
\left|\log{\cL(1+\ie t)}\right| \leq m\log{\left(2\log{\log{\tau}}\right)} + m\gamma + \frac{e^{\alpha}+1}{4\alpha\log{\log{\tau}}} \\ 
+ O\left(\frac{m\log{\log{\tau}}}{\log{\tau}}+\frac{\sdeg+m_{\cL}}{\log{\tau}\log{\log{\tau}}}+\frac{m_{\cL}}{\log{\log{\tau}}}\left(\frac{\sq}{\tau}\right)^{\frac{2}{\sdeg}}\right) + O_{\lambda^{+},\lambda^{-}}\left(\frac{f}{\log{\tau}\log{\log{\tau}}}\right)
\end{multline}
for sufficiently large $\tau$ and $|t|\geq\mu^{+}/\lambda^{-}$, and $\gamma$ is the Euler--Mascheroni constant.
\end{theorem}

Note that estimate~\eqref{eq:LogDerOn1} is a generalisation of~\eqref{eq:CHS}. It also shows that, for fixed $\cL$ and sufficiently large $\tau$, one can bound the right-hand side of~\eqref{eq:LogDerOn1} by $2m\log{\log{\tau}}$. For further discussion, see Section~\ref{sec:discussion}. 

\begin{remark}
\label{thm:1EulerProduct}
Let $\cL\in\cSP$ with a polynomial Euler product of order $m$ and assume that $\cL$ is entire. Assume the Generalized Riemann Hypothesis for $\cL$ and the Riemann Hypothesis for $\zeta(s)$.  Fix $\alpha\geq\log{2}$. Then for large enough $\sq$, we can estimate $\left|\cL'(1)/\cL(1)\right|$ and $\left|\log{\cL(1)}\right|$ by setting $m_\cL=0$ and replacing $\tau$ with $\sq$, and $O\left(\sdeg/\log{\tau}\right)$ and $O\left(\sdeg/\left(\log{\tau}\log{\log{\tau}}\right)\right)$ with $O_{\lambda^{-},\mu^{+}}\left(\sdeg/\log{\sq}\right)$ and $O_{\lambda^{-},\mu^{+}}\left(\sdeg/\left(\log{\sq}\log{\log{\sq}}\right)\right)$ in \eqref{eq:LogDerOn1} and \eqref{eq:LogOn1}, respectively.
\end{remark}

\section{The main results -- explicit setting}
\label{sec:Explicit}

In this section we are stating explicit versions of Theorems~\ref{thm:MainNonExplicit} and~\ref{thm:1line} under additional assumption of the strong $\lambda$-conjecture. The first result is left to quite long form since we want to provide sharp estimates for further use. We will provide shorter explicit estimates in Section~\ref{subsec:simplfied}.

\begin{theorem}
\label{thm:MainExplicit}
Let $\cL\in\cSP$ and assume the Generalized Riemann Hypothesis for $\cL$, the Riemann Hypothesis for $\zeta(s)$ and the strong $\lambda$-conjecture. Consider the following three cases:
\begin{enumerate}
  \item Let $2\alpha_1>\alpha\geq\log{2}$ and
        \begin{gather}
        \frac{1}{2}+\frac{\alpha_1}{\log{\log{\tau}}}\leq\sigma<1, \label{eq:sigmaRange1} \\
        \tau\geq\tau_0>\max\left\{e^{\sqrt{60}},\exp{\left(2^{1/\left(2-\alpha/\alpha_1\right)}\right)},
        \exp{\left(e^{2\alpha_1}\right)}\right\}, \quad
        |t| \ge t_0 \ge \max\left\{2\max_{1\leq j\leq f}\left\{\left|\mu_j\right|\right\},1\right\}. \nonumber
        \end{gather}
  \item Let $\alpha\geq\log{2}$, $\alpha_2>0$, $\alpha_3>0$ and
        \begin{gather}
        1-\frac{\alpha_2}{\log{\log{\tau}}} \leq \sigma \leq 1+\frac{\alpha_3}{\log{\log{\tau}}}, \label{eq:sigmaRange2} \\
        \tau\geq\tau_0>\max\left\{e^{\sqrt{60}},\exp{\left(\sqrt{2}\exp{\left(\frac{\alpha}{2\sigma-1}\right)}\right)},
        \exp{\left(e^{\alpha+2\alpha_2}\right)},
        \exp{\left(e^{4\alpha_2}\right)},\exp{\left(e^{2\alpha_3}\right)}\right\}, \nonumber \\
        |t| \ge t_0 \ge \max\left\{2\max_{1\leq j\leq f}\left\{\left|\mu_j\right|\right\},1\right\}. \nonumber
        \end{gather}
  \item Let $\alpha_3>0$, $\tau>e$ and $\sigma\geq 1+\alpha_3/\log{\log{\tau}}$.
\end{enumerate}
Under the conditions from~(1) we have
\begin{flalign}
\label{eq:1stCase}
\left|\frac{\cL'}{\cL}(s)\right| &\leq \widehat{A}\left(m,\alpha,\sigma\right)\left(\log{\tau}\right)^{2-2\sigma} - \frac{m\sigma 2^{1-\sigma}}{1-\sigma} + \frac{\mathfrak{a}\left(e^{\alpha}+1\right)}{\alpha}\left(\log{\tau}\right)^{3-4\sigma} \nonumber \\
&+ \frac{m\sigma 2^{\frac{3}{2}-\sigma}\left(4+\left(2+(2\sigma-1)\log{2}\right)^{2}\right)}{8\pi(2\sigma-1)^{3}} + \left(\frac{m\left(2\left(\alpha+1\right)\sigma-1\right)\left(e^{\alpha}-1\right)}{2\pi\alpha}\right)
\dfrac{(\log\log{\tau})^2}{(2\sigma-1)(\log{\tau})^{2\sigma-1}} \nonumber \\
&+ \left(\frac{2m\left(e^{\alpha}+1\right)}{\alpha}\right)
\frac{\log{\log{\tau}}}{\left(\log{\tau}\right)^{2\sigma-1}} + \frac{\mathfrak{b}\left(e^{\alpha}+1\right)}{\alpha\left(\log{\tau}\right)^{2\sigma-1}} \nonumber \\
&+ \frac{4.3\sdeg\left(\sigma-\frac{1}{2}\right)\left(1+\left(\log{\tau_0}\right)^{-\frac{\alpha}{\alpha_1}\sigma}\right)}
{\alpha\left(\log{\tau}\right)^{\left(2-\frac{\alpha}{\alpha_1}\right)\sigma}} + \frac{m_{\cL}}{\alpha}\left(\frac{\sq}{\tau}\right)^{\frac{2}{\sdeg}}\left(\log{\tau}\right)^{2-2\sigma}
\end{flalign}
and
\begin{flalign}
\label{eq:1stCaseB}
\left|\log{\cL(s)}\right| &\leq \widetilde{A}\left(m,\alpha,\sigma,\tau\right)\frac{\left(\log{\tau}\right)^{2-2\sigma}}{\log{\log{\tau}}}
- \frac{m\eta(\alpha,\sigma,\tau)}{(1-\sigma)\log{\log{\tau}}} + m\log{\left(2\log{\log{\tau}}\right)} \nonumber \\
&+m\left(\left(\nu_{2}\eta(\alpha,\sigma,\tau)\right)^{2}\exp{\left(-\frac{2\alpha(1-\sigma)}{2\sigma-1}\right)}\right)
\frac{\left(\log{\tau}\right)^{2-2\sigma}}{(1-\sigma)^{2}\left(\log{\log{\tau}}\right)^{2}}
+ \frac{\mathfrak{a}\left(e^{\alpha}+1\right)\left(\log{\tau}\right)^{3-4\sigma}}{4\alpha\log{\log{\tau}}} \nonumber \\
&+ \frac{m2^{\frac{3}{2}-\sigma}\left(-1+\sigma\left(4-\log{2}\right)+\sigma^2\log{4}\right)}{8\pi\left(2\sigma-1\right)^{2}}
+ \frac{m}{\nu_{1}^{2}}\exp{\left(-\frac{2\alpha(1-\sigma)}{(2\sigma-1)\nu_{2}}\right)}\left(\log{\tau}\right)^{\frac{2}{\nu_{2}}(1-\sigma)} \nonumber \\
&+ m\left(\frac{1-\sigma 2^{1-\sigma}}{(1-\sigma)\log{2}} - \log{\log{2}} + \int_{0}^{\nu_1}\theta_{1}(u)\dif{u}\right)
+ \frac{m\left(e^{\alpha}-1\right)\log{\log{\tau}}}{4\pi\alpha\left(\log{\tau}\right)^{2\sigma-1}} + \frac{m\left(e^{\alpha}+1\right)}{\alpha\left(\log{\tau}\right)^{2\sigma-1}} \nonumber \\
&+ \frac{\max\{0,\mathfrak{b}\}\left(e^{\alpha}+1\right)}{2\alpha\left(\log{\tau}\right)^{2\sigma-1}\log{\log{\tau}}}
+ \frac{2m\eta\left(\alpha,\sigma,\tau\right)\exp{\left(\frac{\alpha}{2\sigma-1}\right)}}{\log{\tau}\log{\log{\tau}}}
+ \frac{5m\exp{\left(\frac{2\alpha}{2\sigma-1}\right)}\left(1+\frac{\log{\log{\tau}}}{\eta\left(\alpha,\sigma,\tau\right)}\right)}{16\pi\log^{2}{\tau}} \nonumber \\
&+ \frac{4.3\sdeg\alpha_1\left(1+\left(\log{\tau_0}\right)^{-\frac{\alpha}{2\alpha_1}}\right)}
{\alpha\left(2\alpha_1-\alpha\right)\left(\log{\tau}\right)^{\left(2-\frac{\alpha}{\alpha_1}\right)\sigma}\log{\log{\tau}}}
+ \frac{m_{\cL}\left(1+e^{\frac{\alpha}{2}}\right)}{2\alpha}\left(\frac{\sq}{\tau}\right)^{\frac{2}{\sdeg}}
\frac{\left(\log{\tau}\right)^{2-2\sigma}}{\log{\log{\tau}}}.
\end{flalign}
Here, $\widehat{A}\left(m,\alpha,\sigma\right)$ and $\widetilde{A}\left(m,\alpha,\sigma,\tau\right)$ are defined by~\eqref{eq:WidehatA} and~\eqref{eq:TildeA}, respectively, and the functions $\mathfrak{a}=\mathfrak{a}\left(m_{\cL},\alpha_1,t_0\right)$ and $\mathfrak{b}=\mathfrak{b}\left(\sdeg,m,m_{\cL},\alpha_1,t_0,\tau_0\right)$ are defined by~\eqref{eq:Funa} and~\eqref{eq:funb}, respectively. Additionally, $\nu_1>0$ and $\nu_2>1$, and $\eta(\alpha,\sigma,\tau)$ and $\theta_1(u)$ are defined by~\eqref{eq:eta} and~\eqref{eq:thetas}, respectively.

Under the conditions from~(2) we have
\begin{flalign}
\label{eq:2ndCase}
\left|\frac{\cL'}{\cL}(s)\right| &\leq 2m\log{\log{\tau}} + m\left(1-\sigma\log{2}\right) + \frac{e^{\alpha}+1}{2\alpha} \nonumber \\
&+\left(4m\theta_{1}(M)+\frac{\left(e^{\alpha}+1\right)\theta_2(M)}{\alpha\log{\log{\tau_0}}}+\frac{m\sigma(\log{2})^2\theta_1\left(\frac{M\log{2}}{2\log{\log{\tau_0}}}\right)}{\left(\log{\log{\tau_0}}\right)^2}
\right)|1-\sigma|\left(\log{\log{\tau}}\right)^{2} \nonumber \\
&+ \frac{m\sigma 2^{\frac{3}{2}-\sigma}\left(4+\left(2+(2\sigma-1)\log{2}\right)^{2}\right)}{8\pi(2\sigma-1)^{3}} + \left(\frac{m\left(2\left(\alpha+1\right)\sigma-1\right)\left(e^{\alpha}-1\right)}{2\pi\alpha}\right)
\dfrac{(\log\log{\tau})^2}{(2\sigma-1)(\log{\tau})^{2\sigma-1}} \nonumber \\
&+ \left(\frac{2m\left(e^{\alpha}+1\right)}{\alpha}\right)
\frac{\log{\log{\tau}}}{\left(\log{\tau}\right)^{2\sigma-1}} + \frac{e^{\alpha}+1}{\alpha}\left(\frac{\mathfrak{a}_1}{\left(\log{\tau}\right)^{4\sigma-3}}+\frac{\mathfrak{b}_1}{\left(\log{\tau}\right)^{2\sigma-1}}\right)
\nonumber \\
&+ \frac{4.3\sdeg e^{2\alpha_2}\left(\frac{1}{2}+\frac{\alpha_3}{\log{\log{\tau_0}}}\right)
\left(1+\exp{\left(\frac{2\alpha\left(\log{\log{\tau_0}}-\alpha_2\right)}{\log{\log{\tau_0}}-2\alpha_2}\right)}\right)}{\alpha\left(\log{\tau}\right)^2}
\nonumber \\
&+ \frac{m_{\cL}e^{2\alpha_2}}{\alpha}\left(\frac{1}{2}+\frac{\alpha_3}{\log{\log{\tau_0}}}\right)
\left(1+\exp{\left(\frac{2\alpha\alpha_3}{\log{\log{\tau_0}}+2\alpha_3}\right)}\right)\left(\frac{\sq}{\tau}\right)^{\frac{2}{\sdeg}}
\end{flalign}
and
\begin{flalign}
\label{eq:2ndCaseB}
\left|\log{\cL(s)}\right| &\leq m\log{\left(2\log{\log{\tau}}\right)} + m\left(\frac{1}{\log{2}}-\log{\log{2}}-\sigma+\frac{1}{1-\frac{\alpha}{\log{\log{\tau_0}}-2\alpha_2}}\right) \nonumber \\
&+ \left(\frac{2m\theta_1(M)}{1-\frac{\alpha}{\log{\log{\tau_0}}-2\alpha_2}}+2m\theta_1(2\alpha_2)+\frac{m\sigma\left(\log{2}\right)\theta_1\left(\frac{M\log{2}}{2\log{\log{\tau_0}}}\right)}{\log{\log{\tau_0}}}\right)\left|1-\sigma\right|\log{\log{\tau}} \nonumber \\
&+ \frac{m2^{\frac{3}{2}-\sigma}\left(-1+\sigma\left(4-\log{2}\right)+\sigma^{2}\log{4}\right)}{8\pi\left(2\sigma-1\right)^{2}} + \frac{m\left(e^{\alpha}-1\right)\log{\log{\tau}}}{4\pi\alpha\left(\log{\tau}\right)^{2\sigma-1}} + \frac{\left(e^{\alpha}+1\right)\left(\log{\tau}\right)^{2-2\sigma}}{4\alpha\log{\log{\tau}}} \nonumber \\
&+ \frac{e^{\alpha}+1}{4\alpha\log{\log{\tau}}}\left(\frac{\mathfrak{a}_{1}}{\left(\log{\tau}\right)^{4\sigma-3}}+\frac{2\max\{0,\mathfrak{b}_1\}}{\left(\log{\tau}\right)^{2\sigma-1}}\right) + \frac{m\left(e^{\alpha}+1\right)}{\alpha\left(\log{\tau}\right)^{2\sigma-1}} \nonumber \\
&+ \frac{m\exp{\left(\frac{\alpha\log{\log{\tau_0}}}{\log{\log{\tau_0}}-2\alpha_2}\right)}}{\left(1-\frac{\alpha}{\log{\log{\tau_0}}-2\alpha_2}\right)\log{\tau}\log{\log{\tau}}} + \frac{5m\exp{\left(\frac{2\alpha\log{\log{\tau_0}}}{\log{\log{\tau_0}}-2\alpha_2}\right)}\left(1+2\log{\log{\tau}}\right)}{16\pi\left(\log{\tau}\right)^{2}} \nonumber \\
&+ \frac{4.3\sdeg\left(1+\exp{\left(\frac{2\alpha\left(\log{\log{\tau_0}}-\alpha_2\right)}{\log{\log{\tau_0}}-2\alpha_2}\right)}\right)}{2\alpha\left(\log{\tau}\right)^{2\sigma}\log{\log{\tau}}} + \frac{m_{\cL}e^{2\alpha_2}\left(1+e^{\frac{\alpha}{2}}\right)}{2\alpha\log{\log{\tau}}}\left(\frac{\sq}{\tau}\right)^{\frac{2}{\sdeg}}.
\end{flalign}
Here,
\begin{equation}
\label{eq:thetas}
M\de 2\max\left\{\alpha_2,\alpha_3\right\}, \quad \textrm{and} \quad
\theta_1(u) \de \frac{e^{u}-u-1}{u^2}, \quad \theta_2(u) \de \frac{e^{u}-1}{u}
\end{equation}
for $u>0$. Additionally, $\mathfrak{a}_1\de \mathfrak{a}\left(m_{\cL},\frac{1}{2}\log{\log{\tau_0}}-\alpha_2,t_0\right)$ and $\mathfrak{b}_1\de\mathfrak{b}\left(\sdeg,m,m_{\cL},\frac{1}{2}\log{\log{\tau_0}}-\alpha_2,t_0,\tau_0\right)$ with $\mathfrak{a}$ and $\mathfrak{b}$ defined by~\eqref{eq:Funa} and~\eqref{eq:funb}, respectively.

Under the conditions from~(3) we have
\begin{equation}
\label{eq:3rdCase}
\left|\frac{\cL'}{\cL}(s)\right| \leq \frac{m}{\alpha_3}\log{\log{\tau}}, \quad \left|\log{\cL(s)}\right| \leq m\log{\left(\frac{1}{\alpha_3}\log{\log{\tau}}\right)} + \frac{m\gamma\alpha_3}{\log{\log{\tau}}},
\end{equation}
where $\gamma$ is the Euler--Mascheroni constant.
\end{theorem}

The following theorem provides a simplified estimate in the case $\sigma=1$ when $|t|$ is large enough.

\begin{theorem}
\label{thm:1LineExplicit}
Let $\cL\in\cSP$ and assume the Generalized Riemann Hypothesis for $\cL$, the Riemann Hypothesis for $\zeta(s)$ and the strong $\lambda$-conjecture. If $\alpha\geq\log{2}$ and
\[
\tau \geq \tau_0 \geq \exp{\left(e^{\alpha}\sqrt{60}\right)}, \quad
|t| \ge t_0\ge \max\left\{2\max_{1\leq j\leq f}\left\{\left|\mu_j\right|\right\},1\right\},
\]
then
\begin{flalign}
\label{eq:LogDer1LineExp}
\left|\frac{\cL'}{\cL}(1+\ie t)\right| &\leq 2m\log{\log{\tau}} - m\left(\gamma+\alpha\right) + \frac{e^{\alpha}+1}{2\alpha} + \left(\frac{m\left(e^{\alpha}-1\right)(2\alpha+1)}{2\pi\alpha}\right)\frac{\left(\log{\log{\tau}}\right)^2}{\log{\tau}} \nonumber \\
&+ \left(\frac{2m\left(e^{\alpha}+1\right)}{\alpha}\right)\frac{\log{\log{\tau}}}{\log{\tau}}
+ \left(0.24me^{\alpha}+\frac{1}{\alpha}\left(e^{\alpha}+1\right)
\left(\mathfrak{a}_2+\mathfrak{b}_2\right)\right)\frac{1}{\log{\tau}} \nonumber \\
&+\frac{2.15\sdeg\left(e^{2\alpha}+1\right)}{\alpha\log^{2}{\tau}} + \frac{m_{\cL}}{\alpha}\left(\frac{\sq}{\tau}\right)^{\frac{2}{\sdeg}}
\end{flalign}
and
\begin{flalign}
\label{eq:Log1LineExp}
\left|\log{\cL(1+\ie t)}\right| &\leq m\log{\left(2\log{\log{\tau}}\right)} + m\gamma + \frac{e^{\alpha}+1}{4\alpha\log{\log{\tau}}} 
+ \left(\frac{m\left(1+e^{\alpha}\left(4\alpha-1\right)\right)}{4\pi\alpha}\right)\frac{\log{\log{\tau}}}{\log{\tau}} \nonumber \\
&+ \frac{m}{\alpha}\left(e^{\alpha}+1+\frac{e^{\alpha}\left(7\alpha-4\left(\alpha^2+1\right)\right)+4}{4\pi}\right)\frac{1}{\log{\tau}} \nonumber \\
&+ \left(\frac{\left(e^{\alpha}+1\right)\left(\mathfrak{a}_2+2\max\{0,\mathfrak{b}_2\}\right)}{4\alpha}+\frac{me^{\alpha}\log{\log{\tau_0}}}{\log{\log{\tau_0}}-\alpha}\right)\frac{1}{\log{\tau}\log{\log{\tau}}} \nonumber \\
&+ \frac{5me^{2\alpha}\log{\log{\tau}}}{8\pi\log^{2}{\tau}}
+ \frac{5me^{2\alpha}\left(1-2\alpha\right)}{16\pi\log^{2}{\tau}}
+ \frac{2.15\sdeg\left(e^{2\alpha}+1\right)}{\alpha\left(\log{\tau}\right)^{2}\log{\log{\tau}}} + \frac{m_{\cL}\left(1+e^{\frac{\alpha}{2}}\right)}{2\alpha\log{\log{\tau}}}\left(\frac{\sq}{\tau}\right)^{\frac{2}{\sdeg}},
\end{flalign}
where $\mathfrak{a}_2\de \mathfrak{a}\left(m_{\cL},\frac{1}{2}\log{\log{\tau_0}},t_0\right)$ and $\mathfrak{b}_2\de\mathfrak{b}\left(\sdeg,m,m_{\cL},\frac{1}{2}\log{\log{\tau_0}},t_0,\tau_0\right)$ with $\mathfrak{a}$ and $\mathfrak{b}$ defined by~\eqref{eq:Funa} and~\eqref{eq:funb}, respectively, and $\gamma$ is the Euler--Mascheroni constant.
\end{theorem}

\begin{remark}
\label{rmk:sigma1Explicit}
Let $\cL\in\cSP$ be entire and assume the Generalized Riemann Hypothesis for $\cL$, the Riemann Hypothesis for $\zeta(s)$ and the strong $\lambda$-conjecture. Assume also that $\alpha\geq\log{2}$ and $\sq \geq {\sq}_0 \geq \exp{\left(e^{\alpha}\sqrt{60}\right)}$. Then~\eqref{eq:LogDer1LineExp} and~\eqref{eq:Log1LineExp} are valid for $t=0$ after taking $m_{\cL}=0$, and replacing $\tau$ and $\tau_0$ with $\sq$ and ${\sq}_0$, respectively, and also after replacing $\mathfrak{a}_2$ and $\mathfrak{b}_2$ with $\mathfrak{a}_3:=\mathfrak{a}\left(0,\frac{1}{2}\log\log{{\sq}_0},1\right)$ and $\mathfrak{b}_3:=\mathfrak{b}_3\left(\sdeg,m,{\sq}_0\right)$, that are defined by~\eqref{eq:Funa} and~\eqref{def:b3}, respectively. 
\end{remark}

\subsection{Some simplified results}
\label{subsec:simplfied}

In this section we are stating two simplified explicit estimates that follow from the first part of Theorem~\ref{thm:MainExplicit} and from Theorem~\ref{thm:1LineExplicit}. The first one is for the Riemann zeta-function and $\sigma\in[1/2+1/\log{\log{t}},1)$, while the second one is for a large family of $L$-functions with $\sigma\in[0.6,1)$.

\begin{corollary}
\label{cor:ZetaExplicit}
Let $t\geq 10^6$, $1/2+1/\log{\log{t}}\leq\sigma<1$, $\alpha=1.278$, and assume the Riemann Hypothesis. Then
\begin{equation}
\label{eq:LogDerZetaExplicit}
\left|\frac{\zeta'}{\zeta}(s)\right| \leq \widehat{A}\left(1,\alpha,\sigma\right)\left(\log{t}\right)^{2-2\sigma} - \frac{\sigma 2^{1-\sigma}}{1-\sigma} + 4.2\left(\log{t}\right)^{3-4\sigma} + \frac{0.64}{(2\sigma-1)^{3}} + \frac{4\left(\log{\log{t}}\right)^{2}}{(2\sigma-1)\left(\log{t}\right)^{2\sigma-1}} + \frac{2}{\left(\log{t}\right)^{\frac{1}{3}}}
\end{equation}
and
\begin{flalign}
\label{eq:LogZetaExplicit}
\left|\log{\zeta(s)}\right| &\leq \widetilde{A}\left(1,\alpha,\sigma,t\right)\frac{\left(\log{t}\right)^{2-2\sigma}}{\log{\log{t}}} - \frac{\eta(\alpha,\sigma,t)}{(1-\sigma)\log{\log{t}}} + \log{\left(2\log{\log{t}}\right)} + \frac{3.24\left(\log{t}\right)^{2-2\sigma}}{(1-\sigma)^{2}\left(\log{\log{t}}\right)^{2}} \nonumber \\
&+ \frac{1.04\left(\log{t}\right)^{3-4\sigma}}{\log{\log{t}}} + \frac{4.7}{(2\sigma-1)^{2}} + \frac{1.53\log{\log{t}}}{\left(\log{t}\right)^{2\sigma-1}} + \frac{7.75}{\left(\log{t}\right)^{\frac{1}{3}}\log{\log{t}}} + \frac{0.21\log{\log{t}}}{\left(\log{t}\right)^{\frac{2}{3}}},
\end{flalign}
where functions $\widehat{A}(1,\alpha,\sigma)$, $\widetilde{A}(1,\alpha,\sigma,t)$ and $\eta(\alpha,\sigma,t)$ are defined by~\eqref{eq:WidehatA},~\eqref{eq:TildeA} and~\eqref{eq:eta}, respectively.
\end{corollary}

\begin{corollary}
\label{corol:FixedRegions}

Let $\cL\in\cSP$ and assume the Generalized Riemann Hypothesis for $\cL$, the Riemann Hypothesis for $\zeta(s)$ and the strong $\lambda$-conjecture. In addition, assume also that $m=\sdeg$. Then, for $0.6\leq\sigma<1$ and
\[
\tau\geq \exp{\left(\exp{(13)}\right)}, \quad |t|\geq \max\left\{2\max_{1\leq j\leq f}\left\{\left|\mu_j\right|\right\},m_{\cL}+2\cdot10^3\right\},
\]
we have
\begin{equation}
\label{eq:VerySimplifiedLogDer4}
\left|\frac{\cL'}{\cL}(s)\right| \leq \sdeg\left(\frac{\left(\log{\tau}\right)^{2-2\sigma}-\sigma2^{1-\sigma}}{1-\sigma}\right) + \left(1.796-1.278\sdeg\right)\left(\log{\tau}\right)^{2-2\sigma} + 4\left(\log{\tau}\right)^{3-4\sigma} + 90\sdeg
\end{equation}
and
\begin{multline}
\label{eq:Very2ndCaseSimpler}
\left|\log{\cL(s)}\right| \leq \sdeg\frac{\left(\log{\tau}\right)^{2-2\sigma}-1}{2(1-\sigma)\log{\log{\tau}}} + \left(0.898-0.639\sdeg\right)\frac{\left(\log{\tau}\right)^{2-2\sigma}}{\log{\log{\tau}}} \\ 
+ \sdeg\left(\log{\log{\log{\tau}}}+8\right) + 3.4\sdeg\frac{\left(\log{\tau}\right)^{2-2\sigma}}{(1-\sigma)^{2}\left(\log{\log{\tau}}\right)^{2}}.
\end{multline}
Moreover,
\begin{equation}
\label{eq:VerySimplifiedLogDer42}
\left|\frac{\cL'}{\cL}\left(1+\ie t\right)\right| \leq 2\sdeg\log{\log{\tau}} + 2.265 - 2.763\sdeg + 4\sdeg\frac{\left(\log{\log{\tau}}\right)^{2}}{\log{\tau}}
\end{equation}
and
\begin{equation}
\label{eq:VerySimplifiedLog42}
\left|\log{\cL\left(1+\ie t\right)}\right| \leq \sdeg\log{\log{\log{\tau}}} + \sdeg\log{\left(2e^{\gamma}\right)} + \frac{\sdeg}{\log{\log{\tau}}}.
\end{equation}
\end{corollary}

\section{The Selberg moment formula for functions in $\mathcal{S}$}
\label{sec:SMF}

Selberg~\cite[Lemma 2]{SelbergOnTheNormal} discovered an interesting connection (also known as the Selberg moment formula) between the logarithmic derivative of the Riemann zeta-function and special truncated Dirichlet series. In this paper we follow the same idea. 
Hence, the following lemma generalizes Selberg's result to functions in the Selberg class and follows from using similar contour integral as in \cite[the beginning of Section 13.2]{MontgomeryVaughan}.

\begin{lemma}
\label{lem:SelbergMomentFormula}
Assume that $\cL\in\cS$, $\cL(1)\neq 0$ and $s\neq 1$ or $m_{\cL}=0$. Let $q_j(k)\de\left(k+\mu_j\right)/\lambda_j$ for $j\in\{1,\ldots,f\}$ and $k\in\N_0$. Take $x\geq 2$ and $y\geq 2$. Then
\begin{flalign}
\label{eq:selbergMomentFormula}
\frac{\cL'}{\cL}(s) = -\sum_{n\leq xy}\frac{\Lambda_{\cL,x,y}(n)}{n^s} &+ m_{\cL}\frac{(xy)^{1-s}-x^{1-s}}{(1-s)^2\log{y}} \nonumber \\
&+ \frac{1}{\log{y}}\sum_{j=1}^{f}\sum_{k=0}^{\infty}\frac{x^{-q_j(k)-s}-(xy)^{-q_j(k)-s}}{\left(q_j(k)+s\right)^2} + \frac{1}{\log{y}}\sum_{\rho}\frac{x^{\rho-s}-(xy)^{\rho-s}}{(\rho-s)^2}
\end{flalign}
for $s\notin\left\{-q_j(k)\colon k\in\N_0, 1\leq j\leq f\right\}\cup\{1\}$ and $s\neq\rho$, where
\[
	\Lambda_{\cL,x,y}(n) \de \left\{\begin{array}{ll}
		\Lambda_{\cL}(n), & 1\leq n\leq x, \\
		\Lambda_{\cL}(n)\frac{\log{\left(xy/n\right)}}{\log{y}}, & x<n\leq xy.
	\end{array}
	\right.
\]
\end{lemma}

Let $\sigma\in(1/2,3/2]$ and $\alpha\geq\log{2}$, and take
\begin{equation}
\label{eq:xy}
y = \exp{\left(\frac{\alpha}{\sigma-\frac{1}{2}}\right)}, \quad x = y^{-1}\log^{2}{\tau}.
\end{equation}
It is not hard to verify that $x\geq 2$ and $y\geq 2$ under the conditions~(1) and~(2) from Theorem~\ref{thm:MainExplicit}, and also under conditions from Theorem~\ref{thm:1LineExplicit}. Therefore, we can use Lemma~\ref{lem:SelbergMomentFormula} for $x$ and $y$ which are defined by~\eqref{eq:xy} in order to prove these theorems. Moreover, because these conditions are independent on the assumptions of Conjectures~\ref{conj:SelbergVariant} and~\ref{conj:SelbergVariant2}, or the strong $\lambda$-conjecture or a polynomial Euler product, we can use the same approach also to prove Theorems~\ref{thm:MainGeneral}--\ref{thm:1line}.

Let us bound the terms on the right-hand side of~\eqref{eq:selbergMomentFormula}. The first term is estimated easily by
\begin{equation}
\label{eq:S}
\left|\sum_{n\leq xy}\frac{\Lambda_{\cL,x,y}(n)}{n^s}\right|\leq \sum_{n\leq x}\frac{\left|\Lambda_{\cL}(n)\right|}{n^{\sigma}} + \frac{1}{\log{y}}\sum_{x<n\leq xy}\frac{\left|\Lambda_{\cL}(n)\right|\log{\frac{xy}{n}}}{n^{\sigma}} =:S_{\cL,x,y}(\sigma).
\end{equation}
Under GRH we can estimate the sum over zeros as
\begin{flalign}
\label{eq:SumOverZeros}
\frac{1}{\log{y}}\left|\sum_{\rho}\frac{x^{\rho-s}-(xy)^{\rho-s}}{\left(\rho-s\right)^{2}}\right|
&\leq \frac{(xy)^{\frac{1}{2}-\sigma}\left(1+y^{\sigma-\frac{1}{2}}\right)}{\log{y}}\sum_{\rho}\frac{1}{\left|\rho-s\right|^{2}} \nonumber \\
&\leq \frac{e^{\alpha}+1}{\alpha}\left(\log{\tau}\right)^{1-2\sigma}\sum_{\gamma}\frac{\sigma-\frac{1}{2}}{\left(\sigma-\frac{1}{2}\right)^2+(t-\gamma)^2}
\ed \mathcal{Z}(\sigma).
\end{flalign}
We bound the remaining two terms as
\begin{flalign}
\label{eq:sum2}
\frac{1}{\log{y}}\left|\sum_{j=1}^{f}\sum_{k=0}^{\infty}\frac{x^{-q_j(k)-s}-(xy)^{-q_j(k)-s}}{\left(q_j(k)+s\right)^2}\right|
&\leq
\frac{(xy)^{-\sigma}\left(1+y^{\sigma}\right)}{\log{y}}\sum_{j=1}^{f}\sum_{k=0}^{\infty}\left(\frac{k}{\lambda_j}+\frac{1}{2}\right)^{-2} \nonumber \\
&\leq \frac{\left(\sigma-\frac{1}{2}\right)\left(1+\exp{\left(\frac{2\alpha\sigma}{2\sigma-1}\right)}\right)}{\alpha\left(\log{\tau}\right)^{2\sigma}}
\sum_{j=1}^{f}\lambda_j^2\left(\frac{\Gamma'}{\Gamma}\right)'\left(\frac{\lambda_j}{2}\right) \ed \mathcal{R}_1(\sigma)
\end{flalign}
and
\begin{equation}
\label{eq:pole}
m_{\cL}\left|\frac{(xy)^{1-s}-x^{1-s}}{(1-s)^2\log{y}}\right| \leq \frac{m_{\cL}\left(\sigma-\frac{1}{2}\right)\left(1+\exp{\left(\frac{2\alpha(\sigma-1)}{2\sigma-1}\right)}\right)}{\alpha}\left(\frac{\sq}{\tau}\right)^{\frac{2}{\sdeg}}
\left(\log{\tau}\right)^{2-2\sigma} \ed \mathcal{R}_2(\sigma).
\end{equation}
An obvious strategy to produce the desired bounds for $\cL'/\cL$ is to employ estimates~\eqref{eq:S}--\eqref{eq:pole} into the Selberg moment formula~\eqref{eq:selbergMomentFormula}, that is
\begin{equation}
\label{eq:LogDerEffective}
\left|\frac{\cL'}{\cL}(s)\right| \leq S_{\cL,x,y}(\sigma) + \mathcal{Z}(\sigma) + \mathcal{R}_1(\sigma) + \mathcal{R}_2(\sigma).
\end{equation}
However, in order to do so we need to obtain suitable bounds for $S_{\cL,x,y}(\sigma)$ and $\mathcal{Z}(\sigma)$. We are going to do this in the following two sections. At this stage it is not hard to see how to estimate $\log{\cL}$ with the help of $\cL'/\cL$. Under GRH we can write
\[
\log{\cL(s)} = \log{\cL\left(\frac{3}{2}+\ie t\right)} - \int_{\sigma}^{\frac{3}{2}}\frac{\cL'}{\cL}\left(\sigma'+\ie t\right)\dif{\sigma'}.
\]
Because then
\[
\int_{\sigma}^{\frac{3}{2}}\left|\frac{\cL'}{\cL}\left(\sigma'+\ie t\right)\right|\dif{\sigma'} \leq \sum_{n\leq xy}\frac{\left|\Lambda_{\cL,x,y}(n)\right|}{n^{\sigma}\log{n}} - \sum_{n\leq x}\frac{\left|\Lambda_{\cL}(n)\right|}{n^{\frac{3}{2}}\log{n}}
+\int_{\sigma}^{\frac{3}{2}}\left(\mathcal{Z}+\mathcal{R}_1+\mathcal{R}_2\right)\left(\sigma'\right)\dif{\sigma'}
\]
by Lemma~\ref{lem:SelbergMomentFormula}, we obtain
\begin{equation}
\label{eq:mainForLog}
\left|\log{\cL(s)}\right| \leq \widehat{S}_{\cL,x,y}(\sigma)
+ \mathcal{E}_{x} + \int_{\sigma}^{\frac{3}{2}}\left(\mathcal{Z}+\mathcal{R}_1+\mathcal{R}_2\right)\left(\sigma'\right)\dif{\sigma'},
\end{equation}
where
\begin{equation}
\label{eq:SHatEx}
\widehat{S}_{\cL,x,y}(\sigma) \de \sum_{n\leq x}\frac{\left|\Lambda_{\cL}(n)\right|}{n^{\sigma}\log{n}}
+ \frac{1}{\log{y}}\sum_{x<n\leq xy}\frac{\left|\Lambda_{\cL}(n)\right|\log{\frac{xy}{n}}}{n^{\sigma}\log{n}}, \quad
\mathcal{E}_{x} \de \sum_{n>x}\frac{\left|\Lambda_{\cL}(n)\right|}{n^{\frac{3}{2}}\log{n}}.
\end{equation}
We are using~\eqref{eq:mainForLog} for the estimation of $\log{\cL}$.

\begin{remark}
The above methods work when $|t|$ is large enough. Since we want to derive upper bounds also in the case $s= 1$ when $\cL$ is entire, and thus we do slightly different substitutions than~\eqref{eq:xy}. Let us set
\begin{equation}
\label{eq:xySeond}
y = \exp{\left(\frac{\alpha}{\sigma-\frac{1}{2}}\right)}, \quad x = y^{-1}\log^{2}{\sq},
\end{equation}
where $\alpha \geq \log{2}$ if $1/2<\sigma \leq 3/2$. Note that $x\geq 2$ and $y \geq 2$ if $\sq \geq \exp\left(\sqrt{2}e^{\alpha}\right)$. With these substitutions, we have $\log{\sq}$ instead of $\log{\tau}$ in~\eqref{eq:SumOverZeros} and~\eqref{eq:sum2}, and we denote the new terms by $\mathcal{Z}_1(\sigma)$ and $\mathcal{R}_{1,1}(\sigma)$, respectively. 
Thus
\begin{equation}
\label{eq:oneFormulaLogDer}
    \left|\frac{\cL'}{\cL}(1)\right| \leq S_{\cL,x,y}(1) + \mathcal{Z}_1(1) + \mathcal{R}_{1,1}(1).
\end{equation}
We estimate the term $\log{\cL (s)}$ similarly as in the case~\eqref{eq:mainForLog}, but with the same replacements as in \eqref{eq:oneFormulaLogDer} and $\mathcal{R}_{2}(1)=0$ since $\cL$ is entire.
\end{remark}

In the next sections, we are going to estimate the terms described above. Because our treatment of $\widehat{S}_{\cL,x,y}(\sigma)$ and $\mathcal{E}_{x}$ is similar to that of $S_{\cL,x,y}(\sigma)$, we will do this in the next section.

\section{Various sums over prime numbers}
\label{sec:SumsPrimes}
In this section we obtain several upper bounds for $S_{\cL,x,y}(\sigma)$, $\widehat{S}_{\cL,x,y}(\sigma)$ and $\mathcal{E}_{x}$ coming from the Selberg moment formula, see~\eqref{eq:S} and~\eqref{eq:SHatEx} for definitions, according to whether $\cL\in\mathcal{S}$ or $\cL\in\mathcal{SP}$.

\subsection{The case when $\cL\in\cS$}
\label{sec:PrimesForS}

Firstly, we will show that~\eqref{eq:PsiTilde} implies~\eqref{eq:MainForPsi}. By the Euler product representation (Axiom 4),
\[
\sum_{k=2}^{\lfloor\frac{\log{x}}{\log{2}}\rfloor}\sum_{p\leq x^{\frac{1}{k}}}\left|\Lambda_{\cL}\left(p^k\right)\right| \leq \mathcal{C}^{E}_{\cL} \sum_{k=2}^{\lfloor\frac{\log{x}}{\log{2}}\rfloor} k\sum_{p\leq x^{\frac{1}{k}}}p^{k\theta}\log{p}
\]
for some $\theta\in[0,1/2)$ since $\left|b\left(p^k\right)\right|\leq \mathcal{C}^{E}_{\cL} p^{k\theta}$. Trivial estimation assures that
\[
\sum_{p\leq X}p^{k\theta}\log{p} \leq X^{k\theta}\vartheta(X) \ll X^{k\theta+1}
\]
for all $X\geq 2$, where $\vartheta(X)=\sum_{p\leq X}\log{p}$. 
Therefore,
\[
\sum_{k=2}^{\lfloor\frac{\log{x}}{\log{2}}\rfloor}\sum_{p\leq x^{\frac{1}{k}}}\left|\Lambda_{\cL}\left(p^k\right)\right| \ll \mathcal{C}_{\cL}^{E} x^{\theta}\sum_{k=2}^{\lfloor\frac{\log{x}}{\log{2}}\rfloor}kx^{\frac{1}{k}} \ll \mathcal{C}_{\cL}^{E} x^{\frac{1}{2}+\theta}\log^{2}{x},
\]
which gives~\eqref{eq:MainForPsi}. Let $\varepsilon\in(0,1/2)$. 
By the quantitative version of the Ramanujan hypothesis, $\left|a(p)\right|\leq \mathcal{C}_{\cL}^{R}(\varepsilon)p^{\varepsilon}$. Therefore, because 
\[
\sum_{p\leq x}\left|a(p)\right|\log{p} \leq \mathcal{C}_{\cL}^{R}(\varepsilon)
x^{\varepsilon}\vartheta(x),
\]
RH in the equivalent form $\vartheta(x)=x+O\left(\sqrt{x}\log^{2}{x}\right)$ 
and~\eqref{eq:MainForPsi} then guarantee
\begin{equation}
\label{eq:main}
\widetilde{\psi}_{\cL}(x) \leq \mathcal{C}_{\cL}^{R}(\varepsilon)x^{1+\varepsilon} + O\left(\left(\mathcal{C}_{\cL}^{R}(\varepsilon)+\mathcal{C}_{\cL}^{E}\right)x^{\frac{1}{2}+\max\left\{\theta,\varepsilon\right\}}\log^{2}{x}\right)
\end{equation}
for $x\geq 2$, where the implied constant is uniform and $\widetilde{\psi}$ is given in \eqref{eq:PsiTilde}. Inequality~\eqref{eq:main} is used in the proof of Lemmas~\ref{lem:GeneralBoundForS} and~\ref{lem:GeneralBoundForGW}, and consequently in the proof of Theorem~\ref{thm:MainGeneral}.

Results from Section~\ref{sec:Distr} depend on Conjectures~\ref{conj:SelbergVariant} and~\ref{conj:SelbergVariant2}. By the Cauchy--Bunyakovsky--Schwarz inequality,
\[
\sum_{p\leq x}\left|a(p)\right|\log{p} \leq \left(\sum_{p\leq x}\left|a(p)\right|^{2}\right)^{1/2}\left(\sum_{p\leq x}\log^{2}{p}\right)^{1/2} \leq \sqrt{x\log{x}+O(x)}\left(\sum_{p\leq x}\left|a(p)\right|^{2}\right)^{1/2}.
\]
Therefore, Conjecture~\ref{conj:SelbergVariant} implies 
\begin{equation}
\label{eq:main2}
\widetilde{\psi}_{\cL}(x) \leq \sqrt{\mathcal{C}_{\cL}^{P_1}(x)}x + O\left(\sqrt{\mathcal{C}_{\cL}^{P_1}(x)+\mathcal{C}_{\cL}^{P_2}}\frac{x}{\sqrt{\log{x}}}\right) + O\left(\mathcal{C}_{\cL}^{E} x^{\frac{1}{2}+\theta}\log^{2}{x}\right)
\end{equation}
for all $x\geq 2$. 
Conjecture~\ref{conj:SelbergVariant2} implies
\begin{equation}
\label{eq:main3}
\widetilde{\psi}_{\cL}(x) \leq \widehat{\mathcal{C}_{\cL}^{P_1}}(x)x + \widehat{\mathcal{C}_{\cL}^{P_2}}\frac{x}{\log{x}} + O\left(\mathcal{C}_{\cL}^{E} x^{\frac{1}{2}+\theta}\log^{2}{x}\right)
\end{equation}
for all $x\geq 2$ by trivial estimation. The implied constants in both~\eqref{eq:main2} and~\eqref{eq:main3} are absolute. These two inequalities are used in the proof of Lemmas~\ref{lem:GeneralBoundForSV2}--\ref{lem:GeneralBoundForGW}, and consequently in the proof of Theorems~\ref{thm:MainGeneralV2} and~\ref{thm:MainGeneral1line}.

Let $x\geq2$ and $y\geq2$. Assume that we can write $\widetilde{\psi}_{\cL}(u)\leq f(u) + g(u)$ for $u\in[2,xy]$, where $f$ is continuously differentiable on $(1,\infty)$ and $g$ is integrable on $[2,\infty)$. Let $\sigma\geq0$. Then 
\begin{multline}
\label{eq:ineqForS}
S_{\cL,x,y}(\sigma) \leq \int_{2}^{x}\frac{f'(u)}{u^\sigma}\dif{u} + \frac{1}{\log{y}}\int_{x}^{xy}\frac{f'(u)\log{\frac{xy}{u}}}{u^{\sigma}}\dif{u} + \frac{f(2)}{2^{\sigma}} \\ 
+ \sigma\int_{2}^{x}\frac{g(u)}{u^{1+\sigma}}\dif{u} + \frac{1}{\log{y}}\int_{x}^{xy}\frac{\left(1+\sigma\log{\frac{xy}{u}}\right)g(u)}{u^{1+\sigma}}\dif{u}
\end{multline}
and
\begin{multline}
\label{eq:ineqForShat}
\widehat{S}_{\cL,x,y}(\sigma) \leq \int_{2}^{x}\frac{f'(u)}{u^{\sigma}\log{u}}\dif{u} + \frac{1}{\log{y}}\int_{x}^{xy}\frac{f'(u)\log{\frac{xy}{u}}}{u^{\sigma}\log{u}}\dif{u} + \frac{f(2)}{2^{\sigma}\log{2}} \\
+ \int_{2}^{x}\frac{\left(1+\sigma\log{u}\right)g(u)}{u^{1+\sigma}\log^{2}{u}}\dif{u} + 
\frac{1}{\log{y}}\int_{x}^{xy}\frac{\left(\left(1+\sigma\log{u}\right)\log{\frac{xy}{u}}+\log{u}\right)g(u)}{u^{1+\sigma}\log^{2}{u}}\dif{u}
\end{multline}
by employing partial summation. If $\widetilde{\psi}_{\cL}(u)\leq f(u) + g(u)$ holds for $u\geq x$, where $f$ and $g$ are continuous on $[x,\infty)$, then
\begin{equation}
\label{eq:ineqForEx}
\mathcal{E}_x \leq \int_{x}^{\infty}\frac{\left(f(u)+g(u)\right)\left(1+\frac{3}{2}\log{u}\right)}{u^{5/2}\log^{2}{u}}\dif{u}
\end{equation}
again by partial summation. We are going to use the above inequalities with $x$ and $y$ as in~\eqref{eq:xy}. 

\begin{lemma}
\label{lem:GeneralBoundForS}
Let $\cL\in\cS$ and assume the Riemann Hypothesis. Fix $\alpha\geq\log{2}$ and $\delta\in(0,1/2)$. Then
\begin{multline}
\label{eq:GeneralBoundForS_1}
S_{\cL,x,y}(\sigma) \leq A\left((1+\varepsilon)\mathcal{C}_{\cL}^{R}(\varepsilon),\alpha,\sigma-\varepsilon,\sigma\right)\left(\log{\tau}\right)^{2(1-\sigma+\varepsilon)} 
- \frac{\left(1+\varepsilon\right)\mathcal{C}_{\cL}^{R}(\varepsilon)\sigma 2^{1-\sigma+\varepsilon}}{1-\sigma+\varepsilon} \\
+ O\left(A_1\left(\mathcal{C}_{\cL}^{R}(\varepsilon)+\mathcal{C}_{\cL}^{E},2,\varepsilon,\theta,\sigma,\tau\right)\right),
\end{multline}
\begin{multline}
\label{eq:GeneralBoundForS_2}
\widehat{S}_{\cL,x,y}(\sigma) \leq \eta(\alpha,\sigma,\tau)A\left((1+\varepsilon)\mathcal{C}_{\cL}^{R}(\varepsilon),\alpha,\sigma-\varepsilon,\sigma\right)\frac{\left(\log{\tau}\right)^{2(1-\sigma+\varepsilon)}}{\log{\log{\tau}}} - \frac{(1+\varepsilon)\mathcal{C}_{\cL}^{R}(\varepsilon)\eta(\alpha,\sigma,\tau)}{(1-\sigma+\varepsilon)\log{\log{\tau}}} \\
+ \mathcal{C}_{\cL}^{R}(\varepsilon)\log{\left(2\log{\log{\tau}}\right)}
+O\left(\frac{\mathcal{C}_{\cL}^{R}(\varepsilon)\left(\log{\tau}\right)^{2(1-\sigma+\varepsilon)}}{(1-\sigma+\varepsilon)^{2}\left(\log{\log{\tau}}\right)^2}\right) + O\left(A_1\left(\mathcal{C}_{\cL}^{R}(\varepsilon)+\mathcal{C}_{\cL}^{E},1,\varepsilon,\theta,\sigma,\tau\right)\right) 
\end{multline}
and
\begin{equation}
\label{eq:GeneralBoundForS_3}
\mathcal{E}_{x} \ll 
A_3\left(\frac{\mathcal{C}_{\cL}^{R}(\varepsilon)}{1-2\varepsilon},\mathcal{C}_{\cL}^{R}(\varepsilon)+\mathcal{C}_{\cL}^{E},\varepsilon,\theta,\tau\right)
\end{equation}
for $\varepsilon\in(0,1/2)$, $1/2+\delta\leq\sigma<1$ and sufficiently large $\tau$. Here, $A\left(a,\alpha,u,\sigma\right)$ and $\eta(\alpha,\sigma,\tau)$ are defined by~\eqref{eq:A} and~\eqref{eq:eta}, respectively, and $A_1(a,k,\varepsilon,\theta,\sigma,\tau)$ and $A_3(a,b,\varepsilon,\theta,\tau)$ are defined by~\eqref{eq:MainGeneralA_1} and~\eqref{eq:MainGeneralA_3}, respectively.
\end{lemma}

\begin{proof}
Because we are using~\eqref{eq:main}, we are taking
\[
f(u) = \mathcal{C}_{\cL}^{R}(\varepsilon)u^{1+\varepsilon}, \quad g(u) =  O(1)\cdot\left(\left(\mathcal{C}_{\cL}^{R}(\varepsilon)+\mathcal{C}_{\cL}^{E}\right)u^{\frac{1}{2}+\max\left\{\theta,\varepsilon\right\}}\log^{2}{u}\right)
\]
in~\eqref{eq:ineqForS},~\eqref{eq:ineqForShat} and~\eqref{eq:ineqForEx}. Note that, for sufficiently large $\tau$, conditions of Lemma~\ref{lem:GeneralBoundForS} guarantee that $x\geq 2$ and $y\geq 2$. Also, $x\asymp\log^{2}{\tau}$ with the implied constants being uniform. The sum of the first three terms on the right-hand side of~\eqref{eq:ineqForS} is
\[
\leq \left(1+\varepsilon\right)\mathcal{C}_{\cL}^{R}(\varepsilon)\frac{(xy)^{1-\sigma+\varepsilon}-x^{1-\sigma+\varepsilon}}{(1-\sigma+\varepsilon)^{2}\log{y}} - \frac{\left(1+\varepsilon\right)\mathcal{C}_{\cL}^{R}(\varepsilon)\sigma 2^{1-\sigma+\varepsilon}}{1-\sigma+\varepsilon}, 
\]
while the sum of the last last two terms of~\eqref{eq:ineqForS} is
\begin{equation*}
\ll \int_{2}^{xy}\frac{|g(u)|}{u^{1+\sigma}}\dif{u} 
\ll \left(\mathcal{C}_{\cL}^{R}(\varepsilon)+\mathcal{C}_{\cL}^{E}\right)\int_{2}^{xy}u^{\Theta_{\theta,\varepsilon}(\sigma)-1}\log^{2}{u} \dif{u}.
\end{equation*}
Here $\Theta_{\theta,\varepsilon}(\sigma)$ is given as in \eqref{def:theta}. We estimate the integral on the right-hand side of the estimate.

Note that for $b\neq 0$ (and $a\geq2$) one has
\[
\int_{2}^{a}u^{b-1}\log{u} \dif{u} = \frac{1}{b^2}\left(a^{b}\left(b\log{a}-1\right)-2^{b}\left(b\log{2}-1\right)\right).
\]
If $b=0$, the above integral is $\ll \log^2{a}$. If $b>0$, then the integral is $\ll \left(a^b/b\right)\log{a}$, and if $b<0$, then it is $ \ll 2^{b}\left(|b|\log{2}+1\right)/b^2\ll 1/b^2$ when $b$ is bounded. In both cases, which include $|b|\geq 1/\log{a}$, we thus have $\ll 1/b^2 + \left(a^b/b\right)\log{a}$. If $|b|<1/\log{a}$, then the right hand-side of the above equality shows that $\ll \log^{2}{a}$. But then also $0\leq|b|a^b\log{a}\ll 1$. Using similar argument for the case $k=2$, we can deduce that
\begin{equation}
\label{eq:integral}
\int_{2}^{a}u^{b-1}\log^{k}{u} \dif{u} \ll \left(1+|b|^{k}a^{b}\log^{k}{a}\right)\min\left\{\frac{1}{|b|^{k+1}},\log^{k+1}{a}\right\}
\end{equation}
for every $a\geq 2$ and $b\in\R$, and $k\in\{1,2\}$. Inequality~\eqref{eq:GeneralBoundForS_1} now easily follows from~\eqref{eq:ineqForS}. 

Turning to~\eqref{eq:GeneralBoundForS_2}, we first estimate the first three terms on the right-hand side of~\eqref{eq:ineqForShat}.
Following~\cite[Appendix A]{CarneiroChandee}, for real number $b$ and $a \geq 2$, we have
\begin{flalign}
0 \leq \int_{2}^{a}\frac{\dif{u}}{u^{b}\log{u}} &= \log{\log{a}} - \log{\log{2}} + \int_{(1-b)\log{2}}^{(1-b)\log{x}}\theta_2(u)\dif{u} \nonumber \\
&= \frac{a^{1-b}-1}{(1-b)\log{a}} + \log{\log{a}} - \frac{2^{1-b}-1}{(1-b)\log{2}} - \log{\log{2}} + \int_{(1-b)\log{2}}^{(1-b)\log{a}}\theta_1(u)\dif{u}, \label{eq:integral2}
\end{flalign}
where $\theta_1(u)$ and $\theta_2(u)$ are defined in~\eqref{eq:thetas}.

If $b\in[0,1)$, then for some $\nu_1>0$ and $\nu_2>1$, we have
\begin{flalign}
\label{eq:Case2}
\int_{(1-b)\log{2}}^{(1-b)\log{a}}\theta_1(u)\dif{u} &\leq \int_{0}^{(1-b)\log{a}}\theta_1(u)\dif{u} = \left(\int_{0}^{\nu_1}+\int_{\nu_1}^{\frac{(1-b)\log{a}}{\nu_2}}+\int_{\frac{(1-\sigma)\log{a}}{\nu_2}}^{(1-b)\log{a}}\right)
\theta_1(u)\dif{u} \nonumber \\
&\leq \frac{\nu_{2}^{2}a^{1-b}}{(1-b)^2\log^{2}{a}} + \frac{1}{\nu_{1}^{2}}a^{\frac{1-b}{\nu_2}} + \int_{0}^{\nu_1}\theta_1(u)\dif{u}.
\end{flalign}
Choosing $\nu_j=j$, we obtain
\begin{equation}
\label{eq:int2}
\int_{2}^{a}\frac{\dif{u}}{u^{b}\log{u}} = \frac{a^{1-b}-1}{(1-b)\log{a}} + \log{\log{a}} + O\left(\frac{a^{1-b}}{(1-b)^{2}\log^{2}{a}}\right)
\end{equation}
for $a\geq2$ and $b\in[0,1)$. 
In addition, we have
\begin{equation*}
    \int_{x}^{xy} \frac{\log{\left(\frac{xy}{u}\right)}}{u^b \log{u}} \dif{u} \leq \frac{1}{\log{x}}\int_{x}^{xy} \frac{\log{\left(\frac{xy}{u}\right)}}{u^b} \dif{u}. 
\end{equation*}
Hence, the sum of the first three terms on the right-hand side of~\eqref{eq:ineqForShat} is
\begin{equation*}
\leq 
\left(1+\varepsilon\right)\mathcal{C}_{\cL}^{R}(\varepsilon)\frac{(xy)^{1-\sigma+\varepsilon}-x^{1-\sigma+\varepsilon}}{(1-\sigma+\varepsilon)^{2}\log{x}\log{y}}
- \frac{(1+\varepsilon)\mathcal{C}_{\cL}^{R}(\varepsilon)}{(1-\sigma+\varepsilon)\log{x}}
+ \mathcal{C}_{\cL}^{R}(\varepsilon)\log{\log{x}}
+ O\left(\frac{\mathcal{C}_{\cL}^{R}(\varepsilon)x^{1-\sigma+\varepsilon}}{(1-\sigma+\varepsilon)^{2}\log^{2}{x}}\right),
\end{equation*}
while the sum of the last two terms is
\[
\ll \int_{2}^{xy}\frac{|g(u)|}{u^{1+\sigma}\log{u}}\dif{u} \ll \left(\mathcal{C}_{\cL}^{R}(\varepsilon)+\mathcal{C}_{\cL}^{E}\right)\int_{2}^{xy}u^{\Theta_{\theta,\varepsilon}(\sigma)-1}\log{u}\dif{u}.
\]
Inequality~\eqref{eq:GeneralBoundForS_2} now follows from~\eqref{eq:integral} and~\eqref{eq:ineqForShat}. 

In order to prove~\eqref{eq:GeneralBoundForS_3}, we have
\[
\mathcal{E}_x \ll \frac{\mathcal{C}_{\cL}^{R}(\varepsilon)}{\left(1-2\varepsilon\right)x^{1/2-\varepsilon}\log{x}} + \frac{\left(\mathcal{C}_{\cL}^{R}(\varepsilon)+\mathcal{C}_{\cL}^{E}\right)\log{x}}{x^{-\Theta_{\theta,\varepsilon}(3/2)}}
\]
by~\eqref{eq:ineqForEx}. Inequality~\eqref{eq:GeneralBoundForS_3} now easily follows from the above.
\end{proof}

\begin{lemma}
\label{lem:GeneralBoundForSV2}
Let $\cL\in\cS$. Fix $\alpha\geq\log{2}$ and $\delta\in(0,1/2)$. Define $m_{i}(\tau)$ for $i\in\{1,2,3,4\}$ by~\eqref{eq:m123} if Conjecture~\ref{conj:SelbergVariant} is true, and by~\eqref{eq:m123V2} if Conjecture~\ref{conj:SelbergVariant2} is true. Then
\begin{multline}
\label{eq:GeneralBoundForSV2_1}
S_{\cL,x,y}(\sigma) \leq A\left(m_1(\tau),\alpha,\sigma,\sigma\right)\left(\log{\tau}\right)^{2-2\sigma} + O\left(A_4(\sigma,\tau)\left(\log{\tau}\right)^{2-2\sigma}\right) \\ 
- \frac{m_1(\tau)\sigma 2^{1-\sigma}}{1-\sigma} +O\left(A_1\left(\mathcal{C}_{\cL}^{E},2,0,\theta,\sigma,\tau\right)\right) 
\end{multline}
and
\begin{multline}
\label{eq:GeneralBoundForSV2_2}
\widehat{S}_{\cL,x,y}(\sigma) \leq \eta(\alpha,\sigma,\tau)A\left(m_1(\tau),\alpha,\sigma,\sigma\right)\frac{\left(\log{\tau}\right)^{2-2\sigma}}{\log{\log{\tau}}}+O\left(A_5(\sigma,\tau)\frac{\left(\log{\tau}\right)^{2-2\sigma}}{\log{\log{\tau}}}\right)
- \frac{m_1(\tau)\eta(\alpha,\sigma,\tau)}{(1-\sigma)\log{\log{\tau}}} \\
+ m_1(\tau)\log{\left(2\log{\log{\tau}}\right)} + O\left(A_1\left(\mathcal{C}_{\cL}^{E},1,0,\theta,\sigma,\tau\right)\right) 
\end{multline}
for $1/2+\delta\leq\sigma<1$ and sufficiently large $\tau$. Moreover, for sufficiently large $\tau$ we also have
\begin{equation}
\label{eq:GeneralBoundForSV2_3}
\mathcal{E}_x \ll A_3\left(m_4(\tau),\mathcal{C}_{\cL}^{E},0,\theta,\tau\right)
\end{equation}
 Here, $A\left(a,\alpha,u,\sigma\right)$ and $\eta(\alpha,\sigma,\tau)$ are defined by~\eqref{eq:A} and~\eqref{eq:eta}, respectively, and $A_1(a,k,\varepsilon,\theta,\sigma,\tau)$ and $A_3(a,b,\varepsilon,\theta,\tau)$ are defined by~\eqref{eq:MainGeneralA_1} and~\eqref{eq:MainGeneralA_3}, respectively, and $A_4(\sigma,\tau)$ and $A_5(\sigma,\tau)$ are defined by~\eqref{eq:MainGeneralA_4} and~\eqref{eq:MainGeneralA_5}, respectively.
\end{lemma}

\begin{proof}
We are going to prove Lemma~\ref{lem:GeneralBoundForSV2} in the case that Conjecture~\ref{conj:SelbergVariant} is true. The proof of the other case is similar and is thus left to the reader. 

Because we are using~\eqref{eq:main2}, we are taking
\begin{equation}
\label{eq:fg}
f(u) = u\sqrt{\mathcal{C}_{\cL}^{P_1}\left(\log^{2}{\tau}\right)}, \quad g(u) = O(1)\cdot\left(\sqrt{\mathcal{C}_{\cL}^{P_1}\left(\log^{2}{\tau}\right)+\mathcal{C}_{\cL}^{P_2}}\frac{u}{\sqrt{\log{u}}}+\mathcal{C}_{\cL}^{E}u^{\frac{1}{2}+\theta}\log^{2}{u}\right)
\end{equation}
in~\eqref{eq:ineqForS} and~\eqref{eq:ineqForShat} in order to prove~\eqref{eq:GeneralBoundForSV2_1} and~\eqref{eq:GeneralBoundForSV2_2}. Note that this is allowed since $\mathcal{C}_{\cL}^{P_1}(u)$ is an increasing function and $u\leq xy=\log^{2}{\tau}$. The proof is very similar to that of Lemma~\ref{lem:GeneralBoundForS}. 

To estimate the first three terms in \eqref{eq:GeneralBoundForSV2_1} and \eqref{eq:GeneralBoundForSV2_1}, we just change the constants in-front and replace the terms $\sigma-\varepsilon$ with $\sigma$. In case for~\eqref{eq:GeneralBoundForSV2_1}, for the last two terms, we are using
\begin{multline*}
\int_{2}^{xy}\frac{u}{u^{1+\sigma}\sqrt{\log{u}}}\dif{u} \leq \sqrt{\log{(xy)}}\int_{2}^{xy}\frac{\dif{u}}{u^{\sigma}\log{u}} \\ 
= \frac{(xy)^{1-\sigma}-1}{(1-\sigma)\sqrt{\log{(xy)}}} + \sqrt{\log{(xy)}}\log{\log{(xy)}} + O\left(\frac{(xy)^{1-\sigma}}{(1-\sigma)^{2}\sqrt{\log{(xy)}}\log{(xy)}}\right)
\end{multline*}
where the inequality follows from~\eqref{eq:int2}. Also, using partial integration and estimate ~\eqref{eq:int2} again, we obtain
\begin{flalign*}
\int_{2}^{xy}\frac{u}{u^{1+\sigma}\sqrt{\log{u}}\log{u}}\dif{u} &\leq \sqrt{\log{(xy)}}\int_{2}^{xy}\frac{\dif{u}}{u^{\sigma}\log^{2}{u}}=\sqrt{\log{(xy)}}\left(-\left[\frac{1}{u^{\sigma-1}\log{u}}\right]_2^{xy}-\int_2^{xy} \frac{\sigma-1}{u^\sigma\log{u}} \dif{u}\right) \\ 
&\leq \left((1-\sigma)\log{\log{(xy)}}+\frac{2^{1-\sigma}}{\log{2}}\right)\sqrt{\log{(xy)}} + O\left(\frac{(xy)^{1-\sigma}}{(1-\sigma)\sqrt{\log{(xy)}}\log{(xy)}}\right)
\end{flalign*}
for~\eqref{eq:GeneralBoundForSV2_2}. 

For the proof of~\eqref{eq:GeneralBoundForSV2_3} we are taking
\[
f(u) = u\sqrt{\mathcal{C}_{\cL}^{P_1}(u)+\mathcal{C}_{\cL}^{P_2}}, \quad g(u) = O(1)\cdot\mathcal{C}_{\cL}^{E}u^{\frac{1}{2}+\theta}\log^{2}{u}
\]
in~\eqref{eq:ineqForEx}. Then
\begin{flalign*}
\int_{x}^{\infty}\frac{\left(1+\frac{3}{2}\log{u}\right)f(u)}{u^{5/2}\log^{2}{u}}\dif{u} &\ll \frac{1}{\log{x}}\int_{x}^{\infty}u^{-5/4}\sqrt{\frac{\mathcal{C}_{\cL}^{P_1}(u)+\mathcal{C}_{\cL}^{P_2}}{u^{1/2}}}\dif{u} \\ 
&\leq \frac{1}{\log{x}}\sqrt{\frac{\mathcal{C}_{\cL}^{P_1}(x)+\mathcal{C}_{\cL}^{P_2}}{x^{1/2}}}\int_{x}^{\infty}u^{-5/4}\dif{x} \ll \frac{\sqrt{\mathcal{C}_{\cL}^{P_1}(x)+\mathcal{C}_{\cL}^{P_2}}}{x^{1/2}\log{x}}
\end{flalign*}
for sufficiently large $x$. The first term in~\eqref{eq:GeneralBoundForSV2_3} now follows since $\log^{2}{\tau}\ll x\leq\log^{2}{\tau}$. Derivation of the second term in~\eqref{eq:GeneralBoundForSV2_3} is the same as in the proof of Lemma~\ref{lem:GeneralBoundForS} and will thus be omitted.
\end{proof}

\begin{lemma}
\label{lem:GeneralBoundFor1Line}
Let $\cL\in\cS$. Define $m_{i}(\tau)$ for $i\in\{1,2,3\}$ by~\eqref{eq:m123} if Conjecture~\ref{conj:SelbergVariant} is true, and by~\eqref{eq:m123V2} if Conjecture~\ref{conj:SelbergVariant2} is true. Then
\begin{equation}
\label{eq:GeneralBoundFor1Line1}
S_{\cL,x,y}(1) \leq \left(2m_1(\tau)+O\left(m_3(\tau)\log{\log{\log{\tau}}}\right)\right)\log{\log{\tau}} + O\left(\frac{\mathcal{C}_{\cL}^{E}}{\left(\frac{1}{2}-\theta\right)^3}\right)
\end{equation}
and
\begin{equation}
\label{eq:GeneralBoundFor1Line2}
\widehat{S}_{\cL,x,y}(1) \leq m_1(\tau)\left(\log{\left(2\log{\log{\tau}}\right)+2}\right) + O\left(m_2(\tau)+\frac{\mathcal{C}_{\cL}^{E}}{\left(\frac{1}{2}-\theta\right)^2}\right)
\end{equation}
for sufficiently large $\tau$. 
\end{lemma}

\begin{proof}
We are going to prove Lemma~\ref{lem:GeneralBoundFor1Line} in the case that Conjecture~\ref{conj:SelbergVariant} is true. The proof of the other case is similar and is thus left to the reader.

Let $\sigma=1$ and $\alpha=1$. Because we are using~\eqref{eq:main2}, we are taking~\eqref{eq:fg} in~\eqref{eq:ineqForS} and~\eqref{eq:ineqForShat} in order to prove~\eqref{eq:GeneralBoundFor1Line1} and~\eqref{eq:GeneralBoundFor1Line2}. The sum of the first three terms on the right-hand side of~\eqref{eq:ineqForS} is 
\[
m_1(\tau)\left(\log{(xy)}-\frac{1}{2}\log{y}-\log{2}+1\right) = m_1(\tau)\left(2\log{\log{\tau}}-\alpha-\log{2}+1\right) \leq 2m_1(\tau)\log{\log{\tau}},
\]
while the sum of the last two terms is
\[
\ll m_2(\tau)\int_{2}^{xy}\frac{\dif{u}}{u\sqrt{\log{u}}} + A_1\left(\mathcal{C}_{\cL}^{E},2,0,\theta,1,\tau\right) \ll m_2(\tau)\sqrt{\log{xy}}\log{\log{(xy)}} + \frac{\mathcal{C}_{\cL}^{E}}{\left(\frac{1}{2}-\theta\right)^{3}}.
\]
This proves~\eqref{eq:GeneralBoundFor1Line1}. Similarly, the sum of the first three terms on the right-hand side of~\eqref{eq:ineqForShat} is 
\begin{flalign*}
&\leq m_1(\tau)\int_2^x \frac{\dif{u}}{u\log{u}}+\frac{m_1(\tau)}{\log{y}\log{x}}\int_x^{xy} \frac{\log{\frac{xy}{u}}}{u} \dif{u}+\frac{f(2)}{2\log{2}}\\
&= m_1(\tau)\left(\log{\log{x}}+\frac{\log{y}}{2\log{x}}+\frac{1}{\log{2}}-\log{\log{2}}\right) 
\leq m_1(\tau)\left(\log{\left(2\log{\log{\tau}}\right)}+2\right),
\end{flalign*}
while the sum of the last two terms is
\[
\ll m_2(\tau)\int_{2}^{xy}\frac{\dif{u}}{u\log^{3/2}{u}} + A_1\left(\mathcal{C}_{\cL}^{E},1,0,\theta,1,\tau\right) \ll m_2(\tau) + \frac{\mathcal{C}_{\cL}^{E}}{\left(\frac{1}{2}-\theta\right)^{2}}.
\]
This proves~\eqref{eq:GeneralBoundFor1Line2}. The proof is thus complete.
\end{proof}

The following result is needed in order to estimate a certain term in the Guinand--Weil formula, see Remark~\ref{rem:LemmaOverZeros}.

\begin{lemma}
\label{lem:GeneralBoundForGW}
Let $\cL\in\cS$. Then 
\[
\sum_{n\leq\left(\log{\tau}\right)^{2}}\frac{\left|\Lambda_{\cL}(n)\right|}{n} \ll 
\left\{
\begin{array}{ll}
  A_2\left(\mathcal{C}_{\cL}^{R}(\varepsilon),\mathcal{C}_{\cL}^{R}(\varepsilon)+\mathcal{C}_{\cL}^{E},\varepsilon,\theta,\tau\right), & \textrm{RH and}\;\varepsilon\in(0,1/2), \\
  A_2\left(m_4(\tau),\mathcal{C}_{\cL}^{E},0,\theta,\tau\right), & \textrm{Conjecture~\ref{conj:SelbergVariant} and}\;m_4(\tau)\; \textrm{from~\eqref{eq:m123}}, \\
  A_2\left(m_4(\tau),\mathcal{C}_{\cL}^{E},0,\theta,\tau\right), & \textrm{Conjecture~\ref{conj:SelbergVariant2} and}\;m_4(\tau)\; \textrm{from~\eqref{eq:m123V2}}
\end{array}
\right.
\]
for sufficiently large $\tau$, and $A_2\left(a,b,\varepsilon,\theta,\tau\right)$ is defined by~\eqref{eq:MainGeneralA_2}.
\end{lemma}

\begin{proof}
By partial summation, using notation \eqref{eq:PsiTilde},
\begin{equation}
\label{eq:lemmaPS}    
\sum_{n\leq X}\frac{\left|\Lambda_{\cL}(n)\right|}{n} = \frac{\widetilde{\psi}_{\cL}(X)}{X} + \int_{2}^{X}\frac{\widetilde{\psi}_{\cL}(u)}{u^{2}}\dif{u}
\end{equation}
for $X\geq 2$. By~\eqref{eq:main}, the right-hand side of~\eqref{eq:lemmaPS} is
\[
\leq \mathcal{C}_{\cL}^{R}(\varepsilon)\left(X^{\varepsilon}+\frac{X^{\varepsilon}-2^{\varepsilon}}{\varepsilon}\right) + O\left(\left(\mathcal{C}_{\cL}^{R}(\varepsilon)+\mathcal{C}_{\cL}^{E}\right)\left(X^{\Theta_{\theta,\varepsilon}(1)}\log^{2}{X}+\int_{2}^{X}u^{\Theta_{\theta,\varepsilon}(2)}\log^{2}{u}\dif{u}\right)\right).
\]
The first term is
\[
\ll \mathcal{C}_{\cL}^{R}(\varepsilon) X^{\varepsilon}\min\left\{\frac{1}{\varepsilon},\log{X}\right\},
\]
and in order to bound the second term similarly as in \eqref{eq:integral}, we note that
\[
\int_{2}^{a}u^{b}\log^{2}{u}\dif{u} \ll \min\left\{\frac{1}{|1+b|^3},\log^{3}{a}\right\}
\]
for every $a\geq 2$ and $b\leq-1$. The first estimate from Lemma~\ref{lem:GeneralBoundForGW} now follows since $\Theta_{\theta,\varepsilon}(2)+1=\Theta_{\theta,\varepsilon}(1)<0$ and $\log^{2}{a}\ll b^{-3}a^{b}$ uniformly for $a\geq 2$ and $b>0$. Derivation of the remaining two estimates is very similar to the above, just now with employing~\eqref{eq:main2} and~\eqref{eq:main3} in~\eqref{eq:lemmaPS}, and estimating $x/\log{x} \ll x/\sqrt{\log{x}} \ll x$. 
\end{proof}

\begin{remark}
In all of the lemmas in this section, in the case $s \in \mathbb{R}$ the term $\tau$ can be replaced with $\sq$.
\end{remark}

\subsection{The case when $\cL\in\cSP$}
\label{sec:PrimesForSP}

Throughout this section the truth of RH is assumed.

\begin{lemma}
\label{sec:Ex}
If $\cL\in\mathcal{SP}$, then
\begin{equation}
\label{eq:Ex}
\frac{1}{m}\mathcal{E}_{x} \leq \frac{2\eta\left(\alpha,\sigma,\tau\right)\exp{\left(\frac{\alpha}{2\sigma-1}\right)}}{\log{\tau}\log{\log{\tau}}}
+ \frac{5\exp{\left(\frac{2\alpha}{2\sigma-1}\right)}\left(1+\frac{\log{\log{\tau}}}{\eta\left(\alpha,\sigma,\tau\right)}\right)}{16\pi\log^{2}{\tau}},
\end{equation}
where $\eta\left(\alpha,\sigma,\tau\right)$ is defined by~\eqref{eq:eta}.
\end{lemma}

\begin{proof}
If $x \geq 74$, then by using~\cite[Equation 6.2]{SchoenfeldSharperRH}, i.e., $|\psi(u)-u|\leq \frac{1}{8\pi}\sqrt{u}\log^{2}{u}$ for $u\geq 74$, we have
\begin{flalign*}
\frac{1}{m}\mathcal{E}_{x} &\leq \sum_{n>x}\frac{\Lambda(n)}{n^{\frac{3}{2}}\log{n}}
= \int_{x}^{\infty}\frac{\dif{u}}{u^{\frac{3}{2}}\log{u}} - \frac{\psi(x)-x}{x^{\frac{3}{2}}\log{x}}
+\frac{1}{2}\int_{x}^{\infty}\frac{\left(2+3\log{u}\right)\left(\psi(u)-u\right)}{u^{\frac{5}{2}}\log^{2}{u}}\dif{u} \nonumber \\
&\leq \frac{2}{\sqrt{x}\log{x}} + \frac{5\left(1+\log{x}\right)}{16\pi x}
= \frac{2\eta\left(\alpha,\sigma,\tau\right)\exp{\left(\frac{\alpha}{2\sigma-1}\right)}}{\log{\tau}\log{\log{\tau}}}
+ \frac{5\exp{\left(\frac{2\alpha}{2\sigma-1}\right)}\left(1+\frac{\log{\log{\tau}}}{\eta\left(\alpha,\sigma,\tau\right)}\right)}{16\pi\log^{2}{\tau}}.
\end{flalign*}
Further, if $2\leq x<74$, then we can estimate
\begin{multline*}
\frac{1}{m}\mathcal{E}_{x} \leq \sum_{n>x}\frac{\Lambda(n)}{n^{\frac{3}{2}}\log{n}} = \left(\sum_{x<n\leq 74}+\sum_{n>74}\right)\frac{\Lambda(n)}{n^{\frac{3}{2}}\log{n}} < \left(\sum_{x<n\leq 74}+\sum_{74<n\leq 10^5}\right)\frac{\Lambda(n)}{n^{\frac{3}{2}}\log{n}}+\sum_{n>10^{5}}\frac{1}{n^{\frac{3}{2}}} \\
\leq 0.05 +\sum_{x<n\leq74}\frac{\Lambda(n)}{n^{\frac{3}{2}}\log{n}} \leq \frac{2}{\sqrt{x}\log{x}} + \frac{5\left(1+\log{x}\right)}{16\pi x},
\end{multline*}
with the last inequality directly verified for $2\leq x\leq 74$ by computer.
\end{proof}

Next we estimate $S_{\cL,x,y}(\sigma)$ when $\cL\in\mathcal{SP}$. 

\begin{lemma}
\label{sec:SforSP}
Assume $\cL\in\mathcal{SP}$. Then 
\begin{multline}
\label{eq:BoundForS}
\frac{1}{m}S_{\cL,x,y}(\sigma) \leq A(1,\alpha,\sigma,\sigma)\left(\log{\tau}\right)^{2-2\sigma}-\dfrac{\sigma2^{1-\sigma}}{1-\sigma} \\
+ \frac{\sigma 2^{\frac{3}{2}-\sigma}\left(4+\left(2+(2\sigma-1)\log{2}\right)^{2}\right)}{8\pi(2\sigma-1)^{3}}
+ \left(\frac{\left(2\left(\alpha+1\right)\sigma-1\right)\left(e^{\alpha}-1\right)}{2\pi\alpha}\right)
\dfrac{(\log\log{\tau})^2}{(2\sigma-1)(\log{\tau})^{2\sigma-1}}
\end{multline}
for the variables $\alpha$, $\sigma\neq 1$ and $\tau$, which satisfy conditions from the cases~(1) or~(2) of Theorem~\ref{thm:MainExplicit}, and $A\left(a,\alpha,u,\sigma\right)$ is defined by~\eqref{eq:A}. Further, for $\sigma$ from the range~\eqref{eq:sigmaRange2} we have
\begin{multline}
\label{eq:BoundForSAround1}
\frac{1}{m}S_{\cL,x,y}(\sigma) \leq 2\log{\log{\tau}} + 1 - \sigma\log{2} + 4\theta_1(M)|1-\sigma|\left(\log{\log{\tau}}\right)^{2} + \left(\log{2}\right)^{2}\theta_1\left(\frac{M\log{2}}{2\log{\log{\tau_0}}}\right)\sigma|1-\sigma| \\
+ \frac{\sigma 2^{\frac{3}{2}-\sigma}\left(4+\left(2+(2\sigma-1)\log{2}\right)^{2}\right)}{8\pi(2\sigma-1)^{3}} +
\left(\frac{\left(2\left(\alpha+1\right)\sigma-1\right)\left(e^{\alpha}-1\right)}{2\pi\alpha}\right)
\dfrac{(\log\log{\tau})^2}{(2\sigma-1)(\log{\tau})^{2\sigma-1}},
\end{multline}
where $M$ and $\theta_{1}(u)$ are from~\eqref{eq:thetas}. Also,
\begin{equation}
\label{eq:BoundForS1}
\frac{1}{m}S_{\cL,x,y}(1) \leq 2\log{\log{\tau}} - \gamma - \alpha + \frac{\left(e^{\alpha}-1\right)(2\alpha+1)\left(\log{\log{\tau}}\right)^{2}}{2\pi\alpha\log{\tau}} + \frac{0.24e^{\alpha}}{\log{\tau}}.
\end{equation}
\end{lemma}

\begin{proof}
We know that if $\cL\in\mathcal{SP}$, then inequality~\eqref{eq:BoundOnLambda} holds and thus by integration by parts
\begin{flalign*}
\frac{1}{m}S_{\cL,x,y}(\sigma) &\leq \sum_{n\leq x}\frac{\Lambda(n)}{n^{\sigma}} + \frac{1}{\log{y}}\sum_{x<n\leq xy}\frac{\Lambda(n)\log{\frac{xy}{n}}}{n^{\sigma}}=\frac{x^{1-\sigma}-\sigma2^{1-\sigma}}{1-\sigma}-\int_{2}^x \left(\psi(u)-u\right)\left(\frac{1}{u^\sigma}\right)'\dif{u} \\
&+\frac{(xy)^{1-\sigma}}{\log y}\int_{1}^y \frac{\log u}{u^{2-\sigma}} \dif{u}-\frac{1}{\log y}\int_{x}^{xy} \left(\psi(u)-u\right)\left(\frac{1}{u^\sigma}\log\frac{xy}{u}\right)'\dif{u} \\
&\leq \frac{(xy)^{1-\sigma}-x^{1-\sigma}}{(1-\sigma)^{2}\log{y}} - \frac{\sigma 2^{1-\sigma}}{1-\sigma} + \frac{\sigma}{8\pi}\int_{2}^{x}\frac{\log^{2}{u}}{u^{\frac{1}{2}+\sigma}}\dif{u}
+\frac{\left(\sigma+\frac{1}{\log y}\right)\left(x^{\frac{1}{2}-\sigma}-(xy)^{\frac{1}{2}-\sigma}\right)\log^2 (xy)}{8\pi\left(\sigma-\frac{1}{2}\right)},
\end{flalign*}
see~\cite[Section~3]{ChirreSimonicHagen} for similar ideas. Taking into account also~\eqref{eq:xy}, the last inequality assures~\eqref{eq:BoundForS}.

Let us now prove \eqref{eq:BoundForS1} since we are going to use this bound to prove \eqref{eq:BoundForSAround1}. Observe that $x=e^{-2\alpha}\log^{2}{\tau}$ since $\sigma=1$. Therefore, if $\tau\geq\exp{\left(e^{\alpha}\sqrt{60}\right)}$, then $x\geq 60$ and thus~\eqref{eq:BoundForS1} holds by~\cite[Equation~(3.6)]{ChirreSimonicHagen}. 

Let us now prove \eqref{eq:BoundForSAround1}. Firstly, assume that $\sigma \neq 1$. By the first paragraph of the proof we know that~\eqref{eq:BoundForS} holds for $\sigma$ from the range~\eqref{eq:sigmaRange2} while $\sigma \neq 1$. Since $e^{u}-1-u \geq 0$ for all real numbers $u$, we have
\begin{flalign}
\label{eq:ASecondEst0}
    A(1,\alpha,\sigma,\sigma)\left(\log{\tau}\right)^{2-2\sigma}-\frac{\sigma2^{1-\sigma}}{1-\sigma} &= -(\log\tau)^{2-2\sigma}\left(\frac{2\alpha}{2\sigma-1}\right)\left(\frac{\exp\left(-\frac{2\alpha(1-\sigma)}{2\sigma-1}\right)-1-\left(-\frac{2\alpha(1-\sigma)}{2\sigma-1}\right)}{\left(-\frac{2\alpha(1-\sigma)}{2\sigma-1}\right)^2}\right) \nonumber \\
    &+\frac{(\log\tau)^{2-2\sigma}}{1-\sigma}-\frac{\sigma2^{1-\sigma}}{1-\sigma}\leq \frac{(\log\tau)^{2-2\sigma}}{1-\sigma}-\frac{\sigma2^{1-\sigma}}{1-\sigma}.
\end{flalign}
Even more, note that $\left|e^{u}-1-u\right|\leq\theta_1\left(u_0\right)u^2$ for $|u|\leq u_0$, where $\theta_1\left(u_0\right)$ is from ~\eqref{eq:thetas}. This implies 
\begin{gather*}
\left|\left(\log{\tau}\right)^{2-2\sigma} - \left(1+ 2(1-\sigma)\log{\log{\tau}}\right)\right| \leq 4\theta_{1}(M)(1-\sigma)^{2}\left(\log{\log{\tau}}\right)^{2}, \\
\left|\frac{\sigma 2^{1-\sigma}}{1-\sigma} - \left(\frac{\sigma}{1-\sigma} + \sigma\log{2}\right)\right| \leq \left(\log{2}\right)^{2}\theta_1\left(\frac{M\log{2}}{2\log{\log{\tau_0}}}\right)\sigma|1-\sigma|,
\end{gather*}
since $|1-\sigma|\leq M/\left(2\log{\log{\tau}}\right)$ with $M$ from~\eqref{eq:thetas}. Therefore, by the previous two inequalities and~\eqref{eq:ASecondEst0} we have
\begin{flalign*}
A(1,\alpha,\sigma,\sigma)\left(\log{\tau}\right)^{2-2\sigma}-\frac{\sigma2^{1-\sigma}}{1-\sigma} &\leq 2\log\log\tau+1-\sigma\log 2 \\
&+4\theta_{1}(M)|1-\sigma|\left(\log{\log{\tau}}\right)^{2}+\left(\log{2}\right)^{2}\theta_1\left(\frac{M\log{2}}{2\log{\log{\tau_0}}}\right)\sigma|1-\sigma|,
\end{flalign*}
which gives~\eqref{eq:BoundForSAround1} for $\sigma\neq 1$ from the range~\eqref{eq:sigmaRange2}. However, estimate~\eqref{eq:BoundForSAround1} holds also for $\sigma=1$ by~\eqref{eq:BoundForS1}. The proof of Lemma~\ref{sec:SforSP} is thus complete.
\end{proof}

Finally, we estimate $\widehat{S}_{\cL,x,y}(\sigma)$ when $\cL\in\mathcal{SP}$. The method we are using here is very similar to that from the proof of Lemma~\ref{sec:SforSP}. 

\begin{lemma}
Assume $\cL\in\mathcal{SP}$, $\nu_1>0$ and $\nu_2>1$, and let $\theta_1(u)$ be defined as in~\eqref{eq:thetas}. Then, under the assumptions from the case (1) of Theorem~\ref{thm:MainExplicit}, we have
\begin{flalign}
\label{eq:SHat1}
\frac{1}{m}\widehat{S}_{\cL,x,y}(\sigma) &\leq \eta(\alpha,\sigma,\tau)
\left(A(1,\alpha,\sigma,\sigma)\frac{\left(\log{\tau}\right)^{2-2\sigma}}{\log{\log{\tau}}}-\frac{1}{(1-\sigma)\log{\log{\tau}}}\right) \nonumber \\
&+\left(\left(\nu_{2}\eta(\alpha,\sigma,\tau)\right)^{2}\exp{\left(-\frac{2\alpha(1-\sigma)}{2\sigma-1}\right)}\right)
\frac{\left(\log{\tau}\right)^{2-2\sigma}}{(1-\sigma)^{2}\left(\log{\log{\tau}}\right)^{2}} \nonumber \\
&+\frac{1}{\nu_{1}^{2}}\exp{\left(-\frac{2\alpha(1-\sigma)}{(2\sigma-1)\nu_{2}}\right)}\left(\log{\tau}\right)^{\frac{2}{\nu_{2}}(1-\sigma)} + \log{\left(2\log{\log{\tau}}\right)}
+ \frac{1-\sigma 2^{1-\sigma}}{(1-\sigma)\log{2}} - \log{\log{2}} \nonumber \\
&+ \int_{0}^{\nu_1}\theta_{1}(u)\dif{u}
+ \frac{2^{\frac{3}{2}-\sigma}\left(-1+\sigma\left(4-\log{2}\right)+\sigma^2\log{4}\right)}{8\pi\left(2\sigma-1\right)^{2}}
+ \frac{\left(e^{\alpha}-1\right)\log{\log{\tau}}}{4\pi\alpha\left(\log{\tau}\right)^{2\sigma-1}},
\end{flalign}
where $A\left(a,\alpha,u,\sigma\right)$ and $\eta(\alpha,\sigma,\tau)$ are defined by~\eqref{eq:A} and~\eqref{eq:eta}, respectively. 

Further, under the assumptions from the case (2) of Theorem \ref{thm:MainExplicit} and assuming that $\sigma \neq 1$, we have
\begin{flalign}
\label{eq:SHat2}
\frac{1}{m}\widehat{S}_{\cL,x,y}(\sigma) &\leq \log{\left(2\log{\log{\tau}}\right)} + 2\left(2\eta(\alpha,\sigma,\tau)\theta_1(M)+\theta_1(2\alpha_2)\right)|1-\sigma|\log{\log{\tau}} \nonumber \\
&+ \frac{1}{\log{2}} - \log{\log{2}} - \sigma + \left(\log{2}\right)\theta_1\left(\frac{M\log{2}}{2\log{\log{\tau_0}}}\right)\sigma|1-\sigma|+ 2\eta(\alpha,\sigma,\tau) \nonumber \\
&+ \frac{2^{\frac{3}{2}-\sigma}\left(-1+\sigma\left(4-\log{2}\right)+\sigma^2\log{4}\right)}{8\pi\left(2\sigma-1\right)^{2}}
+ \frac{\left(e^{\alpha}-1\right)\log{\log{\tau}}}{4\pi\alpha\left(\log{\tau}\right)^{2\sigma-1}},
\end{flalign}
where $M=2\max\left\{\alpha_2,\alpha_3\right\}$.

Finally, if $\sigma=1$, then
\begin{flalign}
\label{eq:SHatOn1}
\frac{1}{m}\widehat{S}_{\cL,x,y}(1) &\leq \log{\left(2\log{\log{\tau}}\right)} + \gamma + \left(1-\frac{\log{\log{\tau}}}{\alpha}\right)\log{\left(1-\frac{\alpha}{\log{\log{\tau}}}\right)} - 1 \nonumber \\
&+ \frac{\left(1+e^{\alpha}\left(4\alpha-1\right)\right)\log{\log{\tau}}+e^{\alpha}\left(7\alpha-4\left(\alpha^2+1\right)\right)+4}{4\pi\alpha\log{\tau}}.
\end{flalign}
\end{lemma}

\begin{proof}
Let $\sigma\in[0,\infty)\setminus\{1\}$, $x\geq 2$ and $y\geq 2$. By partial summation we have
\begin{equation}
\label{eq:SHatAux1}
\sum_{n\leq x}\frac{\Lambda(n)}{n^{\sigma}\log{n}} \leq \int_{2}^{x}\frac{\dif{u}}{u^{\sigma}\log{u}} + \frac{\psi(x)-x}{x^{\sigma}\log{x}} + \frac{2^{1-\sigma}}{\log{2}} + \frac{1}{8\pi}\int_{2}^{x}\frac{1+\sigma\log{u}}{u^{\frac{1}{2}+\sigma}}\dif{u}
\end{equation}
and
\begin{flalign}
\frac{1}{\log{y}}&\sum_{x<n\leq xy}\frac{\Lambda(n)\log{\frac{xy}{n}}}{n^{\sigma}\log{n}} \leq \frac{(xy)^{1-\sigma}}{\log{y}}\int_{1}^{y}\frac{\log{u}\;\dif{u}}{u^{2-\sigma}\log{\frac{xy}{u}}} - \frac{\psi(x)-x}{x^{\sigma}\log{x}} \nonumber \\
&+\frac{1}{8\pi\log{y}}\int_{x}^{xy}\frac{\left(1+\sigma\log{u}\right)\log{\frac{xy}{u}}+\log{u}}{u^{\frac{1}{2}+\sigma}}\dif{u} \label{eq:SHatAux2(a)} \\
&\leq \frac{(xy)^{1-\sigma}}{\log{x}\log{y}}\int_{1}^{y}\frac{\log{u}}{u^{2-\sigma}}\dif{u} - \frac{\psi(x)-x}{x^{\sigma}\log{x}}
+\frac{1}{8\pi}\int_{x}^{xy}\frac{1+\sigma\log{u}}{u^{\frac{1}{2}+\sigma}}\dif{u}
+\frac{1}{8\pi\log{y}}\int_{x}^{xy}\frac{\log{u}}{u^{\frac{1}{2}+\sigma}}\dif{u}, \label{eq:SHatAux2}
\end{flalign}
see~\cite[Section~3]{ChirreSimonicHagen} for similar ideas. We know that if $\cL\in\mathcal{SP}$, then inequality~\eqref{eq:BoundOnLambda} holds. Therefore, combining~\eqref{eq:SHatAux1} and~\eqref{eq:SHatAux2} gives
\begin{multline}
\label{eq:SHatxy}
\frac{1}{m}\widehat{S}_{\cL,x,y}(\sigma) \leq \frac{(xy)^{1-\sigma}-x^{1-\sigma}}{(1-\sigma)^{2}\log{x}\log{y}}-\frac{x^{1-\sigma}}{(1-\sigma)\log{x}}
+\int_{2}^{x}\frac{\dif{u}}{u^{\sigma}\log{u}} \\
+ \frac{2^{1-\sigma}}{\log{2}} + \frac{1}{8\pi}\int_{2}^{xy}\frac{1+\sigma\log{u}}{u^{\frac{1}{2}+\sigma}}\dif{u} + \frac{1}{8\pi\log{y}}\int_{x}^{xy}\frac{\log{u}}{u^{\frac{1}{2}+\sigma}}\dif{u}.
\end{multline}
The strategy to bound~\eqref{eq:SHatxy} is now very similar to that in the proof of Lemma~\ref{sec:SforSP}, but with an exemption of the first integral in~\eqref{eq:SHatxy}. We can bound the integral over $1/(u^\sigma\log{u})$ similarly as in \eqref{eq:integral2} and \eqref{eq:Case2}. Consider now the assumptions from the case (1) of Theorem~\ref{thm:MainExplicit}. Then~\eqref{eq:SHatxy},~\eqref{eq:integral2} and~\eqref{eq:Case2} imply inequality \eqref{eq:SHat1}.

For $\sigma\geq 1-\alpha_2/\log{\log{\tau}}$, we also have
\begin{equation}
\label{eq:Case1}
\int_{(1-\sigma)\log{2}}^{(1-\sigma)\log{x}}\theta_1(u)\dif{u} \leq \theta_{1}(2\alpha_2)|1-\sigma|\log{x}
\end{equation}
since the integral in~\eqref{eq:Case1} is negative for $\sigma>1$ and $(1-\sigma)\log{x}\leq 2\alpha_2$. 
Consider the assumptions from the case (2) of Theorem~\ref{thm:MainExplicit} and assume $\sigma \neq 1$. Then similarly as in the proof of Lemma~\ref{sec:SforSP},~\eqref{eq:SHatxy},~\eqref{eq:integral2} and~\eqref{eq:Case1} imply \eqref{eq:SHat2}.

As in the proof of Lemma~\ref{sec:SforSP}, we can estimate $\widehat{S}_{\cL,x,y}(1)$ even better. We have
\begin{flalign}
\label{eq:SHatAux1On1}
\sum_{n\leq x}\frac{\Lambda(n)}{n\log{n}} \leq \sum_{p\leq x}\frac{1}{p} + \sum_{p}\sum_{k=2}^{\infty}\frac{1}{kp^k} &= \sum_{p\leq x}\frac{1}{p} - \sum_{p}\left(\frac{1}{p}+\log{\left(1-\frac{1}{p}\right)}\right) \nonumber \\
&\leq \log{\log{x}} + \gamma + \frac{\psi(x)-x}{x\log{x}} + \frac{3+\log{x}}{4\pi\sqrt{x}}
\end{flalign}
for $x\geq 60$. The last inequality follows for $x\geq599$ by~\cite[Equations 2.7 and 4.15]{RosserSchoenfeld},~\cite[Equation 6.3]{SchoenfeldSharperRH} and $\vartheta(x)-x\leq\psi(x)-x$, see also the proof of~\cite[Corollary 2]{SchoenfeldSharperRH}. We can verify by computer that it holds also for $60\leq x\leq 599$. Next, note that~\eqref{eq:SHatAux2(a)} holds also for $\sigma=1$. Therefore,
\begin{flalign}
\label{eq:SHatAux2On1}
\frac{1}{\log{y}}\sum_{x<n\leq xy}\frac{\Lambda(n)\log{\frac{xy}{n}}}{n\log{n}} &\leq -\frac{\log{(xy)}}{\log{y}}\log{\left(1-\frac{\log{y}}{\log{(xy)}}\right)} - 1 - \frac{\psi(x)-x}{x\log{x}} \nonumber \\
&+\frac{8+\log{(xy)}}{4\pi\sqrt{xy}\log{y}} + \frac{3+\log{x}}{4\pi\sqrt{x}} - \frac{8+\log{x}}{4\pi\sqrt{x}\log{y}}.
\end{flalign}
Then~\eqref{eq:xy},~\eqref{eq:SHatAux1On1} and~\eqref{eq:SHatAux2On1} imply \eqref{eq:SHatOn1}.
\end{proof}
 
\begin{remark}
\label{rmk:Sreal}
Note that in the proofs of the lemmas in this section we do not need any assumptions for $t$, the assumptions for $\sigma$ and $\tau$ are sufficient. Hence, if $s \in \mathbb{R}$, we can replace $\tau$ with $\sq$ in all of the lemmas in this section. 
\end{remark}

\section{The sum over the non-trivial zeros}
\label{sec:SumZeros}

In this section we obtain an upper estimate for the sum from~\eqref{eq:SumOverZeros}, which we obtained from the Selberg moment formula, over the non-trivial zeros of $\cL(s)$. We will do this with the help of the Guinand--Weil exact formula for functions in the Selberg class, see the next lemma.

\begin{lemma}
\label{lemma:Guinand}
Assume that $\cL\in\cS$ and $\cL(1)\neq 0$. Let $h(z)$ be a holomorphic function on a strip
\[
\left\{z \in \mathbb{C}\colon -\frac{1}{2}-\varepsilon< \Im\{z\} <\frac{1}{2}+\varepsilon\right\}
\]
for some $\varepsilon >0$, such that $h(z)\left(1+|z|\right)^{1+\delta}$ is bounded for some $\delta>0$. Then
\begin{flalign}
\label{eq:GW}
\sum_{\rho} h\left(\frac{\rho-\frac{1}{2}}{\ie}\right) &= m_{\cL}\left(h\left(\frac{1}{2\ie}\right)+h\left(-\frac{1}{2\ie}\right)\right)+\frac{\log Q}{\pi}\widehat{h}(0) \nonumber \\
&+\frac{1}{\pi}\sum_{j=1}^{f}\int_{-\infty}^{\infty}h(u)\lambda_{j}\Re \left\{\frac{\Gamma'}{\Gamma}\left(\frac{\lambda_j}{2}+\mu_j+\ie\lambda_{j}u\right)\right\}\dif{u} \nonumber \\
&- \frac{1}{2\pi}\sum_{n=2}^\infty \frac{1}{\sqrt{n}}\left(\Lambda_{\cL}(n) \widehat{h}\left(\frac{\log n}{2\pi}\right)+\overline{\Lambda_{\cL}(n)}\widehat{h}\left(-\frac{\log n}{2\pi}\right)\right),
\end{flalign}
where $\widehat{h}(\xi)=\int_{-\infty}^{\infty}h(u)e^{-2\pi\ie u\xi}\dif{u}$ denotes the Fourier transform of $h(z)$ and the sum runs over all non-trivial zeros $\rho$ of $\cL(s)$.
\end{lemma}

\begin{proof}
The proof is almost identical to the proof of~\cite[Lemma 5]{FinderPhD}, so we will highlight only the main differences. Write $\mathfrak{L}(s)=\cL(s)\mathfrak{L}_{1}(s)$, where $\mathfrak{L}(s)$ is from the functional equation. Then the only zeros $\rho$ of $\mathfrak{L}(s)$ are the non-trivial zeros $\rho$ of $\cL(s)$, where $\rho\notin\{0,1\}$ by the assumption $\cL(1)\neq0$. Let $\varepsilon>0$. By the residue theorem
\begin{equation*}
\sum_{\rho} h\left(\frac{\rho-\frac{1}{2}}{\ie}\right) - m_{\cL}\left(h\left(\frac{1}{2\ie}\right)+h\left(-\frac{1}{2\ie}\right)\right) =
\frac{1}{2\pi\ie}\left(\int_{\gamma_1}+\int_{\gamma_2}\right)h\left(\frac{z-\frac{1}{2}}{\ie}\right)\frac{\mathfrak{L}'}{\mathfrak{L}}(z)\dif{z},
\end{equation*}
where $\gamma_1=\left\{z\in\C\colon \Re\{z\}=-\varepsilon/2\right\}$ and $\gamma_2=\left\{z\in\C\colon \Re\{z\}=1+\varepsilon/2\right\}$ with the appropriate orientation, see~\cite[p.~15]{FinderPhD}. Hence, we have obtained the first term in the right-hand side of \eqref{eq:GW}. Furthermore,
\begin{multline*}
\sum_{\rho} h\left(\frac{\rho-\frac{1}{2}}{\ie}\right) - m_{\cL}\left(h\left(\frac{1}{2\ie}\right)+h\left(-\frac{1}{2\ie}\right)\right) =\frac{1}{2\pi\ie}\int_{\gamma_2}\left(h\left(\frac{z-\frac{1}{2}}{\ie}\right)\frac{\cL'}{\cL}(z)+h\left(\frac{\frac{1}{2}-z}{\ie}\right)\overline{\frac{\cL'}{\cL}(\bar{z})}\right)\dif{z} \\
+ \frac{1}{2\pi\ie}\int_{\gamma_2}\left(h\left(\frac{z-\frac{1}{2}}{\ie}\right)\frac{\mathfrak{L}_{1}'}{\mathfrak{L}_{1}}(z)+h\left(\frac{\frac{1}{2}-z}{\ie}\right)\overline{\frac{\mathfrak{L}_{1}'}{\mathfrak{L}_{1}}(\bar{z})}\right)\dif{z}
\end{multline*}
by employing the functional equation, see~\cite[p.~16]{FinderPhD}. Note that
\[
\frac{\mathfrak{L}_{1}'}{\mathfrak{L}_{1}}(z) = \log{Q} + \sum_{j=1}^{f}\lambda_{j}\frac{\Gamma'}{\Gamma}\left(\lambda_{j}z+\mu_{j}\right).
\]
The last sum in~\eqref{eq:GW} comes from the first two integrals in the latter equation, see Step~3 in~\cite[p.~17]{FinderPhD}. After moving the line of integration in the last two integrals to $\Re\{z\}=1/2$ and noticing that $\left(\mathfrak{L}_{1}'/\mathfrak{L}_{1}\right)(z)$ is holomorphic on domain $\left\{z\in\C\colon \Re\{z\}>0\right\}$, see Step~4 in~\cite[p.~18]{FinderPhD}, we obtain the last two remaining terms from~\eqref{eq:GW}.
\end{proof}

\begin{remark}
\label{rem:GW}
Beside the conditions from Lemma~\ref{lemma:Guinand} assume also GRH and that $h(z)$ is real for $z\in\R$. Note that the last condition assures that $\overline{h(z)}=h(\bar{z})$ and also that $\overline{\widehat{h}(\xi)}=\widehat{h}(-\xi)$ for $\xi\in\R$. We can see that Lemma~\ref{lemma:Guinand} then implies
\begin{flalign*}
\sum_{\gamma}h(\gamma) = 2m_{\cL}\Re\left\{h\left(\frac{1}{2\ie}\right)\right\} + \frac{\log{Q}}{\pi}\widehat{h}(0) &- \frac{1}{\pi}\sum_{n=2}^\infty \frac{1}{\sqrt{n}}\Re\left\{\Lambda_{\cL}(n) \widehat{h}\left(\frac{\log n}{2\pi}\right)\right\} \\
&+\frac{1}{\pi}\sum_{j=1}^{f}\int_{-\infty}^{\infty}h(u)\lambda_{j}\Re \left\{\frac{\Gamma'}{\Gamma}\left(\frac{\lambda_j}{2}+\mu_j+\ie\lambda_{j}u\right)\right\}\dif{u},
\end{flalign*}
where the sum runs over the imaginary parts of the non-trivial zeros of $\cL(s)$.
\end{remark}

\begin{lemma}
\label{lem:SumOverZeros}
Let $\cL\in\mathcal{SP}$. Assume the Generalized Riemann Hypothesis for $\cL$ and the strong $\lambda$-conjecture. Assume also that
\begin{equation}
\label{eq:RangeForSigmaLemma}
\frac{1}{2} + \frac{\alpha_1}{\log{\log{\tau}}} \leq \sigma \leq \frac{3}{2}
\end{equation}
for $\alpha_1>0$, and
\[
\tau \geq \tau_0 \geq \max\left\{e^{\sqrt{60}},\exp{\left(e^{\alpha_1}\right)}\right\}, \quad
|t| \ge t_0\ge \max\left\{2\max_{1\leq j\leq f}\left\{\left|\mu_j\right|\right\},1\right\}.
\]
Then we have
\begin{equation}
\label{eq:SumOverZerosGeneral}
\sum_{\gamma}\frac{\sigma-\frac{1}{2}}{\left(\sigma-\frac{1}{2}\right)^2+(t-\gamma)^2} \leq \frac{\log{\tau}}{2} + \mathfrak{a}\left(m_{\cL},\alpha_1,t_0\right)\left(\log{\tau}\right)^{2-2\sigma} + 2m\log{\log{\tau}} + \mathfrak{b}\left(\sdeg,m,m_{\cL},\alpha_1,t_0,\tau_0\right),
\end{equation}
where
\begin{equation}
\label{eq:Funa}
\mathfrak{a}\left(m_{\cL},\alpha_1,t_0\right) \de \frac{1}{1-e^{-2\alpha_1}}\left(1+\frac{4m_{\cL}}{\left(t_0^2-\frac{3}{4}\right)\left(1-e^{-2\alpha_1}\right)}\right)
\end{equation}
and
\begin{multline}
\label{eq:funb}
\mathfrak{b}\left(\sdeg,m,m_{\cL},\alpha_1,t_0,\tau_0\right) \de \frac{\sdeg}{\pi}\coth^{2}{(\alpha_1)}\int_{0}^{\infty} \frac{\log{\left(1+\frac{y}{2\pi t_0}\right)}}{1+y^2}\dif{y} \\
+ \frac{2m_{\cL}\left(1+e^{-4\alpha_1}\right)}{\left(t_0^2-\frac{3}{4}\right)\left(1-e^{-2\alpha_1}\right)^2}
+ m\left(-1-\gamma+\frac{\frac{4\log{\log{\tau_0}}}{\log{\tau_0}}+\frac{2.24}{\log{\tau_0}-1}}{1-(\log{\tau_0})^{-1}}\right).
\end{multline}
\end{lemma}

\begin{proof}
The proof is similar to the one in~\cite[Sections~4 and~7]{ChirreSimonicHagen}. Indeed, we are using Lemma~\ref{lemma:Guinand} in order to estimate the sum over zeros.

The left-hand side of~\eqref{eq:SumOverZerosGeneral} can be written as $\sum_\gamma f_a(t-\gamma)$, where $a=\sigma-1/2$ and $f_{a}\colon\R\to\R$ is defined by
\[
f_a(x)=\frac{a}{a^2+x^2}.
\]
However, the function $f_a(x)$ does not satisfy the conditions from Lemma~\ref{lemma:Guinand}. So, as in~\cite{ChirreSimonicHagen}, let us set
$$
h(s)=h_{a,\Delta}(s)=\left(\frac{a}{a^2+s^2}\right)\left(\frac{e^{2\pi a \Delta}+e^{-2\pi a \Delta}-2\cos\left(2\pi \Delta s\right)}{\left(e^{\pi a \Delta}-e^{-\pi a \Delta}\right)^2}\right).
$$
By~\cite[Lemma 9]{Carneiro}, for any $\Delta >0$ the function $h(s)$ is an entire function of exponential type of $2\pi \Delta$, $f_a(u) \le h(u)$ for all $u \in \mathbb{R}$ and the Fourier transform $\widehat{h}(\xi)$ satisfies $\widehat{h}(\xi) \ge 0$ if $|\xi| \le \Delta$, $\widehat{h}(\xi)=0$ if $|\xi|>\Delta$ and $\widehat{h}(0)=\pi \coth\left({\pi a \Delta}\right)$. Therefore, the function $h(s)$ satisfies the assumptions given in Lemma~\ref{lemma:Guinand}. Employing Remark~\ref{rem:GW} for the function $z\mapsto h(t-z)$, we obtain
\begin{flalign}
\label{eq:sumWithGuinand}
\sum_{\gamma}f_a(t-\gamma) &\le \sum_{\gamma} h(t-\gamma) \le 2m_{\cL}\left|h\left(t+\frac{1}{2\ie}\right)\right|+\frac{\log Q}{\pi}\widehat{h}(0) \nonumber \\
&+\frac{1}{\pi}\sum_{j=1}^f\int_{-\infty}^{\infty}h(u)\lambda_{j}\Re \left\{\frac{\Gamma'}{\Gamma}\left(\frac{\lambda_j}{2}+\mu_j+\lambda_j(t-u)\ie\right)\right\}\dif{u} + \frac{1}{\pi}\sum_{n=2}^\infty \frac{1}{\sqrt{n}} \left|\Lambda_{\cL}(n) \widehat{h}\left(\frac{\log n}{2\pi}\right)\right|.
\end{flalign}
In what follows we are taking $\Delta=\frac{1}{\pi}\log\log \tau$.
Let us consider each term on the right-hand side of~\eqref{eq:sumWithGuinand} separately.

Using similar idea as in~\cite[Section~4]{ChirreSimonicHagen}, we have that the value of the first term on the right-hand side of~\eqref{eq:sumWithGuinand} does not exceed
\begin{flalign}
\label{eq:1stTerm}
&\frac{2m_{\cL}a\left|e^{2\pi a\Delta}+e^{-2\pi a\Delta}-\cos(2t\pi \Delta)\left(e^{\pi \Delta}+e^{-\pi \Delta}\right)-\ie\sin(2t\pi\Delta)\left(e^{\pi \Delta}-e^{-\pi \Delta}\right)\right|}{\left|a^2+(t-\ie/2)^2\right|\left(e^{\pi a\Delta}-e^{-\pi a\Delta}\right)^2} \nonumber \\
&\leq \frac{2m_{\cL}a}{ \left|a^2-1/4+t^2\right|}\left(\frac{2e^{(1-2a)\pi\Delta}+1+e^{-4\pi a\Delta}}{\left(1-e^{-2\pi a\Delta}\right)^2}\right) \leq \frac{m_{\cL}(2\sigma-1)}{t^2-3/4}\left(\frac{2(\log \tau)^{2-2\sigma}+1+(\log{\tau})^{-2(2\sigma-1)}}{\left(1-(\log{\tau})^{- (2\sigma-1)}\right)^2}\right).
\end{flalign}

Let us now consider the second and third terms. Using $\Re\left\{\left(\Gamma'/\Gamma\right)(z)\right\}\leq \log{|z|}$ for $\Re\{z\}\geq 1/4$, see~\cite[Lemma 2.3]{ChandeeExplBounds}, we obtain
\begin{flalign}
\label{eq:BoundOnLogDerGamma}
\Re\left\{\frac{\Gamma'}{\Gamma}\left(\frac{\lambda_j}{2}+\mu_j+\lambda_j(t-u)\ie\right)\right\}
&\leq \log{\left|\frac{\lambda_j}{2}+\mu_j+\lambda_j(t-u)\ie\right|} \nonumber \\
&\leq \log{\left(2\lambda_j\pi|t|\right)}+\log{\left(\frac{1}{4\pi|t|}+\frac{\max_{1\leq j\leq f}\left\{\left|\mu_j\right|\right\}}{2\min_{1\leq j\leq f}\left\{\lambda_j\right\}\pi |t|}+\frac{1}{2\pi}\left|1-\frac{u}{t}\right|\right)} \nonumber \\
&\leq \log{\left(2\lambda_j\pi|t|\right)}+\log{\left(1+\frac{|u|}{2\pi t_0}\right)}
\end{flalign}
since we are assuming the strong $\lambda$-conjecture. Therefore,
\begin{equation*}
\sum_{j=1}^{f}\lambda_{j}\Re\left\{\frac{\Gamma'}{\Gamma}\left(\frac{\lambda_j}{2}+\mu_j+\lambda_j(t-u)\ie\right)\right\}
\leq \frac{1}{2}\log{\tau}-\log Q+\frac{\sdeg}{2}\log{\left(1+\frac{|u|}{2\pi t_0}\right)}.
\end{equation*}
Furthermore, because $\pi a \Delta\ge \alpha_1$ and
\begin{flalign}
\label{eq:2ndAnd3rdTerm}
&\frac{\log Q}{\pi}\widehat{h}(0) +\frac{1}{\pi}\sum_{j=1}^f\int_{-\infty}^{\infty}h(u)\lambda_{j}\Re \left\{\frac{\Gamma'}{\Gamma}\left(\frac{\lambda_j}{2}+\mu_j+\lambda_j(t-u)\ie\right)\right\}\dif{u} \nonumber \\
&\leq  \frac{\log Q}{\pi}\widehat{h}(0)+\frac{\widehat{h}(0)}{\pi}\left(\frac{1}{2}\log{\tau}-\log{Q}\right) +\frac{\coth{(\pi a \Delta)^2}}{\pi}\int_{-\infty}^{\infty}\frac{a}{a^2+u^2}\left(\frac{\sdeg}{2}\log{\left(1+\frac{|u|}{2\pi t_0}\right)}\right)\dif{u} \nonumber \\
&\leq
\frac{\log{\tau}}{2}+\frac{(\log \tau)^{2-2\sigma}}{1-e^{-2\alpha_1}}+
\frac{\sdeg}{\pi}\coth^{2}{(\alpha_1)}
\int_0^{\infty} \frac{\log{\left(1+\frac{y}{2\pi t_0}\right)}}{1+y^2}\dif{y},
\end{flalign}
where we also used~\cite[Inequality~(4.6)]{ChirreSimonicHagen} and the fact that $0<a\leq 1$.

We are left with the last term. Since the function $\cL(s)$ has a polynomial Euler product of order $m$, we know that~\eqref{eq:BoundOnLambda} holds. Thus, by the proof of~\cite[Theorem~5]{ChirreSimonicHagen}, the fourth term does not exceed
\begin{multline}
\label{eq:4thTerm}
m\left(\frac{2\log\log \tau-\gamma-1+0.24(\log \tau)^{-1}}{\left(1-(\log \tau)^{-1}\right)^2}\right) \leq 2m\log\log{\tau} \\
+ m\left(-1-\gamma+\frac{1}{1-(\log{\tau_0})^{-1}}\left(\frac{4\log{\log{\tau_0}}}{\log{\tau_0}}+\frac{2.24}{\log{\tau_0}-1}\right)\right).
\end{multline}
Simplifying~\eqref{eq:1stTerm} by using~\eqref{eq:RangeForSigmaLemma}, and then adding~\eqref{eq:2ndAnd3rdTerm} and~\eqref{eq:4thTerm}, we finally arrive at~\eqref{eq:SumOverZerosGeneral}. The proof of Lemma~\ref{lem:SumOverZeros} is thus complete.
\end{proof}

\begin{remark}
\label{rem:LemmaOverZeros}
The only instance where we used the strong $\lambda$-conjecture in the proof of Lemma~\ref{lem:SumOverZeros} was in~\eqref{eq:BoundOnLogDerGamma}. Without this assumption we know that
\[
\lambda_j\Re\left\{\frac{\Gamma'}{\Gamma}\left(\frac{\lambda_j}{2}+\mu_j+\lambda_j(t-u)\ie\right)\right\} \leq \lambda_j\log{\left(2\lambda_j\pi|t|\right)}+\lambda_j \log{\left(1+|u|\right)} + O_{\lambda^{-}}(1)
\]
by Stirling's formula $\left(\Gamma'/\Gamma\right)(z)=\log{z}+O\left(1/|z|\right)$, where $|t|\geq\mu^{+}/\lambda^{-}$ is sufficiently large, and $\lambda^{-}$ and $\mu^{+}$ are from~\eqref{eq:lambdapm}. Therefore, GRH implies that
\[
\sum_{\gamma}\frac{\sigma-\frac{1}{2}}{\left(\sigma-\frac{1}{2}\right)^2+(t-\gamma)^2} \leq \frac{\log{\tau}}{2} + O\left(m_{\cL}(\log{\tau})^{2-2\sigma}+\sdeg+m_{\cL}\right) + 2m\log{\log{\tau}} + O_{\lambda^{-}}(f)
\]
uniformly for $1/2+1/\log{\log{\tau}}\leq\sigma\leq3/2$ and sufficiently large $\tau$ and $|t|\geq\mu^{+}/\lambda^{-}$. Furthermore, if we are considering the full Selberg class $\cS$, then additionally to the above the estimate~\eqref{eq:4thTerm} should be replaced by the first bound from Lemma~\ref{lem:GeneralBoundForGW} (see the second last line of p. 9 of \cite{ChirreSimonicHagen} to see how the sum appears). Therefore, under this condition GRH implies that
\begin{flalign*}
\sum_{\gamma}\frac{\sigma-\frac{1}{2}}{\left(\sigma-\frac{1}{2}\right)^2+(t-\gamma)^2} &\leq \frac{\log{\tau}}{2} +  O\left(m_{\cL}(\log{\tau})^{2-2\sigma}\right) \\
&+ O\left(A_2\left(\mathcal{C}_{\cL}^{R}(\varepsilon),\mathcal{C}_{\cL}^{R}(\varepsilon)+\mathcal{C}_{\cL}^{E},\varepsilon,\theta,\tau\right)\right) + O\left(\sdeg+m_{\cL}\right) + O_{\lambda^{-}}(f)
\end{flalign*}
uniformly for $1/2+1/\log{\log{\tau}}\leq\sigma\leq3/2$ and sufficiently large $\tau$ and $|t|\geq\mu^{+}/\lambda^{-}$. Taking into account also Conjecture~\ref{conj:SelbergVariant} or~\ref{conj:SelbergVariant2}, then $A_2\left(a,b,\varepsilon,\theta,\tau\right)$ in the above inequality should be replaced accordingly to Lemma~\ref{lem:GeneralBoundForGW}. 
\end{remark}

\begin{remark}
\label{eq:Zsum1}
We are going to describe how to modify the above proof when $t=0$ and $\cL$ is entire. Note that then $m_{\cL}=0$ and we do not need estimate \eqref{eq:1stTerm}.
Moreover, under the strong $\lambda$-conjecture, we have
\begin{equation*}
\sum_{j=1}^{f}\lambda_{j}\Re\left\{\frac{\Gamma'}{\Gamma}\left(\frac{\lambda_j}{2}+\mu_j-\lambda_ju\ie\right)\right\}
\leq \frac{1}{2}\log{\sq}-\log Q+\frac{\sdeg}{2}\log{\left(\frac{1}{4\pi}+\frac{\mu^+}{\pi}+\frac{|u|}{2\pi}\right)},
\end{equation*}
and without that condition
\begin{equation*}
\sum_{j=1}^{f}\lambda_{j}\Re\left\{\frac{\Gamma'}{\Gamma}\left(\frac{\lambda_j}{2}+\mu_j-\lambda_ju\ie\right)\right\}
\leq \frac{1}{2}\log{\sq}-\log Q+\frac{\sdeg}{2}\log{\left(\frac{1}{4\pi}+\frac{\mu^+}{2\lambda^-\pi}+\frac{|u|}{2\pi}\right)}+O_{\lambda^-}(1).
\end{equation*}
To estimate \eqref{eq:4thTerm}, we need to replace $\tau$ with $\sq$ and $\tau_0$ with ${\sq}_0$. Therefore, under the assumptions of Lemma \ref{lem:SumOverZeros}, where $\tau$ is replaced by $\sq$ and we do not have any conditions for $t$,
\begin{equation*}
\sum_{\gamma}\frac{\sigma-\frac{1}{2}}{\left(\sigma-\frac{1}{2}\right)^2+\gamma^2}   \leq \frac{\log{\sq}}{2}+\mathfrak{a}\left(0,\frac{\log\log{{\sq}_0}}{2},1\right)\left(\log{\sq}\right)^{2-2\sigma} + 2m\log{\log{\sq}} + \mathfrak{b}_3\left(\sdeg,m, {\sq}_0\right),
\end{equation*}
where $\mathfrak{a}\left(m_{\cL},\alpha_1,t_0\right)$ is defined by~\eqref{eq:Funa} and
\begin{multline}
\label{def:b3}
\mathfrak{b}_3\left(\sdeg,m,{\sq}_0\right) \de \frac{\sdeg}{\pi}\coth^{2}{\left(\frac{\log\log{{\sq}_0}}{2}\right)}\int_{0}^{\infty} \frac{\log{\left(\frac{1}{4\pi}+\frac{\mu^+}{\pi}+\frac{y}{2\pi}\right)}}{1+y^2}\dif{y} \\
+ m\left(-1-\gamma+\frac{\frac{4\log{\log{{\sq}_0}}}{\log{{\sq}_0}}+\frac{2.24}{\log{{\sq}_0}-1}}{1-(\log{{\sq}_0})^{-1}}\right).
\end{multline}
Without the strong $\lambda$-conjecture, for sufficiently large $\sq$ and $\sigma \geq 1$, we have
\[
\sum_{\gamma}\frac{\sigma-\frac{1}{2}}{\left(\sigma-\frac{1}{2}\right)^2+\gamma^2}  \leq \frac{\log{\sq}}{2} + 2m\log{\log{\sq}} + O_{\lambda^{-},\mu^+}(\sdeg)+O_{\lambda^{-}}(f),
\]
and for the full Selberg class we have
\begin{flalign*}
\sum_{\gamma}\frac{\sigma-\frac{1}{2}}{\left(\sigma-\frac{1}{2}\right)^2+\gamma^2}  &\leq \frac{\log{\sq}}{2}+ O\left(A_2\left(\mathcal{C}_{\cL}^{R}(\varepsilon),\mathcal{C}_{\cL}^{R}(\varepsilon)+\mathcal{C}_{\cL}^{E},\varepsilon,\theta,\sq\right)\right) + O_{\lambda^{-},\mu^+}(\sdeg)+ O_{\lambda^{-}}(f).
\end{flalign*}
Under Conjectures~\ref{conj:SelbergVariant} or~\ref{conj:SelbergVariant2}, the term $A_2\left(a,b,\varepsilon,\theta,\sq\right)$ in the above inequality should be replaced accordingly to Lemma~\ref{lem:GeneralBoundForGW}. 
\end{remark}

\section{Proofs of Theorems~\ref{thm:MainGeneral}--\ref{thm:1line}}
\label{sec:ProofFirst}

\begin{proof}[Proof of Theorem~\ref{thm:MainGeneral}]
Assume the notation and conditions from Theorem~\ref{thm:MainGeneral}. Observe that fixed $\alpha\geq\log{2}$ guarantees that $x$ and $y$ from~\eqref{eq:xy} are not less than $2$ for sufficiently large $\tau$. By~\eqref{eq:SumOverZeros} and Remark~\ref{rem:LemmaOverZeros} we have that
\begin{multline}
\label{eq:proofThm1V2}
\mathcal{Z}(\sigma) \leq \frac{e^{\alpha}+1}{2\alpha}\left(\log{\tau}\right)^{2-2\sigma} + O\left(m_{\cL}\left(\log{\tau}\right)^{3-4\sigma}\right) \\ 
+ O\left(\frac{A_2\left(\mathcal{C}_{\cL}^{R}(\varepsilon),\mathcal{C}_{\cL}^{R}(\varepsilon)+\mathcal{C}_{\cL}^{E},\varepsilon,\theta,\tau\right)+\sdeg+m_{\cL}}{\left(\log{\tau}\right)^{2\sigma-1}}\right) + O_{\lambda^{-}}\left(\frac{f}{\left(\log{\tau}\right)^{2\sigma-1}}\right),
\end{multline}
and by~\eqref{eq:sum2} and~\eqref{eq:pole} also that
\begin{equation}
\label{eq:proofThm1}
\mathcal{R}_1(\sigma) \ll_{\lambda^{+}} \frac{f}{\left(\log{\tau}\right)^{2\sigma}}, \quad \mathcal{R}_2(\sigma)\ll m_{\cL}\left(\frac{\sq}{\tau}\right)^{\frac{2}{\sdeg}}\left(\log{\tau}\right)^{2-2\sigma},
\end{equation}
all valid for $1/2+\delta\leq\sigma\leq 3/2$ with the implied constants being absolute. Also, by~\eqref{eq:proofThm1V2} and~\eqref{eq:proofThm1},
\begin{multline}
\label{eq:proofThm1V3}
\int_{\sigma}^{\frac{3}{2}}\left(\mathcal{Z}+\mathcal{R}_1+\mathcal{R}_2\right)(\sigma')\dif{\sigma'} \leq \left(\frac{e^{\alpha}+1}{4\alpha}\right)\frac{\left(\log{\tau}\right)^{2-2\sigma}}{\log{\log{\tau}}} + O\left(\frac{m_{\cL}\left(\log{\tau}\right)^{3-4\sigma}}{\log{\log{\tau}}}\right) + O_{\lambda^{+},\lambda^{-}}\left(\frac{\left(\log{\tau}\right)^{1-2\sigma}f}{\log{\log{\tau}}}\right) \\
+ O\left(\frac{A_2\left(\mathcal{C}_{\cL}^{R}(\varepsilon),\mathcal{C}_{\cL}^{R}(\varepsilon)+\mathcal{C}_{\cL}^{E},\varepsilon,\theta,\tau\right)+\sdeg+m_{\cL}}{\left(\log{\tau}\right)^{2\sigma-1}\log{\log{\tau}}}\right) + O\left(\frac{m_{\cL}\left(\frac{\sq}{\tau}\right)^{\frac{2}{\sdeg}}\left(\log{\tau}\right)^{2-2\sigma}}{\log{\log{\tau}}}\right)
\end{multline}
for $1/2+\delta\leq\sigma\leq3/2$. The proof of Theorem~\ref{thm:MainGeneral} now follows by employing~\eqref{eq:proofThm1V2},~\eqref{eq:proofThm1},~\eqref{eq:proofThm1V3} and Lemma~\ref{lem:GeneralBoundForS} in~\eqref{eq:LogDerEffective} and~\eqref{eq:mainForLog}.
\end{proof}

\begin{proof}[Proof of Theorem~\ref{thm:MainGeneralV2}]
The approach is the same as in the proof of Theorem~\ref{thm:MainGeneral}, just now we are using inequalities~\eqref{eq:proofThm1V2} and~\eqref{eq:proofThm1V3} with $A_2\left(a,b,\varepsilon,\theta,\tau\right)$ that is in accordance with Lemma~\ref{lem:GeneralBoundForGW}, and also Lemma~\ref{lem:GeneralBoundForSV2} in place of Lemma~\ref{lem:GeneralBoundForS}.
\end{proof}

\begin{proof}[Proof of Theorem~\ref{thm:MainGeneral1line}]
Inequality~\eqref{eq:MainGeneral1lineV1} easily follows by employing ~\eqref{eq:proofThm1V2} with $A_2\left(a,b,\varepsilon,\theta,\tau\right)$ that is in accordance with Lemma~\ref{lem:GeneralBoundForGW} and~\eqref{eq:proofThm1} for $\sigma=1$, together with Lemma~\ref{lem:GeneralBoundFor1Line}, in~\eqref{eq:LogDerEffective}. Inequality~\eqref{eq:MainGeneral1lineV2} follows by employing~\eqref{eq:proofThm1V3} for $\sigma=1$ and with $A_2\left(a,b,\varepsilon,\theta,\tau\right)$ that is in accordance with Lemma~\ref{lem:GeneralBoundForGW}, together with Lemma~\ref{lem:GeneralBoundForGW} and~\eqref{eq:GeneralBoundForSV2_3} in~\eqref{eq:mainForLog}. 
\end{proof}

\begin{proof}[Proof of Theorem~\ref{thm:MainNonExplicit}]
Assume the notation and conditions from Theorem~\ref{thm:MainNonExplicit}. Observe that fixed $\alpha\in[\log{2},2)$ guarantees that $x$ and $y$ from~\eqref{eq:xy} are not less than $2$ for sufficiently large $\tau$. In what follows we are assuming that $\tau$ and $|t|\geq\mu^{+}/\lambda^{-}$ are sufficiently large. We can see from Remark~\ref{rem:LemmaOverZeros} that
\begin{equation}
\label{eq:proofThm2}
\mathcal{Z}(\sigma) \leq
\frac{e^{\alpha}+1}{2\alpha}\left(\log{\tau}\right)^{2-2\sigma} + O\left(m_{\cL}\left(\log{\tau}\right)^{3-4\sigma}
+ \frac{m\log{\log{\tau}}}{\left(\log{\tau}\right)^{2\sigma-1}} + \frac{\sdeg+m_{\cL}}{\left(\log{\tau}\right)^{2\sigma-1}}\right)
+ O_{\lambda^{-}}\left(\frac{f}{\left(\log{\tau}\right)^{2\sigma-1}}\right)
\end{equation}
uniformly for $1/2+1/\log{\log{\tau}}\leq\sigma\leq3/2$, where $\mathcal{Z}(\sigma)$ is defined by~\eqref{eq:SumOverZeros}. Note that
\[
\mathcal{Z}(\sigma) = O\left(1+\frac{m\log{\log{\tau}}}{\log{\tau}}+\frac{\sdeg+m_{\cL}}{\log{\tau}}\right) + O_{\lambda^{-}}\left(\frac{f}{\log{\tau}}\right)
\]
uniformly for $|1-\sigma|\leq 1/\log{\log{\tau}}$. We also have
\begin{equation}
\label{eq:R1R2}
\mathcal{R}_1(\sigma) \ll_{\lambda^{+}} \frac{f}{\left(\log{\tau}\right)^{(2-\alpha)\sigma}}, \quad
\mathcal{R}_2(\sigma) \ll m_{\cL}\left(\frac{\sq}{\tau}\right)^{\frac{2}{\sdeg}}\left(\log{\tau}\right)^{2-2\sigma}
\end{equation}
uniformly for $1/2+1/\log{\log{\tau}}\leq\sigma\leq3/2$, where $\mathcal{R}_1(\sigma)$ and $\mathcal{R}_2(\sigma)$ are defined by~\eqref{eq:sum2} and~\eqref{eq:pole}, respectively. In addition,
\begin{equation}
\label{eq:R1R2Near1}
\mathcal{R}_1(\sigma) \ll_{\lambda^{+}} \frac{f}{\left(\log{\tau}\right)^{2}}, \quad
\mathcal{R}_2(\sigma) \ll m_{\cL}\left(\frac{\sq}{\tau}\right)^{\frac{2}{\sdeg}}
\end{equation}
uniformly for $|1-\sigma|\leq 1/\log{\log{\tau}}$. Note that
\[
\frac{\left(\log{\log{\tau}}\right)^{2}}{(2\sigma-1)\left(\log{\tau}\right)^{2\sigma-1}} \ll \frac{1}{(2\sigma-1)^{3}}
\]
uniformly for $\sigma>1/2$. By~\eqref{eq:BoundForS} we thus have
\[
S_{\cL,x,y}(\sigma) \leq A(m,\alpha,\sigma,\sigma)\left(\log{\tau}\right)^{2-2\sigma}-\dfrac{m\sigma2^{1-\sigma}}{1-\sigma} +
O\left(\frac{m}{(2\sigma-1)^3}\right)
\]
uniformly for $1/2+1/\log{\log{\tau}}\leq\sigma\leq1-1/\log{\log{\tau}}$, where $A(a,\alpha,u,\sigma)$ is defined by~\eqref{eq:A}. In addition, by~\eqref{eq:BoundForSAround1} we have
\[
S_{\cL,x,y}(\sigma) \leq 2m\log{\log{\tau}} + O\left(m+m|1-\sigma|\left(\log{\log{\tau}}\right)^{2}\right)
\]
uniformly for $|1-\sigma|\leq 1/\log{\log{\tau}}$. The proof of inequality~\eqref{eq:MainNonExplicit} is complete after employing the above estimates on all terms from~\eqref{eq:LogDerEffective}.

We are using~\eqref{eq:mainForLog} to prove the second part of Theorem~\ref{thm:MainNonExplicit}. It follows by~\eqref{eq:proofThm2} and~\eqref{eq:R1R2} that
\begin{multline*}
\int_{\sigma}^{\frac{3}{2}}\mathcal{Z}\left(\sigma'\right)\dif{\sigma'} \leq \left(\frac{e^{\alpha}+1}{4\alpha}\right)\frac{\left(\log{\tau}\right)^{2-2\sigma}}{\log{\log{\tau}}} \\
+ O\left(\frac{m_{\cL}\left(\log{\tau}\right)^{3-4\sigma}}{\log{\log{\tau}}}+\frac{m}{\left(\log{\tau}\right)^{2\sigma-1}}
+\frac{\sdeg+m_{\cL}}{\left(\log{\tau}\right)^{2\sigma-1}\log{\log{\tau}}}\right)
+ O_{\lambda^{-}}\left(\frac{f}{\left(\log{\tau}\right)^{2\sigma-1}\log{\log{\tau}}}\right)
\end{multline*}
and
\begin{gather}
\label{eq:IntegralR1SP}
\int_{\sigma}^{\frac{3}{2}}\mathcal{R}_{1}\left(\sigma'\right)\dif{\sigma'} \ll_{\lambda^{+}}
\frac{f}{(2-\alpha)\left(\log{\tau}\right)^{(2-\alpha)\sigma}\log{\log{\tau}}}, \\
\int_{\sigma}^{\frac{3}{2}}\mathcal{R}_{2}\left(\sigma'\right)\dif{\sigma'} \ll
m_{\cL}\left(\frac{\sq}{\tau}\right)^{\frac{2}{\sdeg}}\frac{\left(\log{\tau}\right)^{2-2\sigma}}{\log{\log{\tau}}} \nonumber
\end{gather}
uniformly for $1/2+1/\log{\log{\tau}}\leq\sigma\leq3/2$. Because $\mathcal{R}_1(\sigma)\ll_{\lambda^{+}}f\cdot\left(\log{\tau}\right)^{-2\sigma}$ uniformly for $1-1/\log{\log{\tau}}\leq\sigma\leq 3/2$, we have by the above estimates that
\[
\int_{\sigma}^{\frac{3}{2}}\left(\mathcal{Z}+\mathcal{R}_1+\mathcal{R}_2\right)(\sigma')\dif{\sigma'} =
O\left(m+\frac{\sdeg+m_{\cL}}{\log{\tau}\log{\log{\tau}}}
+\left(\frac{\sq}{\tau}\right)^{\frac{2}{\sdeg}}\frac{m_{\cL}}{\log{\log{\tau}}}\right) + O_{\lambda^{+},\lambda^{-}}\left(\frac{f}{\log{\tau}\log{\log{\tau}}}\right)
\]
uniformly for $|1-\sigma|\leq 1/\log{\log{\tau}}$. By~\eqref{eq:SHat1} and~\eqref{eq:Ex} we have
\begin{multline*}
\widehat{S}_{\cL,x,y}(\sigma) \leq \eta(\alpha,\sigma,\tau)A\left(m,\alpha,\sigma,\sigma\right)
\frac{\left(\log{\tau}\right)^{2-2\sigma}}{\log{\log{\tau}}} - \frac{m\eta(\alpha,\sigma,\tau)}{(1-\sigma)\log{\log{\tau}}} \\
+ m\log{\left(2\log{\log{\tau}}\right)}
+ O\left(\frac{m\left(\log{\tau}\right)^{2-2\sigma}}{(1-\sigma)^{2}\left(\log{\log{\tau}}\right)^2}+\frac{m}{(2\sigma-1)^2}\right)
\end{multline*}
and
\[
\mathcal{E}_x \ll \frac{m}{\left(\log{\tau}\right)^{1-\frac{\alpha}{2}}\log{\log{\tau}}}
\]
uniformly for $1/2+1/\log{\log{\tau}}\leq\sigma\leq1-1/\log{\log{\tau}}$, where $\eta(\alpha,\sigma,\tau)$ is defined by~\eqref{eq:eta}. In addition, by~\eqref{eq:SHat2} and~\eqref{eq:Ex} we have
\[
\widehat{S}_{\cL,x,y}(\sigma) \leq m\log{\left(2\log{\log{\tau}}\right)} + O(m)
\]
and $\mathcal{E}_x \ll m$ uniformly for $|1-\sigma|\leq 1/\log{\log{\tau}}$. The proof of inequality~\eqref{eq:MainNonExplicit2} is complete after employing the above estimates on all terms from~\eqref{eq:mainForLog}.
\end{proof}

\begin{proof}[Proof of Theorem~\ref{thm:1line}]
Assume the notation and conditions from Theorem~\ref{thm:1line}. Let $\sigma=1$, and let $\tau$ and $|t|\geq\mu^{+}/\lambda^{-}$ be sufficiently large. By~\eqref{eq:proofThm2} we have
\[
\mathcal{Z}(1) \leq \frac{e^{\alpha}+1}{2\alpha} + O\left(\frac{m\log{\log{\tau}}}{\log{\tau}}+\frac{\sdeg+m_{\cL}}{\log{\tau}}\right) + O_{\lambda^{-}}\left(\frac{f}{\log{\tau}}\right)
\]
and by~\eqref{eq:BoundForS1} we also have
\[
S_{\cL,x,y}(1) \leq 2m\log{\log{\tau}} - m\left(\gamma + \alpha\right) + O\left(\frac{m\left(\log{\log{\tau}}\right)^{2}}{\log{\tau}}\right).
\]
Then~\eqref{eq:LogDerOn1} follows by taking the last two bounds and~\eqref{eq:R1R2Near1} in~\eqref{eq:LogDerEffective}. To prove~\eqref{eq:LogOn1}, note that
\begin{flalign*}
\int_{1}^{\frac{3}{2}}\left(\mathcal{Z}+\mathcal{R}_1+\mathcal{R}_2\right)(\sigma')\dif{\sigma'} &\leq \frac{e^{\alpha}+1}{4\alpha\log{\log{\tau}}} + O\left(\frac{m}{\log{\tau}}+\frac{\sdeg+m_{\cL}}{\log{\tau}\log{\log{\tau}}}\right) + O_{\lambda^{-}}\left(\frac{f}{\log{\tau}\log{\log{\tau}}}\right) \\ 
&+ O_{\lambda^{+}}\left(\frac{f}{\left(\log{\tau}\right)^{2}\log{\log{\tau}}}\right) + O\left(\frac{m_{\cL}}{\log{\log{\tau}}}\left(\frac{\sq}{\tau}\right)^{\frac{2}{\sdeg}}\right),
\end{flalign*}
see the proof of Theorem~\ref{thm:MainNonExplicit} for details. Also, $\mathcal{E}_{x}\ll m/\left(\log{\tau}\log{\log{\tau}}\right)$ by~\eqref{eq:Ex}, and 
\[
\widehat{S}_{\cL,x,y}(1) \leq m\log{\left(2\log{\log{\tau}}\right)} + m\gamma + O\left(\frac{m\log{\log{\tau}}}{\log{\tau}}\right)
\]
by~\eqref{eq:SHatOn1}. Then~\eqref{eq:LogOn1} follows by taking the last three inequalities in~\eqref{eq:mainForLog}.
\end{proof}

\section{Proofs of Theorems~\ref{thm:MainExplicit} and~\ref{thm:1LineExplicit}, and Remarks~\ref{thm:1Conjectures},~\ref{thm:1EulerProduct} and~\ref{rmk:sigma1Explicit}}
\label{sec:ProofSecond}

\begin{proof}[Proof of Theorem~\ref{thm:MainExplicit}]
The proof of both estimates from~\eqref{eq:3rdCase} is relatively simple and holds unconditionally since $\cL\in\cSP$ and thus
\begin{gather*}
\left|\frac{\cL'}{\cL}(s)\right| \leq m\sum_{p}\sum_{k=1}^{\infty}\frac{\log{p}}{p^{k\sigma_0}} = -m\frac{\zeta'}{\zeta}(\sigma_0) \leq \frac{m}{\sigma_0-1}, \\
\left|\log{\cL(s)}\right| \leq m\sum_{p}\sum_{k=1}^{\infty}\frac{1}{kp^{k\sigma_0}} = m\log{\zeta(\sigma_0)} \leq m\log{\frac{1}{\sigma_0-1}} + m\gamma(\sigma_0-1)
\end{gather*}
for $\sigma\geq\sigma_0>1$. The second inequality in the former expression follows from~\cite{Delange} or~\cite[Exercise 11.1.1(3)]{MontgomeryVaughan}, and the second inequality in the latter expression follows from~\cite[Lemma 1]{BastienRogalski2002} or~\cite[Lemma 5.4]{Ramare}. When we set $\sigma_0=1+\alpha_3/\log\log\tau$, we have proved the case \eqref{eq:3rdCase}.

Next, we are going to prove~\eqref{eq:1stCase}, i.e., upper bounds for the terms $\left|\cL'(s)/\cL(s)\right|$ and $\left|\log{\cL(s)}\right|$ under assumption (1) of Theorem~\ref{thm:MainExplicit}. Assume these conditions. By \eqref{eq:LogDerEffective}, it is sufficient to estimate the terms $S_{\cL,x,y}(\sigma)$, $\mathcal{Z}(\sigma)$, $\mathcal{R}_1(\sigma)$ and $\mathcal{R}_2(\sigma)$ in order to estimate the term $\left|\cL'(s)/\cL(s)\right|$. Now $S_{\cL,x,y}(\sigma)$ can be bounded by~\eqref{eq:BoundForS} and Lemma~\ref{lem:SumOverZeros} can be used in order to estimate $\mathcal{Z}(\sigma)$. To bound $\mathcal{R}_1(\sigma)$ from~\eqref{eq:sum2}, note that
\begin{equation*}
\exp{\left(\frac{2\alpha\sigma}{2\sigma-1}\right)} \leq \left(\log{\tau}\right)^{\frac{\alpha}{\alpha_1}\sigma}
\end{equation*}
for $\sigma\geq1/2+\alpha_1/\log{\log{\tau}}$, and
\begin{equation}
\label{eq:SumLambda}
\sum_{j=1}^{f}\lambda_j^2\left(\frac{\Gamma'}{\Gamma}\right)'\left(\frac{\lambda_j}{2}\right) = \frac{f}{4}\left(\frac{\Gamma'}{\Gamma}\right)'\left(\frac{1}{4}\right) \leq 4.3\sdeg
\end{equation}
by the strong $\lambda$-conjecture. Hence,
\begin{equation}
\label{eq:ExpEst}
    \mathcal{R}_1(\sigma)\leq \frac{ 4.3\sdeg \left(\sigma-\frac{1}{2}\right)}{\alpha \left(\log \tau\right)^{2\sigma}}\left(1+\left(\log{\tau}\right)^{\frac{\alpha}{\alpha_1}\sigma}\right).
\end{equation}
We estimate $\mathcal{R}_2(\sigma)$ from~\eqref{eq:pole} trivially by using $\sigma<1$ leading to the estimate
\begin{equation}
\label{eq:estR2}
   \mathcal{R}_2(\sigma) < \frac{m_{\cL}}{\alpha}\left(\frac{\sq}{\tau}\right)^{\frac{2}{\sdeg}}\left(\log{\tau}\right)^{2(1-\sigma)}.
\end{equation}
Collecting all bounds for all terms, we get estimate~\eqref{eq:1stCase} for the term $\left|\cL'(s)/\cL(s)\right|$. Now we are using~\eqref{eq:mainForLog} in order to prove~\eqref{eq:1stCaseB}, i.e., estimate for $\left|\log{\cL(s)}\right|$. Hence, it is sufficient to estimate $\widehat{S}_{\cL,x,y}(\sigma)$, $\mathcal{E}_x$ and integrals of $\mathcal{Z}(\sigma)$, $\mathcal{R}_1(\sigma)$ and $\mathcal{R}_2(\sigma)$. We are estimating $\widehat{S}_{\cL,x,y}(\sigma)$ by~\eqref{eq:SHat1} and $\mathcal{E}_x$ by~\eqref{eq:Ex}. Note that Lemma~\ref{lem:SumOverZeros} implies
\begin{equation}
\label{eq:IntegralZ}
\int_{\sigma}^{\frac{3}{2}}\mathcal{Z}\left(\sigma'\right)\dif{\sigma'} \leq \frac{\left(e^{\alpha}+1\right)\left(\log{\tau}\right)^{2-2\sigma}}{4\alpha\log{\log{\tau}}}
+ \frac{\mathfrak{a}\left(e^{\alpha}+1\right)\left(\log{\tau}\right)^{3-4\sigma}}{4\alpha\log{\log{\tau}}}
+ \frac{m\left(e^{\alpha}+1\right)}{\alpha\left(\log{\tau}\right)^{2\sigma-1}}
+ \frac{\max\{0,\mathfrak{b}\}\left(e^{\alpha}+1\right)}{2\alpha\left(\log{\tau}\right)^{2\sigma-1}\log{\log{\tau}}},
\end{equation}
and that~\eqref{eq:ExpEst} and~\eqref{eq:estR2} guarantee
\begin{gather}
\int_{\sigma}^{\frac{3}{2}}\mathcal{R}_{1}\left(\sigma'\right)\dif{\sigma'} \leq
\frac{4.3\sdeg\alpha_1\left(1+\left(\log{\tau_0}\right)^{-\frac{\alpha}{2\alpha_1}}\right)}
{\alpha\left(2\alpha_1-\alpha\right)\left(\log{\tau}\right)^{\left(2-\frac{\alpha}{\alpha_1}\right)\sigma}\log{\log{\tau}}}, \label{eq:IntegralR1} \\
\int_{\sigma}^{\frac{3}{2}}\mathcal{R}_{2}\left(\sigma'\right)\dif{\sigma'} \leq
\frac{m_{\cL}\left(1+e^{\frac{\alpha}{2}}\right)}{2\alpha}\left(\frac{\sq}{\tau}\right)^{\frac{2}{\sdeg}}
\frac{\left(\log{\tau}\right)^{2-2\sigma}}{\log{\log{\tau}}}. \label{eq:IntegralR2}
\end{gather}
In the derivation of the last inequality we also used the fact that $(\sigma-1)/(2\sigma-1)$ is an increasing function. Collecting all bounds for all terms in~\eqref{eq:mainForLog} finally gives~\eqref{eq:1stCaseB}.

We are going to prove~\eqref{eq:2ndCase}. Assume the conditions from the case~(2) of Theorem~\ref{thm:MainExplicit}. Note that
\begin{equation}
\label{eq:boundary}
\frac{3}{4} < 1 - \frac{\alpha_2}{\log{\log{\tau_0}}} \leq \sigma \leq 1 + \frac{\alpha_3}{\log{\log{\tau_0}}} < \frac{3}{2}.
\end{equation}
Then $S_{\cL,x,y}(\sigma)$ can be bounded by~\eqref{eq:BoundForSAround1} and Lemma~\ref{lem:SumOverZeros} can be used for $\alpha_1=\frac{1}{2}\log{\log{\tau_0}}-\alpha_2$ in order to estimate $\mathcal{Z}(\sigma)$. We are using also
\[
\left(\log{\tau}\right)^{2-2\sigma} \leq 1 + 2\theta_2(M)|1-\sigma|\log{\log{\tau}},
\]
where $\theta_2(u)$ and $M$ are as in Theorem~\ref{thm:MainExplicit}. The functions $\mathcal{R}_1(\sigma)$ and $\mathcal{R}_2(\sigma)$ are bounded trivially by using~\eqref{eq:SumLambda} and~\eqref{eq:boundary}, while observing also that $\sigma/(2\sigma-1)$ is a decreasing function. Collecting all bounds for all terms in~\eqref{eq:LogDerEffective} gives~\eqref{eq:2ndCase}. We are using~\eqref{eq:mainForLog} in order to prove~\eqref{eq:2ndCaseB}. We are estimating $\widehat{S}_{\cL,x,y}(\sigma)$ by~\eqref{eq:SHat2} and $\mathcal{E}_x$ by~\eqref{eq:Ex}, where we are using also
\[
\frac{1}{2} \leq \eta(\alpha,\sigma,\tau) \leq \frac{1}{2}\left(1-\frac{\alpha}{\log{\log{\tau_0}}-2\alpha_2}\right)^{-1}.
\]
Furthermore, the estimate~\eqref{eq:IntegralZ} can be used with $\mathfrak{a}$ and $\mathfrak{b}$ replaced by $\mathfrak{a}_1$ and $\mathfrak{b}_1$, the latter pair of functions being as in Theorem~\ref{thm:MainExplicit}. Also, we are using~\eqref{eq:IntegralR2}, while~\eqref{eq:IntegralR1} is replaced by
\[
    \int_{\sigma}^{\frac{3}{2}}\mathcal{R}_1\left(\sigma'\right)\dif{\sigma'} \leq \frac{4.3\sdeg\left(1+\exp{\left(\frac{2\alpha\left(\log{\log{\tau_0}}-\alpha_2\right)}{\log{\log{\tau_0}}-2\alpha_2}\right)}\right)}{2\alpha\left(\log{\tau}\right)^{2\sigma}\log{\log{\tau}}}.
\]
Collecting all bounds for all terms in~\eqref{eq:mainForLog} finally gives~\eqref{eq:2ndCaseB}. The proof of Theorem~\ref{thm:MainExplicit} is thus complete.
\end{proof}

\begin{proof}[Proof of Theorem~\ref{thm:1LineExplicit}]
Note that the conditions of Theorem~\ref{thm:1LineExplicit} guarantee that $x\geq 60$ and $y\geq 2$. Firstly, we are going to prove~\eqref{eq:LogDer1LineExp}. We can use~\eqref{eq:BoundForS1} for the estimation of $S_{\cL,x,y}(1)$. Moreover, Lemma~\ref{lem:SumOverZeros} can be used for $\alpha_1=\frac{1}{2}\log{\log{\tau_0}}$ in order to bound $\mathcal{Z}(1)$ from~\eqref{eq:SumOverZeros}. Estimates for $\mathcal{R}_1(1)$ and $\mathcal{R}_2(1)$ from~\eqref{eq:sum2} and~\eqref{eq:pole}, respectively, can be obtained by trivial estimation while using also~\eqref{eq:SumLambda}. Taking all estimates in~\eqref{eq:LogDerEffective} furnishes the proof of~\eqref{eq:LogDer1LineExp}. Next, we need to prove also~\eqref{eq:Log1LineExp}. We are estimating $\widehat{S}_{\cL,x,y}(1)$ by~\eqref{eq:SHatOn1} and $\mathcal{E}_x$ by~\eqref{eq:Ex}. Furthermore, estimate~\eqref{eq:IntegralZ} can be used for $\sigma=1$ with $\mathfrak{a}$ and $\mathfrak{b}$ replaced by $\mathfrak{a}_2$ and $\mathfrak{b}_2$, the latter pair of functions being as in Theorem~\ref{thm:1LineExplicit}. For the remaining two terms we are using~\eqref{eq:IntegralR2} for $\sigma=1$, and
\begin{equation}
\label{eq:R1IntCase2}
\int_{1}^{\frac{3}{2}}\mathcal{R}_1\left(\sigma'\right)\dif{\sigma'} \leq \frac{4.3\sdeg\left(e^{2\alpha}+1\right)}{2\alpha\left(\log{\tau}\right)^{2}\log{\log{\tau}}}.
\end{equation}
Inequality~\eqref{eq:Log1LineExp} now follows by~\eqref{eq:mainForLog}. The proof of Theorem~\ref{thm:1LineExplicit} is thus complete.
\end{proof}

\begin{proof}[Proofs of Remarks~\ref{thm:1Conjectures},~\ref{thm:1EulerProduct} and~\ref{rmk:sigma1Explicit}]
In the cases of the logarithmic derivative, let us use estimate~\eqref{eq:oneFormulaLogDer} and Remarks~\ref{rmk:Sreal} and~\ref{eq:Zsum1}. For Remark~\ref{thm:1Conjectures}, we use estimates~\eqref{eq:GeneralBoundFor1Line1} and~\eqref{eq:proofThm1}. Note that in the proof of~\ref{eq:GeneralBoundFor1Line1} we have used $\alpha=1$ and hence we have to do the same choice for the whole theorem. In the case of~Remark \ref{thm:1EulerProduct}, we apply also estimates~\eqref{eq:BoundForS1} and~\eqref{eq:R1R2Near1}. Finally, using estimates~\eqref{eq:BoundForS1} and~\eqref{eq:SumLambda}, the bound for the logarithmic derivative in Remark~\ref{rmk:sigma1Explicit} follows.

Let us now consider logarithms. We use estimate~\eqref{eq:mainForLog} and Remarks~\ref{rmk:Sreal} and~\ref{eq:Zsum1} in every case. First, we prove the estimate for logarithm in Remark~\ref{thm:1Conjectures} applying estimates~\eqref{eq:GeneralBoundForSV2_3},~\eqref{eq:GeneralBoundFor1Line2} and~\eqref{eq:proofThm1} in addition to remembering that in the proof of~\eqref{eq:GeneralBoundFor1Line2} we have used $\alpha=1$. To prove Remark~\ref{thm:1EulerProduct}, we use estimates~\eqref{eq:Ex},~\eqref{eq:SHatOn1} and~\eqref{eq:IntegralR1SP}. Lastly, we can derive the estimate in Remark~\ref{rmk:sigma1Explicit} by using estimates~\eqref{eq:Ex},~\eqref{eq:SHatOn1} and~\eqref{eq:R1IntCase2}.
\end{proof}

\section{Proofs of Corollaries ~\ref{cor:ZetaExplicit} and~\ref{corol:FixedRegions}} 
\label{sec:Corollaries}
\label{sec:simpl}

The next lemma shows that $\alpha\approx1.278$ is an optimal value for $\widehat{A}(m,\alpha,\sigma)$ in the sense that it minimizes this function \emph{for all} $\sigma\in(1/2,1)$. For sufficiently large $\tau$ this is also true for $\widetilde{A}(m,\alpha,\sigma,\tau)$ since then $\widetilde{A}(m,\alpha,\sigma,\tau)\approx\frac{1}{2}\widehat{A}(m,\alpha,\sigma)$. However, note that for $\sigma$ on some smaller intervals one can obtain better estimates.

\begin{lemma}
Let $0<\alpha\leq\alpha_0$, where $\alpha_0\approx 1.278$ is the solution of $(1-\alpha_0)e^{\alpha_0}+1=0$. Take $\sigma\in(1/2,1)$ and $m\in\N$. Then $\widehat{A}\left(m,\alpha_0,\sigma\right)\leq\widehat{A}\left(m,\alpha,\sigma\right)$, where $\widehat{A}\left(m,\alpha,\sigma\right)$ is defined by~\eqref{eq:WidehatA}.   
\end{lemma}

\begin{proof}
By simple calculation we have
\begin{multline*}
\frac{\partial{}}{\partial{\alpha}}\left(\widehat{A}\left(m,\alpha_0,\sigma\right)-\widehat{A}\left(m,\alpha,\sigma\right)\right) = 
\frac{1}{2\alpha^2(1-\sigma)^2}\biggl(\left((1-\alpha)e^{\alpha}+1\right)(1-\sigma)^2\biggr. \\
\left.+m(2\sigma-1)-m\left(2\sigma-1+2\alpha(1-\sigma)\right)\exp{\left(-\frac{2\alpha(1-\sigma)}{2\sigma-1}\right)}\right).
\end{multline*}
By the definition of $\alpha_0$ we clearly have $\left((1-\alpha)e^{\alpha}+1\right)(1-\sigma)^2\geq 0$. Because $e^{-x}\leq 1/(1+x)$ for $x\geq 0$, we also have 
\[
\exp{\left(-\frac{2\alpha(1-\sigma)}{2\sigma-1}\right)} \leq \frac{2\sigma-1}{2\sigma-1+2\alpha(1-\sigma)}.
\]
This implies
\[
m(2\sigma-1)-m\left(2\sigma-1+2\alpha(1-\sigma)\right)\exp{\left(-\frac{2\alpha(1-\sigma)}{2\sigma-1}\right)} \geq 0.
\]
Therefore, the function $\widehat{A}\left(m,\alpha_0,\sigma\right)-\widehat{A}\left(m,\alpha,\sigma\right)$ is increasing in $\alpha$, and the stated inequality now follows since this function is identically zero for $\alpha=\alpha_0$.
\end{proof}

\begin{proof}[Proof of Corollary~\ref{cor:ZetaExplicit}]
We are using the first part of Theorem~\ref{thm:MainExplicit} for $\sdeg=\sq=m=\alpha_1=1$ and $\tau_0=t_0=10^6$, and thus $\tau=|t|$, in order to prove both~\eqref{eq:LogDerZetaExplicit} and~\eqref{eq:LogZetaExplicit}. We additionally set $\alpha=1.278$. Note that the conditions from the case~(1) of Theorem~\ref{thm:MainExplicit} are satisfied. Then~\eqref{eq:LogDerZetaExplicit} easily follows by observing that $\left(2-\alpha/\alpha_1\right)\sigma>1/3$, and also that the sum of the last term and the third last term on the right-hand side of~\eqref{eq:1stCase} is negative. 

To prove~\eqref{eq:LogZetaExplicit}, note first that $\eta=\eta(\alpha,\sigma,\tau)\in[1/2,1.4]$ for $\sigma$ from the range~\eqref{eq:sigmaRange1}. Then the sum of the fourth term and the seventh term on the right-hand side of~\eqref{eq:1stCaseB} is not greater than
\begin{multline*}
\left(\left(1.4\nu_{2}\right)^{2} + \frac{(1-\sigma)^{2}\left(\log{\log{t}}\right)^{2}}{\nu_{1}^{2}\left(\log{t}\right)^{2\left(1-1/\nu_{2}\right)(1-\sigma)}}\right)\frac{\left(\log{t}\right)^{2-2\sigma}}{(1-\sigma)^{2}\left(\log{\log{t}}\right)^2} \\ 
\leq \left(\left(1.4\nu_{2}\right)^{2} + \frac{1}{e^{2}\nu_{1}^{2}\left(1-\frac{1}{\nu_{2}}\right)^{2}}\right)\frac{\left(\log{t}\right)^{2-2\sigma}}{(1-\sigma)^{2}\left(\log{\log{t}}\right)^2},
\end{multline*}
if $\nu_2<2$. Here we used the fact $e^x \geq ex$ for all real numbers $x$.
The parameters $\nu_1$ and $\nu_2$ allow some optimisation on various terms. Observe that for $\sigma=1-1/\log{\log{t}}$ the right-hand side of the above inequality is constant. Therefore, we choose $\nu_1=3.378$ and $\nu_2=1.182$ in order to minimize 
\[
1.4^2\nu_{2}^{2}e^2 + \frac{1}{\nu_{1}^{2}\left(1-\frac{1}{\nu_{2}}\right)^{2}} + \int_{0}^{\nu_1}\theta_1(u)\dif{u}
\]
that comes when we combine the estimate for the the sum of the fourth term and the seventh term on the right-hand side of~\eqref{eq:1stCaseB} together with the integral of $\theta_1$. The fourth term and the sixth term on the right-hand side of~\eqref{eq:LogZetaExplicit} now easily follow. The seventh term and the eight term follow after merging the ninth, tenth and fifteenth term, and the twelfth and fourteenth term on the right-hand side of~\eqref{eq:1stCaseB}, respectively. Note that $2-\alpha>2/3$. The remaining terms in~\eqref{eq:LogZetaExplicit} are easy to deduce similarly.
\end{proof}

\begin{proof}[Proof of Corollary~\ref{corol:FixedRegions}]
We are using estimates~\eqref{eq:1stCase} and~\eqref{eq:1stCaseB} from Theorem~\ref{thm:MainExplicit} in order to prove~\eqref{eq:VerySimplifiedLogDer4} and~\eqref{eq:Very2ndCaseSimpler}. Take $\alpha=1.278$, $\alpha_1=1.3$ and $\tau_0=\exp{\left(\exp{(13)}\right)}$. The conditions from~(1) of Theorem~\ref{thm:MainExplicit} are satisfied with our choice of parameters. We can see that for $\sigma\in(1/2,1)$
\begin{equation}
\label{eq:UpperForA}
A\left(\sdeg,\alpha,\sigma,\sigma\right) - \frac{\sdeg}{1-\sigma} = \frac{\sdeg}{1-\sigma}\left(-1+\frac{2\sigma-1}{2\alpha(1-\sigma)}\left(1-\exp{\left(-\frac{2\alpha(1-\sigma)}{2\sigma-1}\right)}\right)\right) \leq -1.278\sdeg,
\end{equation}
where $A$ is given as in~\eqref{eq:A}. Also, 
\[
\left(\frac{e^{\alpha}+1}{2\alpha} + \frac{m_{\cL}}{\alpha}\left(\frac{\sq}{\tau}\right)^{\frac{2}{\sdeg}}\right)\left(\log{\tau}\right)^{2-2\sigma} \leq 1.796\left(\log{\tau}\right)^{2-2\sigma}
\]
since $|t|\geq m_{\cL}+2\cdot10^3$. Then the sum of the first two terms and the last term on the right-hand side of~\eqref{eq:1stCase} is not greater than the sum of the first two terms on the right-hand side of~\eqref{eq:VerySimplifiedLogDer4}. Next, we can easily see that $\mathfrak{a}\leq 1.1$ and $\mathfrak{b}<0$, where we again used the fact that $t_0\geq m_{\cL}+2\cdot10^3\geq 2\cdot10^3$ and that both $\mathfrak{a}=\mathfrak{a}\left(m_{\cL},\alpha_1,t_0\right)$ and $\mathfrak{b}=\mathfrak{b}\left(\sdeg,\sdeg,m_{\cL},\alpha_1,t_0,\tau_0\right)$ are decreasing functions in $t_0$. The estimate for the remaining terms now easily follows since each term is a decreasing function in $\tau$.

Estimate~\eqref{eq:Very2ndCaseSimpler} is derived in a similar way. Now we have
\begin{flalign}
\label{eq:UpperForAtilde}
&\widetilde{A}\left(\sdeg,\alpha,\sigma,\tau\right)\frac{\left(\log{\tau}\right)^{2-2\sigma}}{\log{\log{\tau}}} - \frac{\sdeg\eta(\alpha,\sigma,\tau)}{(1-\sigma)\log{\log{\tau}}} \nonumber \\ 
&\leq \sdeg\frac{\left(\log{\tau}\right)^{2-2\sigma}-1}{2(1-\sigma)\log{\log{\tau}}} 
+ \left(\frac{e^{\alpha}+1}{4\alpha}-0.639\sdeg\right)\frac{\left(\log{\tau}\right)^{2-2\sigma}}{\log{\log{\tau}}}
+ \frac{\sdeg \alpha(1-\alpha(1-\sigma))(\log{\tau})^{2-2\sigma}}{2(1-\sigma)\left((2\sigma-1)\log{\log{\tau}}-\alpha\right)\log{\log{\tau}}} \nonumber \\
&\leq \sdeg\frac{\left(\log{\tau}\right)^{2-2\sigma}-1}{2(1-\sigma)\log{\log{\tau}}} 
+ \left(\frac{e^{\alpha}+1}{4\alpha}-0.639\sdeg\right)\frac{\left(\log{\tau}\right)^{2-2\sigma}}{\log{\log{\tau}}}
+ \frac{1.23\sdeg\left(\log{\tau}\right)^{2-2\sigma}}{(1-\sigma)^{2}\left(\log{\log{\tau}}\right)^{2}},
\end{flalign}
where we have used~\eqref{eq:UpperForA} and $\eta(\alpha,\sigma,\tau)= 1/2 + \alpha\left(2(2\sigma-1)\log{\log{\tau}}-2\alpha\right)^{-1}$, see~\eqref{eq:eta}. Then the sum of the first two terms and the last term on the right-hand side of~\eqref{eq:1stCaseB} is not greater than the sum of the first two terms on the right-hand side of~\eqref{eq:Very2ndCaseSimpler} and the last term on the right-hand side of~\eqref{eq:UpperForAtilde}. Note that $1/2\leq\eta(\alpha,\sigma,\tau)\leq 1$. Similarly as in the proof of Corollary~\ref{cor:ZetaExplicit}, we optimize parameters $\nu_1$ and $\nu_2$, obtaining $\nu_1=3.049$ and $\nu_2=1.244$. With this we can show that the sum of the fourth, fifth and seventh term on the right-hand side of~\eqref{eq:1stCaseB} and the last term on the right-hand side of~\eqref{eq:UpperForAtilde} is not greater than the last term on the right-hand side of~\eqref{eq:Very2ndCaseSimpler}. The sum of all other terms from~\eqref{eq:1stCaseB} is easily seen to be smaller than $\sdeg\left(\log{\log{\log{\tau}}}+8\right)$. The proof of~\eqref{eq:Very2ndCaseSimpler} is thus complete.

For the proof of~\eqref{eq:VerySimplifiedLogDer42} and~\eqref{eq:VerySimplifiedLog42} we are using estimates~\eqref{eq:LogDer1LineExp} and~\eqref{eq:Log1LineExp} from Theorem~\ref{thm:1LineExplicit}. Take $\alpha=2.186$ for ~\eqref{eq:VerySimplifiedLogDer42} and $\tau_0=\exp{(\exp{(13)})}$ in both cases. The value for $\alpha$ was chosen in order to minimize the constant term in~\eqref{eq:LogDer1LineExp} for $\sdeg=m=1$. We can see that $\mathfrak{a}_2=\mathfrak{a}\left(m_{\cL},13/2,t_0\right)\leq 1.1$ and $\mathfrak{b}_2=\mathfrak{b}\left(\sdeg,\sdeg,m_{\cL},13/2,t_0,\tau_0\right)\leq -1.56$. Then~\eqref{eq:VerySimplifiedLogDer42} easily follows by similar reasoning as before. For the proof of~\eqref{eq:VerySimplifiedLog42} we are taking $\alpha=1.278$ since we are minimizing the third term on the right-hand side of~\eqref{eq:Log1LineExp}. The result now easily follows.
\end{proof}

\section{Discussion}
\label{sec:discussion}

\subsection{On the estimates for $\left|\log{\cL(s)}\right|$}
\label{sec:DiscussionLog}

It is clear that our bounds for $\left|\log{\cL(s)}\right|$ are weaker for $s$ close to the critical line, or weaker even for all $\sigma\in(1/2,1)$ in the case of the Riemann zeta-function, when compared with known results, see Section~\ref{sec:intro}. The reason for this is mainly because our approach relies upon the integration of $\left(\cL'/\cL\right)(s)$ via the Selberg moment formula, see Section~\ref{sec:SMF}. However, the advantage of this is that estimates are valid for the modulus of $\log{\cL(s)}$ rather than just $\pm\log{|\cL(s)|}$, e.g., our estimates cover also the argument of $\cL(s)$ right of the critical line. The method from~\cite{CarneiroChandee} is different and produces upper bounds for $\pm\log{|\zeta(s)|}$ that have currently the sharpest main terms. In a follow-up paper, this method will be generalized to functions in $\cS$ and $\cSP$. In addition to this, generalizing the results for $S_{n,\sigma}(t)$ from~\cite{Carneiro} could also be worthwhile to perform. Recently, explicit and conditional upper bounds for $\left|S_{0,1/2}(t)\right|$ and $\left|S_{1,1/2}(t)\right|$ appeared in~\cite{SimonicSonRH}.

Our estimates on the $1$-line simply follow from this general approach, and thus no special attention was given for possible improvements. In fact, more precise methods were developed in~\cite{Lamzouri} to obtain explicit estimates of the same shape as~\eqref{eq:Littlewood1LineUpper} and~\eqref{eq:Littlewood1Line} for $\zeta(1+\ie t)$ and $L(1,\chi)$, and in~\cite{Lum18} their method was extended to entire $L$-functions. It might be interesting to extend the results from~\cite{Lum18} to the Selberg class of functions.

\subsection{On Corollary~\ref{cor:thm2} and Theorem~\ref{thm:1line}}
\label{sec:discussion2}

Note that even if $\cL\in\cSP$, one should expect that additional information on the distribution of $|a(p)|$, as discussed in Section~\ref{sec:Distr}, would improve the main terms in Corollary~\ref{cor:thm2} and Theorem~\ref{thm:1line}. We are giving several examples that illustrate this claim. 

Consider the Dirichlet character $\chi(n)$ modulo $5$, generated by $\chi(2)=i$. Note that this character is primitive and not quadratic. Let $\cL(s)=\zeta(s)L(s,\chi)$. Then 
\[
a(p) = \left\{
        \begin{array}{ll}
        2, & p\equiv 1\, (5), \\
        1+\ie, & p\equiv 2\, (5), \\
        1-\ie, & p\equiv 3\, (5), \\
        0, & p\equiv 4\, (5).
        \end{array}
        \right.
\]
Therefore,
\begin{equation*}
\sum_{p\leq x}|a(p)| = 2\pi(x;5,1)+\sqrt{2}\left(\pi(x;5,2)+\pi(x;5,3)\right) = \left(\frac{1+\sqrt{2}}{2}\right)\frac{x}{\log{x}} + O\left(\frac{x}{\log^{2}{x}}\right),
\end{equation*}
uniformly for all $x\geq 2$, where $\pi(x;q,a)$ counts prime numbers that not exceed $x$ and that have a residue $a$ modulo $q$, see~\cite[Corollary~11.20]{MontgomeryVaughan}. This shows that for our $\cL$, Conjecture~\ref{conj:SelbergVariant2} holds with $\widehat{\mathcal{C}_{\cL}^{P_1}}(x)\equiv\kappa=\frac{1}{2}\left(1+\sqrt{2}\right)$. On the other hand, $\cL$ has a polynomial Euler product representation with $m=2$. Therefore, the main terms given by Theorem~\ref{thm:MainGeneralV2} (resp.~ Theorem~\ref{thm:MainGeneral1line}) are better than those from Corollary~\ref{cor:thm2} (resp.~Theorem~\ref{thm:1line}) since $\kappa<m$.

In the case of Dedekind zeta-functions $\zeta_{\K}(s)$, their Dirichlet coefficients $a(n)$ count the number of integral ideals with norm equal to $n$. By the prime ideal theorem we thus have 
\begin{equation}
\label{eq:Dedekind}
\sum_{p\leq x}|a(p)| \leq \frac{x}{\log{x}} + O_{\K}\left(\frac{x}{\log^{2}{x}}\right),
\end{equation}
implying that Conjecture~\ref{conj:SelbergVariant2} holds with $\widehat{\mathcal{C}_{\cL}^{P_1}}(x)\equiv\kappa=1$. However, $m=n_{\K}$. This shows that estimates for $\zeta_{\K}(s)$ from Corollary~\ref{cor:thm2} can be improved via Theorem~\ref{thm:MainGeneralV2}. Under GRH for $\zeta_{\K}(s)$, we know that the implied constant in~\eqref{eq:Dedekind} can be uniformly bounded by $\ll 1+\log{\left|\Delta_{\K}\right|}$ for all $x\geq 2$, see~\cite[Corollary~1.4]{GrenieMolteni16} for example. Therefore, Theorem~\ref{thm:MainGeneral1line} gives
\begin{equation}
\begin{gathered}
\label{eq:Dedekind}
\left|\frac{\zeta_{\K}'}{\zeta_{\K}}(1+\ie t)\right| \leq 2\log{\log{\tau}} + O\left(\left(1+\log{\left|\Delta_{\K}\right|}\right)\log{\log{\log{\tau}}}\right), \\
\left|\log{\zeta_{\K}(1+\ie t)}\right| \leq \log{\log{\log{\tau}}} + O\left(1+\log{\left|\Delta_{\K}\right|}\right)
\end{gathered}
\end{equation}
for sufficiently large $\tau$. Note that these inequalities asymptotically improve what Theorem~\ref{thm:1line} gives for $\cL(s)=\zeta_{\K}(s)$.

The following example can be seen as a blend of the previous two examples. Let $q\geq 3$ be a prime number and let
\[
\cL(s) = \prod_{\substack{\chi\,\mathrm{primitive}\\ \mathrm{mod}\,q}} L(s,\chi) = \frac{1}{\zeta(s)}\zeta_{\mathbb{Q}(\zeta_q)}(s),
\]
where $\mathbb{Q}(\zeta_q)$ is the $q$-th cyclotomic field. Then $a(p)=\sum_{\chi}\chi(p)$, where the sum runs through all primitive characters $\chi$ modulo $q$. Note that, under GRH for $\cL$, we have 
\[
\pi(x;q,a) = \frac{x}{\varphi(q)\log{x}} + O\left(\left(\frac{1}{\varphi(q)}+\frac{\log^{3}{q}}{\sqrt{q}}\right)\frac{x}{\log^{2}{x}}\right),
\]
uniformly for $x\geq2$ and $q\leq x$. Trivially, if $2\leq x<q$, then $\pi(x;q,a)=1$ if $a\leq x$ is a prime number, and $\pi(x;q,a)=0$ otherwise. Therefore, in both cases we obtain
\[
\sum_{p\leq x}\left|a(p)\right| = (q-2)\pi(x;q,1) + \sum_{k=2}^{q-1}\pi(x;q,k) \leq \left(\frac{2(q-2)}{q-1}\right)\frac{x}{\log{x}} + O\left(\frac{\left(\sqrt{q}\log^{3}{q}\right)x}{\log^{2}{x}}\right),
\]
uniformly for all $x\geq 2$. This shows that for the above $\cL$, Conjecture~\ref{conj:SelbergVariant2} holds with $\widehat{\mathcal{C}_{\cL}^{P_1}}(x)\equiv 2(q-2)/(q-1)$ and $\widehat{\mathcal{C}_{\cL}^{P_2}}=O\left(\sqrt{q}\log^{3}{q}\right)$. Now,
\[
\left|b\left(p^{k}\right)\right| = \frac{1}{k}\left|\sum_{\substack{\chi\,\mathrm{primitive}\\ \mathrm{mod}\,q}}\chi\left(p^k\right)\right| \leq 
\left\{
        \begin{array}{ll}
        q, & p^k\equiv 1\, (q), \\
        1, & \mathrm{otherwise}.
        \end{array}
\right.
\]
Therefore, by choosing $\theta=1/2-1/\log{q}$, we can take $\mathcal{C}_{\cL}^{E}=q^{1-\theta}=e\sqrt{q}$. Note that $\sdeg=q-2$ and $\sq=q^{q-2}$. Therefore, results from Corollary~\ref{cor:thm2} for our $\cL$ can be improved via Theorem~\ref{thm:MainGeneralV2}, and Theorem~\ref{thm:MainGeneral1line} gives
\begin{gather*}
\left|\frac{\cL'}{\cL}(1+\ie t)\right| \leq 4\left(1-\frac{1}{q-1}\right)\log{\log{(q|t|)^{q-2}}} + O\left(\sqrt{q}\left(\log{q}\right)^{3}\log{\log{\log{(q|t|)^{q-2}}}}\right), \\
\left|\log{\cL(1+\ie t)}\right| \leq 2\left(1-\frac{1}{q-1}\right)\log{\log{\log{(q|t|)^{q-2}}}} + O\left(\sqrt{q}\log^{3}{q}\right)
\end{gather*}
for sufficiently large $|t|$ and absolute implied constants. This is not only an improvement over what Theorem~\ref{thm:1line} would give, but also what~\eqref{eq:Dedekind} trivially assures. 

Some other examples are also given by normalized $L$-functions attached to holomorphic newforms, see~\cite[p.~39]{SteudingBook}. For these functions $m=2$, while the prime mean-square conjecture with $\kappa=1$ is known to hold due to Rankin~\cite[Theorem~2]{Rankin73}.

\subsection{On the estimates for $s=1$} As we already mentioned in the introduction, our main purpose was to obtain bounds that are valid for $\tau$ and $t$ sufficiently large. However, we simultaneously derive the corresponding results also when $s=1$ and $\cL$ is entire, see Remarks~\ref{thm:1Conjectures}--\ref{rmk:sigma1Explicit}. Although the approach presented here does not always yield the best known results, we can still gain some improvements. For example, an immediate consequence of Remark~\ref{rmk:sigma1Explicit} is the following.

\begin{corollary}
\label{cor:DedekindRes}
Assume the Generalised Riemann Hypothesis for $\zeta_{\mathbb{K}}(s)$ and that $\left(\zeta_{\mathbb{K}}/\zeta\right)(s)$ is entire, i.e., the Dedekind conjecture is true. Let $\kappa_{\mathbb{K}}$ be the residue of $\zeta_{\mathbb{K}}(s)$ at $s=1$. Then 
\[
\left|\kappa_{\mathbb{K}}\right| \leq \left(\exp{\left(1.271+\frac{2.475}{\log{\log{\left|\Delta_{\mathbb{K}}\right|}}}\right)\log{\log{\left|\Delta_{\mathbb{K}}\right|}}}\right)^{n_{\mathbb{K}}-1}
\]
for $n_{\mathbb{K}}\geq 2$ and $\left|\Delta_{\mathbb{K}}\right|\geq 5.4\cdot10^6$.
\end{corollary}

This improves an upper bound on 
$\kappa_{\mathbb{K}}$ from~\cite[Corollary~2]{GarciaLee22} for $\left|\Delta_{\mathbb{K}}\right|\geq 5.4\cdot10^6$. Note that although their estimate is valid for all $\left|\Delta_{\mathbb{K}}\right|\geq 3$, it appears that the method cannot produce better constant than $2.45$ in the exponent (in place of our $1.271$) with just restricting the analysis on larger discriminants.



\providecommand{\bysame}{\leavevmode\hbox to3em{\hrulefill}\thinspace}
\providecommand{\MR}{\relax\ifhmode\unskip\space\fi MR }
\providecommand{\MRhref}[2]{%
  \href{http://www.ams.org/mathscinet-getitem?mr=#1}{#2}
}
\providecommand{\href}[2]{#2}

\end{document}